\documentclass[11pt]{article}
\def\NN{\mathbb{N}}

\def\RR{\mathbb{R}}

\def\CC{\mathbb{C}}

\def\cH{\mathcal{H}}

\def\cI{{\cal{I}}}
\def\cJ{{\cal{J}}}
\def\cU{{\cal{U}}}

\def\uu{\text{\bf u}}
\def\bu{\bar{u}}

\def\uu{\text{\bf u}}

\def\ww{\text{\bf w}}
\def\WW{\text{\bf W}}

\def\UU{\text{\bf U}}
\def\VV{\text{\bf V}}
\def\WW{\text{\bf W}}
\def\FF{\text{\bf F}}

\def\boxc{\square_c}
\def\boxca{\square_{c,\alpha}}

\def\supp{\text{\rm supp}}
\def\dist{\text{\rm dist}}

\def\red{\color{red}}

\usepackage[latin1]{inputenc}
\usepackage[USenglish,english,american]{babel}
\usepackage{amsfonts,setspace}
\usepackage{framed}

\usepackage{cancel}
\usepackage{amssymb}
\usepackage{dsfont}
\usepackage{anysize} 
\usepackage{multicol}
\usepackage[dvipsnames]{xcolor}
\usepackage{graphicx}
\usepackage{float}
\usepackage{bbm}
\usepackage{amsmath}
\usepackage{mathrsfs}
\usepackage{mathtools}
\usepackage{amsthm}
\usepackage{hyperref}
\usepackage{geometry}
\usepackage{comment}
\usepackage{tikz}
\usepackage{pgfplots}
\usetikzlibrary{decorations.pathmorphing,calc}
\usepackage{soul}
\hypersetup{
    colorlinks,
    citecolor=blue,
    filecolor=magenta,
    linkcolor=red,
    urlcolor=black
}

\theoremstyle{theorem}
\newtheorem{theorem}{Theorem}
\theoremstyle{corollary}
\newtheorem{corollary}{Corollary}
\theoremstyle{proposition}
\theoremstyle{lemma}
\newtheorem{lemma}{Lemma}
\theoremstyle{definition}
\newtheorem{definition}{Definition}
\theoremstyle{condition}
\theoremstyle{Assumption}
\theoremstyle{remark}
\newtheorem{remark}{Remark}
\theoremstyle{claim}

\theoremstyle{example}

\theoremstyle{obs}

\newenvironment{myproof}[1][\proofname]{\proof[#1]\mbox{}\\*}{\endproof}

\begin{document}


\begin{center}
\begin{Large}{\bf Inverse photoacoustic tomography problem in  media with fractional attenuation}\end{Large}\\
\vspace{1em}

{\bf Sebasti\'an Acosta\footnotemark[1] \& Benjam\'in Palacios\footnotemark[2]}
\end{center}
\vspace{-1em}

\footnotetext[1]{Department of Pediatrics, Baylor College of Medicine and Texas Children's Hospital, Houston, TX, USA. \\ sebastian.acosta@bcm.edu}
\footnotetext[2]{Department of Mathematics, Pontificia Universidad Cat\'olica de Chile, Santiago, Chile. \\ benjamin.palacios@uc.cl
}
\begin{abstract} 
\noindent{\sc Abstract.} 
We investigate the inverse problem of recovering an initial source for the wave equation with fractional attenuation, motivated by photoacoustic tomography (PAT). The attenuation is modeled by a Caputo fractional derivative of order $\alpha\in(0,1)$. We establish uniqueness under a geometric foliation condition via an adaptation of two types of Carleman estimates to the fractional setting, prove stability through continuity inequalities for fractional time-derivatives of wave solutions, and derive a reconstruction scheme based on a Neumann series. While our results apply directly to PAT, we expect that the analytic approach and tools employed might be of broader relevance to the analysis of PDEs with memory effects and to inverse problems for attenuated wave models.
\end{abstract}

\noindent{\it Key words:} inverse problems, wave equation, fractional derivative, uniqueness, stability, reconstruction.\\
\noindent{\it 2020 Mathematics Subject Classification}. Primary 35R30. Secondary 35L05, 35R11.

\tableofcontents


\section{Introduction}

For a given open and bounded domain $\Omega$ with a smooth boundary $\partial\Omega$ and a given relatively open subset $\Gamma\subset\partial\Omega$, this article focuses on the problem of determining an initial condition $u_0$ supported inside $\Omega$, from observations consisting of the boundary trace at $\Gamma$ of a wave field $u$ generated by $u_0$, and propagating across an acoustically attenuated media. More specifically, for smooth functions $c(x)\geq c_0>0$ and $a(x)\geq 0$, representing respectively the wave speed and internal damping, we analyze the invertibility ---in suitable function spaces--- of the observation map
\[
\Lambda_\alpha: u_0 \mapsto u|_{(0, T) \times \Gamma},
\]
with $u$ solution to the initial value problem (IVP)
\begin{equation} \label{frac_PAT}
\left\{
\begin{array}{ll}
\partial_t^2 u - c^2(x)\Delta u + a(x)\partial_t^\alpha u = 0,&(t,x)\in(0,\infty)\times\RR^n \\[.5em]
(u, \partial_t u)|_{t=0} = \left(u_0, -\textstyle\frac{a(x)}{\Gamma(2 - \alpha)}u_0\right),&x\in\RR^n.
\end{array}
\right.
\end{equation}
In the previous system, $\Gamma(\alpha)$ stands for the Gamma function, while $\partial^\alpha_t$ represents the Caputo fractional derivative of order $\alpha\in(0,1)$ defined as
\[
\partial_t^\alpha f(t) = \frac{1}{\Gamma(1 - \alpha)} \int_0^t \frac{f'(s)}{(t - s)^\alpha} \, ds
\]
for an absolutely continuous function $f$.

Our analysis aims to answer the three main questions related to the inverse problem: uniqueness, stability, and reconstruction; and for this purpose, we adapt an approach previously used by the authors to address the case of a wave equation with memory-type attenuation but involving a regular kernel. The main technical contributions stemming from this work include the refinement of a local unique continuation argument based on a foliation condition and two kinds of Carleman estimates; and specific Sobolev regularity inequalities concerning the Caputo fractional derivative of solutions to the unattenuated wave equation, which are obtained employing a novel formulation of the Caputo fractional derivative in the sense of distributions, the oscillatory integral representation of solutions to the wave equation, and the theory of Fourier Integral Operators (FIO's).


The motivation behind this work arises from the medical imaging modality of photoacoustic tomography (PAT). The standard mathematical formulation in PAT involves recovering a spatially varying initial source term \( u_0(x) \), which represents the ultrasound pressure distribution at time \( t = 0 \), from measurements of ultrasonic waves collected at the boundary of the region of interest. This initial pressure distribution is directly linked to the optical energy absorbed by internal tissues following illumination by a laser pulse. It serves as the input for the second stage of PAT, which seeks to infer optical properties of the medium. The recovery of \( u_0 \) constitutes what is commonly referred to as the \emph{inverse photoacoustic tomography problem}.

The generation of ultrasound waves in PAT is typically modeled by a singular source term of the form \( \delta'(t)u_0(x) \), where \( \delta' \) denotes the time derivative of the Dirac delta distribution. This representation is inspired by the assumption of instantaneous (or rapid) thermal deposition. It is also physically natural to assume that no acoustic signal exists prior to this, implying that the acoustic field \( u(x,t) \) vanishes for \( t < 0 \). Therefore, for a wave operator of the form
\[
P[u] = \partial_t^2 u - c^2(x)\Delta u + \mathcal{A}[u],
\]
where \( c(x) \) is the sound speed and \( \mathcal{A} \) is a linear attenuation operator, the governing model is the inhomogeneous equation \( P[u] = \delta'(t)u_0(x) \), together with the causality condition \( u(x,t) = 0 \) for \( t < 0 \).

It is well known that this system is equivalent to an initial value problem (IVP)
\[
\left\{
\begin{array}{l}
\partial_t^2 u - c^2(x)\Delta u + \mathcal{A}[u] = 0, \\
(u, \partial_t u)|_{t=0} = (u_0, u_1),
\end{array}
\right.
\]
where the initial velocity \( u_1 \) is determined by \( u_0 \) and the form of \( \mathcal{A} \). For example, in the unattenuated case (\( \mathcal{A} = 0 \)), one has \( u_1 = 0 \), while in the damped case, where \( \mathcal{A}[u] = a(x)\partial_t u \) with \( a(x) \geq 0 \), the initial velocity becomes \( u_1 = -a(x)u_0 \) (see, for instance, \cite{Homan}).

In this paper, we study the inverse PAT problem for an attenuation operator defined by a lower order Caputo (fractional) time derivative $\mathcal{A}[u] = a(x)\partial^\alpha_t u$, with $\alpha\in(0,1)$. 
The limiting cases, \( \alpha = 0 \) and \( \alpha = 1 \), corresponding to the unattenuated and (standard) damped wave equations, respectively, have been previously studied in the literature (see, e.g., \cite{StUh2009,StU2011,Homan,Palacios2016,HaN2019,Palacios2022}).

We show below that for the fractional attenuation model $\mathcal{A}[u] = a(x)\partial^\alpha_t u$, the appropriate initial velocity is given by
\[
u_1 = -\frac{a(x)}{\Gamma(2 - \alpha)} u_0(x),
\]
which provides a smooth interpolation between the damped and undamped cases.

For uniqueness, we treat the partial data problem under the assumption that the domain admits an \emph{admissible foliation} by strictly convex hypersurfaces. These sorts of conditions are crucial when one aims to obtain global results from local arguments. In our case, the adaptation of a local Carleman estimate methodology forces us to request rather technical conditions over the foliation (see Definition~\ref{def:adm_foliation}). The main technical result of this section is Theorem \ref{thm:uniqueness_att}, which states the unique continuation principle from where the uniqueness of the inverse problem follows, corresponding to Corollary \ref{cor:injectivity}.

The stability of the inverse problem is stated in Theorem \ref{thm:stability}. The strategy adopted in its proof demands us to analyze the continuous dependence of fractional time-derivatives of the solution to the wave equation with respect to the initial data ---this is in fact an interesting result by itself. See Lemma \ref{lemma:comp_FIO_Caputo}. The desired stability inequality then follows by a perturbation argument concerning the analogous estimate for the PAT problem in the absence of attenuation, and previously obtained in the literature.

Finally, for the reconstruction result, we derive a Neumann series expansion under the assumption of full boundary observations (i.e., data available on the entire boundary) and smallness of the damping coefficient. See Theorem \ref{thm:reconstruction}

\subsection{Previous work}

The inverse problem in photoacoustic tomography has been extensively studied over the past two decades, both from applied and theoretical perspectives. On the theoretical side, a systematic use of microlocal analysis techniques can be traced back to the influential works of Stefanov and Uhlmann \cite{StUh2009,StU2011}. In these papers, the authors provide a comprehensive treatment of the mathematical problem, addressing issues of uniqueness, stability, and reconstruction in both full and partial data settings for the conservative (i.e., unattenuated) wave equation, considering smooth and piecewise smooth wave speeds, respectively. An incomplete list of papers analyzing this inverse problem via microlocal tools is \cite{StUh2009,StU2011,Homan,StYa2015,CO2016,Palacios2016,StYa2017a,StYa2017b,AP2018,EHK2020,EK2021,Palacios2022}. 

Many foundational elements of this approach, however, originate in earlier works, including results on unique continuation and Carleman estimates, as well as insights on the propagation of singularities and time reversal methods. See, for instance, \cite{Finch-Patch-Rakesh-2004,HKN2008}. 


In the context of attenuated media, the adaptation of such techniques to the damped wave equation (i.e., the $\alpha=1$ case) was partially achieved in \cite{Homan}, and later extended by the second author in \cite{Palacios2016,Palacios2022}. The former of the last two articles focuses primarily on reconstruction, a topic that is also central to \cite{HaN2019}, where multiple reconstruction methods are explored; the second article, on the other hand, incorporates boundary conditions and focuses on partial observations. A different model but somehow related to the damped wave equation is introduced in \cite{AM2016}, where, in a collaborative project, the first author of this present article employs control-theoretic ideas to study the mathematical inverse photoacoustic tomography problem in the context of thermodynamic attenuation.

Currently, multiple alternatives have been reported in the literature to mathematically represent attenuation effects in biological tissue. This challenge of accurately modeling attenuation in soft tissues is highlighted in \cite{KS2011} where, in addition, an integral transform relating measurements from the attenuated and unattenuated problems is proposed in the case of homogeneous media (i.e., with constant coefficients). This framework allows for the analysis of various attenuation models, many involving fractional derivatives in time ---and in some cases, in space as well (see also \cite{Ammari2011,KaS2013,HaKN2017,SShi2017}). More recently, an alternative approach has been proposed in \cite{KaR2021} to obtain uniqueness in the inverse problem for homogeneous attenuated media involving time-fractional derivatives. 

For heterogeneous media with fractional attenuation, \cite{Ya2022} investigates the inverse problems of determining either a source term or an initial condition for the same equation considered in this work ---namely, the wave equation with Caputo fractional attenuation of order $\alpha \in (0,1)$. Stability and uniqueness results are established via global Carleman estimates in the case of constant sound speed. It is now well understood, however, that this approach can be extended to variable sound speeds under specific geometric conditions ---specifically, when the domain admits a foliation by concentric, strictly convex Euclidean spheres with respect to the sound speed metric. This can be regarded as a particular instance of our foliation condition (see Definition \ref{def:adm_foliation}). We also point out that system \eqref{frac_PAT} does not fit into the framework of \cite{Ya2022}, where it is assumed that one of the initial conditions vanishes with the purpose of allowing a symmetrization to negative times. We also point out that the analysis carried out in \cite{Ya2022} allows for the fractional exponent $\alpha$ to depend on the space variable. This assumption is not considered in this article (we assume a constant $\alpha$), and we leave the question of whether our techniques can be adapted to that setting for a future work.

This paper aims to extend some of the results obtained in \cite{Ya2022} and provide a more complete treatment of the inverse PAT problem in fractionally attenuated media. Our approach is an adaptation of the methodology introduced by the authors in \cite{AP2018}, which addresses the inverse photoacoustic tomography problem with attenuation modeled by an integro-differential operator involving a regular kernel. In particular, we introduce a more general (albeit more technical) foliation condition that ensures uniqueness for the inverse problem as a consequence of a local Carleman estimate approach. This is a more restrictive class of foliations than the one introduced in \cite{StUh2013} for the inverse problem of determining a sound speed. The reason for this lies in the fact that attenuation terms break the symmetry of the unattenuated wave equation; thus, standard unique continuation results do not hold, and have to be adapted. Additionally, by analyzing continuity inequalities of solutions to the wave equation ---particularly the dependence of time-fractional derivatives with respect to initial data--- we derive a stability estimate under sharper assumptions and rather general observation regions.

\subsection{Organization of the paper}

We begin with a section containing preliminary results that will be used throughout the paper. Specifically, we define a generalization of the Caputo derivative, as introduced in \cite{LiLiu}, and reformulate it as an oscillatory integral using the Fourier and inverse Fourier transforms. We also clarify the motivation behind the particular initial conditions adopted in this work, which arise from formulating the photoacoustic tomography problem as an inhomogeneous wave equation with an instantaneous source at time zero. This preliminary section concludes with results related to the forward problem, notably the energy-dissipative nature of system \eqref{frac_PAT}, which is established using the positivity of a singular kernel associated with the Caputo derivative. We then demonstrate the well-posedness of the forward problem in fractional Sobolev spaces.

The main part of the paper is divided into three sections, each addressing one of the core aspects of the inverse problem: uniqueness, stability, and reconstruction. Section 3, on uniqueness, presents this property as a corollary of a unique continuation theorem. The main ingredients behind its proof are stated in three separate lemmas, two of them involving different types of local Carleman estimates. Section 4 is dedicated to stability and provides a corresponding theorem derived from a continuity inequality involving the Caputo fractional derivative and the solution to the unattenuated wave equation. This result is of independent interest and uses a microlocal representation (parametrix) of solutions to the wave equation. Finally, Section 5 focuses on reconstruction, providing an explicit formula in the form of a Neumann series, which builds on the main results established in the previous sections.

For completeness, we include several appendices that provide key analytical tools used throughout the paper. These cover interpolation theory and fractional Sobolev spaces, Bochner integration, convexity of Riemannian hypersurfaces, and symbol classes and Fourier integral operators.

\section{Preliminaries}

\subsection{Notation}

\begin{itemize}
    \item We denote by $H(t)$ the Heaviside step function.
    \item Gradients and partial derivatives will be denoted by $\nabla f = (f_t,f_x)$, $f_t =f'= \partial_tf$ and $f_x=\partial_xf$; also, $D=\frac{1}{i}\partial_x$ and $D_t=\frac{1}{i}\partial_t$.
    \item We denote by $\langle\cdot,\cdot\rangle$ the duality pairing between distributions and test functions.
    \item For a given smooth sound speed $c(x)>c_0>0$, smooth damping coefficient $a(x)\geq 0$, and $\alpha\in(0,1]$, we denote the wave operator and the (fractionally) attenuated wave operator as
    \[
\square_c = \partial^2_t - c^2\Delta;\quad \square_{c,\alpha} = \square_c +a(x)\partial^\alpha_t.
\]
\item For a given Banach space $X$, we denote its norm by $\|\cdot\|_X$; occasionally, the $L^2$-norm will be denoted simply as $\|\cdot\|$; the norm $\|\cdot\|^2_{(m,\tau)}$ stands for the Sobolev norm with parameter $\tau>0$ given by 
\[
\|f\|^2_{(m,\tau)} := \sum_{|\alpha|=j\leq m}\tau^{2(m-|\alpha|-j)}\|D^\alpha D^j_tf\|_{L^2(\RR^{n+1})}
\]
\item The Fourier transform of a function $f(x)\in L^1(\RR^n)$ will be denoted either by $\mathfrak{F}[f]$ and $\hat{f}$, where
\[
\mathfrak{F}[f](\xi) = \frac{1}{(2\pi)^n}\int e^{-ix\cdot\xi}f(x)dx.
\]
\item $\mathcal{D}'(\RR)$ and $\mathcal{S}'(\RR)$ denote the space of distributions and the space of tempered distributions, respectively.
\item We will denote by $S^{j}(\RR\times\RR^n\times\RR^n)$ the H\"ormander symbol class of order $j$ and type $(1,0)$, this is, the space of all smooth functions in $C^\infty(\RR\times\RR^n\times\RR^n)$ such that for all compact $K\subset \RR\times\RR^n$ and all $\alpha\in\NN^{n+1}$ and $\beta\in \NN^n$, there is a constant $C_{K,\alpha,\beta}$ such that 
\[
|\partial^\alpha_{t,x}\partial^\beta_\xi a(t,x,\xi)|\leq C(1+|\xi|)^{j-|\beta|},\quad\forall (t,x,\xi)\in K\times\RR^n.
\]
It is a Fr\'echet vector space when equipped with the topology generated by the seminorms consisting of the best constants in the inequalities above.

Symbols $a(x,\xi)$ and the symbol class $S^j(\RR^n\times\RR^n)$ are defined analogously.

\item In many estimations we will write $C$ to denote a generic positive constant that might change from line to line. If not apparent on the context, the dependence of this constant on the several parameters of the problem will be pointed out.
\end{itemize}

\subsection{Caputo's fractional derivative for locally integrable functions}

The standard definition of the Caputo fractional derivative $\partial^\alpha_tu$ requires $u$ to have an integrable first derivative (see, for instance, \cite{Pod98}). Based on the fact that the operation of convolution can be applied to two distributions in $\RR$ with support just bounded from below (see \cite{Gelfand_book} and also \cite{LiLiu} for an equivalent definition employing partitions of unity), this fractional operator has been recently extended in \cite{LiLiu} to functions belonging to $L^1_{loc}[0,T)$ ---the set of all functions that are absolutely integrable on any compact set $K\subset[0,T)$. For the sake of completeness, we briefly describe the main definitions and properties of the generalized convolution and Caputo derivative, and we refer the reader to the aforementioned paper for more details.


\begin{definition}[Definition 2.1 in \cite{LiLiu}]\label{def:conv}
    Given $f,g\in\mathcal{D}'(\RR)$ such that there are constants $M_f,M_g\in\RR$ satisfying $\supp(f)\subset[-M_f,\infty)$ and $\supp(g)\subset[-M_g,\infty)$, we set
    \[
    \langle f*g,\varphi\rangle := \sum_{j}\langle f*(h_j g),\varphi\rangle,\quad\forall \varphi\in C_c^\infty(\RR),
    \]
    where $\{h_j\}_{j\geq 0}$ is any locally finite and compactly supported partition of unity in $\RR$, i.e., $h_j\in C_c^\infty(\RR)$;  $0\leq h_j\leq 1$; $\sum h_j(x)=1$ for all $x\in\RR$; on any compact set $K\subset\RR$ there are only finitely many $h_j$'s with support intersecting $K$.

\begin{remark}\label{rmk:def_conv}
    Notice that for a partition of unity as above and
    for any test function $\varphi\in C^\infty_c(\RR)$, the sum in the previous definition is always finite. In fact, the function $t\mapsto \langle (h_jg)(s),\varphi(t+s)\rangle$ has support contained in
    \[
    \text{supp}(\varphi)-(\text{supp}(h_j)\cap [-M_g,\infty)),
    \]
    and it is clear that, because of locally finiteness, for all $j$ large enough
    \[
    [-M_f,\infty)\cap\left(\text{supp}(\varphi)-(\text{supp}(h_j)\cap [-M_g,\infty))\right)=\emptyset.
    \]
    Moreover, for a fixed time $T>0$, if we restrict our test function to those with support in $[0,T]$ we can find an integer $N_T>0$ for which
    \[
    \langle f*g,\varphi\rangle := \sum_{j=0}^{N_T}\langle f*(h_j g),\varphi\rangle,\quad\forall \varphi\in C_c^\infty([0,T]).
    \]
\end{remark}    
\end{definition}

The extension of the convolution defined above satisfies the following properties. More details can be found in \cite{LiLiu}.
\begin{lemma}[Lemma 2.2 in \cite{LiLiu}]

\begin{enumerate}
\item[(i)] The definition is independent of $\{h_j\}$ and agrees with the usual definition of convolution between distributions whenever one of the two distributions is compactly supported.
\item[(ii)] For $f,g$ as in Definition \ref{def:conv}, $f*g\in\mathcal{D}'(\RR)$ and there is a constant $M\in \RR$ such that $\supp(f*g)\subset[-M,\infty)$.
\item[(iii)] $f*g=g*f$,\; $f*(g*h) = (f*g)*h$.
\item[(iv)] Denoting by $D_t$ the distributional derivative, then, for $f,g$ as in Definition \ref{def:conv} we have
\[
(D_tf)*g = D_t(f*g) = f*D_tg.
\]
\end{enumerate}
\end{lemma}

\begin{definition}[Definition 3.4 in \cite{LiLiu}]\label{def:Caputo}
    Let $0<\alpha<1$ and $f\in L^1_{loc}[0,\infty)$ such that $t=0$ is a Lebesgue point from the right, this is, there is $f_0\in\RR$ for which 
    \[
    \lim_{t\to 0^+}\frac{1}{t}\int^t_0|f-f_0|dt=0.
    \]
    The Caputo derivative of order $\alpha$ is defined as
    \[
    \partial^\alpha_t:f\mapsto \Phi_{-\alpha}*(H(t)(f-f_0)),
    \]
    where the kernel $\Phi_{-\alpha}$ corresponds to the following tempered distribution
    \[
    \Phi_{-\alpha}(t) = \frac{-1}{\Gamma(1-\alpha)}\frac{d}{dt}\left(H(t)t^{-\alpha}\right),\quad\text{or equivalently}\quad \mathfrak{F}[\Phi_{-\alpha}](\omega) = (i\omega)^\alpha. 
    \]
   Here, $\mathfrak{F}$ stands for the Fourier transform in the time variable. 
   The convolution above makes sense in light of Definition \ref{def:conv} since both are distributions supported in $[0,\infty)$. 

    Following \cite{LiLiu}, we verify that it is indeed an extension of the usual Caputo's derivative. For $f$ as in the previous definition and such that $f'\in L^1([0,\infty))$, we see that
    \[
    \begin{aligned}
    \Phi_{-\alpha}*(H(t)f) &= \frac{1}{\Gamma(\alpha)}H(t)t^{-\alpha}*\frac{d}{dt}(H(t)f)\\
    &=\frac{1}{\Gamma(\alpha)}H(t)t^{-\alpha}*(H'(t)f_0 + H(t)f')
    \end{aligned}
    \]
    therefore,
    \[
    \partial^{\alpha}_tf = \Phi_{-\alpha}*(H(t)(f-f_0)) = \frac{1}{\Gamma(\alpha)}H(t)t^{-\alpha}*H(t)f' = \frac{1}{\Gamma(\alpha)}\int^t_0\frac{f'(s)}{(t-s)^{\alpha}}ds.
    \]
    
    In addition, for $f\in L^1_{loc}(\RR)$ and since $H(t)h_j(f-f_0)\in\mathcal{S}'(\RR)$ for all functions $h_j$ in the partition of unity from Definition \ref{def:conv}, by Plancherel's theorem we have that for any $\varphi\in C^\infty_c(\RR)$,
    \[
    \begin{aligned}
    \langle\partial^\alpha_tf,\varphi\rangle &= \sum_j \langle \Phi_{-\alpha}*(h_jH(f-f_0)),\varphi\rangle\\
    &= \sum_j \langle (i\omega)^{\alpha}\mathfrak{F}[h_jH(f-f_0)],\mathfrak{F}[\varphi]\rangle\\
    &= \sum_j \langle \mathfrak{F}^{-1}\left[(i\omega)^{\alpha}\mathfrak{F}[h_jH(f-f_0)]\right],\varphi\rangle.
    \end{aligned}
    \]
    Denoting $(\partial^\alpha_tf)_j := \mathfrak{F}^{-1}\left[(i\omega)^{\alpha}\mathfrak{F}[h_jH(f-f_0)]\right]$ we have\begin{equation}\label{eq:Fourier_Caputo}
    \langle\partial^\alpha_tf,\varphi\rangle = \sum_j \langle (\partial^\alpha_tf)_j,\varphi\rangle,
    \end{equation}
    where, formally, we can write
    \begin{equation}\label{eq:Fourier_Caputo_a}
     (\partial^\alpha_tf)_j = \frac{1}{2\pi}\int\int^\infty_0 e^{i(t-s)\omega}(i\omega)^\alpha h_j(s)(f(s)-f_0)dsd\omega.
    \end{equation}
\begin{remark}\label{rmk:partition_unity}
    Without loss of generality, and for simplicity, we will assume that the locally finite and compactly supported partition of unity involved in the next computations is such that $h_0(0)\neq 0$ and $\text{supp}(h_j)\cap\{0\}=\emptyset$ for all $j\neq 0$.
\end{remark}
    
    \begin{remark}\label{rmk:Caputo_der}
    The integrals in \eqref{eq:Fourier_Caputo_a} are non-trivial and absolutely convergent (thus, rigorously defined by \eqref{eq:Fourier_Caputo_a}) for each $h_j$ such that $\supp(h_j)\subset \RR_+$. Furthermore, the function
    \[
    t\mapsto \frac{1}{2\pi}\int\int^\infty_0 e^{i(t-s)\omega}(i\omega)^\alpha h_j(s)f_0dsd\omega= f_0\cdot\mathfrak{F}^{-1}\left[(i\omega)^{\alpha}\mathfrak{F}[h_j](\omega)\right](t)
    \]
    is smooth in $[0,\infty)$ for all $j\geq 1$, therefore
    \[
    (\partial^\alpha_tf)_j = \frac{1}{2\pi}\int\int^\infty_0 e^{i(t-s)\omega}(i\omega)^\alpha h_j(s)f(s)dsd\omega\quad\text{mod}\quad C^\infty([0,\infty)).
    \]
    For $j=0$ instead, the integral has to be interpreted in the following sense of distributional limit,
    \begin{equation}\label{eq:Fourier_Caputo_b}(\partial^\alpha_tf)_0(t):=\overset{\mathcal{D}'(\RR)}{\lim_{\epsilon\to 0}}\;\frac{1}{2\pi}\int\int^\infty_0 e^{i(t-s)\omega}(i\omega)^\alpha h_0(s)(f(s)-f_0)\chi(\epsilon \omega)dsd\omega
    \end{equation}
    where $\chi\in \mathcal{S}(\RR)$ is such that $\chi(0)=1$. We will use the following notation
    \begin{equation}\label{eq:Fourier_Caputo_c}
    (\partial^\alpha_tf)_{0,\epsilon}(t) := \frac{1}{2\pi}\int\int^\infty_0 e^{i(t-s)\omega}(i\omega)^\alpha h_0(s)(f(s)-f_0)\chi(\epsilon \omega)dsd\omega,
    \end{equation}
    hence, \eqref{eq:Fourier_Caputo_b} rewrites as 
    \begin{equation}\label{eq:Fourier_Caputo_d}
(\partial^\alpha_tf)_0(t):=\overset{\mathcal{D}'(\RR)}{\lim_{\epsilon\to 0}} (\partial^\alpha_tf)_{0,\epsilon}(t).
\end{equation}
    \end{remark}
\end{definition}
\bigskip

\subsection{The initial value problem in PAT with fractional attenuation}

We begin with the singular-source formulation of the photoacoustic problem, representing rapid heat deposition of electromagnetic energy and the subsequent ultrasound propagation, which due to causality, we assume no waves are propagating at negative times. In addition, we assume that the attenuation integral runs from $-\infty$ to the present time $t>0$, which means that in principle the strength  of the attenuation depends on the whole past history of the wave field. We then arrive at the following inhomogeneous system 
 \begin{equation}\label{frac_PAT_inhomog}
\left\{\begin{array}{l}
\partial_t^2U - c^2\Delta U + \mathcal{A}[U] = \delta'(t)\otimes f(x),\\
U|_{t<0}=0.
\end{array}\right.
\end{equation}
for an attenuation operator defined in the distributional sense by $\mathcal{A}[U] := a(x)\left(U(t)*\Phi_{-\alpha}(t)\right)$.
This version of the fractional derivative coincides with the Riemann-Liouville and Caputo derivative from $-\infty$, for functions with sufficient regularity. 

The relation between the inhomogeneous problem and the IVP is established in the next lemma.
\begin{lemma}
    Let $u$ be a smooth solution to the IVP \eqref{frac_PAT} in $\RR^n\times(0,\infty)$. Then, $U(x,t)=H(t)u(x,t)$ solves \eqref{frac_PAT_inhomog} in the sense of distributions.
\end{lemma}
\begin{proof}
For an arbitrary test function $\phi\in C^\infty_c(\RR^{n+1})$ we have
\[
\begin{aligned}
\langle \mathcal{A}[Hu],\phi\rangle &=\sum_j \langle a\Phi_{-\alpha}*(h_jHu),\phi\rangle\\
&= \sum_j\left\langle a(h_jHu)(t),\left\langle \textstyle\frac{-1}{\Gamma(1-\alpha)}\partial_t\left(H(t)t^{-\alpha}\right)(s),\phi(t+s)\right\rangle  \right\rangle \\
&= \sum_j\left\langle a(h_jHu)(t),\frac{1}{\Gamma(1-\alpha)}\int^\infty_t\frac{\phi'(r)dr}{(r-t)^\alpha} \right\rangle\\
&= \frac{1}{\Gamma(1-\alpha)}\int_{\RR^n}a(x)\sum_j\int^\infty_0\int^\infty_t\frac{h_j(t)u(t)\phi'(r)}{(r-t)^\alpha}drdtdx\\
&= \int_{\RR^n}a(x)\int^\infty_0\phi'(r)\left(\frac{1}{\Gamma(1-\alpha)}\int^r_0\Big(\sum_jh_j(t)\Big)u(t)\partial_t\left(\frac{(r-t)^{1-\alpha}}{1-\alpha}\right)dt\right)drdx\\
&= \int_{\RR^n}a(x)\int^\infty_0\phi'(r)\left(\frac{-1}{\Gamma(1-\alpha)}\int^r_0u'(t)\frac{(r-t)^{1-\alpha}}{1-\alpha}dt\right)drdx \\
&\quad - \int_{\RR^n}a(x)\int^\infty_0\frac{\phi'(r)u(0)}{(1-\alpha)\Gamma(1-\alpha)}drdx\\
&= \int_{\RR^n}a(x)\int^\infty_0\phi(r)\partial_r\left(\frac{1}{\Gamma(1-\alpha)}\int^r_0u'(t)\frac{(r-t)^{1-\alpha}}{1-\alpha}dt\right)drdx \\
&\quad + \int_{\RR^n}a(x)\frac{\phi(0)u(0)}{(1-\alpha)\Gamma(1-\alpha)}dx\\
&=\langle Ha\partial^\alpha_tu,\phi\rangle + \frac{1}{{(1-\alpha)\Gamma(1-\alpha)}}\langle a\delta u,\phi\rangle.
\end{aligned}
\]
In consequence, noticing that $(1-\alpha)\Gamma(1-\alpha) = \Gamma(2-\alpha)$ and using that $u$ solves \eqref{frac_PAT}, we obtain
\[
\begin{aligned}
\langle(\boxc+\mathcal{A}) U,\phi \rangle &= \langle\delta'u+2\delta u' +\textstyle\frac{1}{{ \Gamma(2-\alpha)}}a\delta u,\phi \rangle \\
&=\int_{\RR^n} \left[-u(0)\phi'(0) -u'(0)\phi(0) +2u'(0)\phi(0) + \textstyle\frac{1}{{ \Gamma(2-\alpha)}}au(0)\phi(0)\right]dx \\
&=\langle f(x)\delta'(t),\phi\rangle,
\end{aligned}
\]
this is, $U(x,t)$ is a solution to $\boxca U = f\delta'$ in $\RR^{n+1}$, and such that $U=0$ for $t<0$. 
\end{proof}

\subsection{Some properties of the forward problem}\label{sec:forward_pblm}
The existence and uniqueness of solutions for the initial boundary value problem (with null Dirichlet conditions), as well as standard regularity and higher regularity results, were proven in \cite{Ya2022}. We complement the analysis of the system  \eqref{frac_PAT} by establishing finite propagation speed, energy dissipation, and fractional Sobolev regularity. 

The finite propagation speed will allow us to treat the IVP in $(0,T)\times\RR^n$ as an initial boundary value problem (IBVP) with null Dirichlet conditions by taking a sufficiently large ball as the domain, while the energy dissipation will be an important feature in the study of a reconstruction formula.

We begin with a preliminary result that implies a positivity property for the fractional derivative operator.

\begin{lemma}\label{lemma:positivity}
Let $\alpha
\in(0,1)$ and $\cU\subset\RR^n$ be a bounded region with a smooth boundary. The fractional derivative kernel $\frac{t^{-\alpha}}{\Gamma(1-\alpha)}$ defines a positive semidefinite operator in $L^2((0,T)\times \cU)$, meaning that
\[
\left\langle \frac{1}{\Gamma(1-\alpha)}\int^t_0\frac{v(x,s)}{(t-s)^\alpha}ds,v(x,t) \right\rangle \geq 0,\quad \forall v\in L^2((0,T)\times \cU),
\]
where $\langle\cdot,\cdot\rangle$ stands for the inner product in $L^2((0,T)\times \cU)$. 
\end{lemma}
\begin{proof}
 The $L^2(0,T)$--coercivity of the convolution operator with kernel $q(t) = \frac{t^{-\alpha}}{\Gamma(1-\alpha)}$ and $\alpha\in(0,1)$ is proven in \cite{Eg1988}. For smooth functions, it is clear from a pointwise application of the previous and Fubini's theorem that
\[
\left\langle \frac{1}{\Gamma(1-\alpha)}\int^t_0\frac{v(x,s)}{(t-s)^\alpha}ds,v(x,t) \right\rangle = \int_\cU \left(\int^T_0 (q*v)vdt\right)dx \geq 0.
\]
The general case follows by density of smooth function in $L^2((0,T)\times U)$ and the next inequality which is obtained from Young's convolution inequality:
\[
\left|\left\langle \frac{1}{\Gamma(1-\alpha)}\int^t_0\frac{u(x,s)}{(t-s)^\alpha}ds,v(x,t) \right\rangle\right|\leq \frac{T^{1-\alpha}}{\Gamma(2-\alpha)}\|u\|_{L^2((0,T)\times \cU)}\|v\|_{L^2((0,T)\times \cU)}.
\]
\end{proof}

\begin{theorem}\label{thm:energy_frac_wave}
The system \eqref{frac_PAT} with $\alpha\in(0,1)$ has finite propagation speed and is dissipative in the sense that for any $\cU\subset\RR^n$ bounded region with a smooth boundary, if $u$ solves \eqref{frac_PAT} and $u=0$ on $(0,T)\times\partial \cU$ then
\[
E_\cU(t)=E_\cU(0) - 2\langle a\partial^\alpha_tu,c^{-2}\partial_tu\rangle
\]
with
\begin{equation}\label{eq:energy}
E_\cU(t) := \int_{\cU} |u_t|^2c^{-2}dx + \int_{\cU} |\nabla u|^2dx.
\end{equation}
and $\langle a\partial^\alpha_tu,c^{-2}\partial_tu\rangle\geq 0$.
\end{theorem}

\begin{proof}
The finite propagation speed is a consequence of the previous and the coarea formula. Indeed, given any smooth function $\phi(x)$ such that $c^{2}(x)|\nabla\phi(x)|^2\leq 1$, we consider the backward cone
\[
\mathcal{C} = \{(x,t)\in\RR^n\times (0,\tau):\phi(x)<\tau-t\}
\]
whose temporal cross-sections and their boundaries are given respectively by
\[
\mathcal{C}(t) = \{x\in\RR^n: \phi(x)<\tau-t\}\quad\text{and}\quad \partial\mathcal{C}(t) = \{x\in\RR^n: \phi(x)=\tau-t\}.
\]
For any smooth solution $u$ to \eqref{frac_PAT} we denote by $E_{\mathcal{C}(t)}$ its energy inside the region $\mathcal{C}(t)$, therefore
\[
\begin{aligned}
\frac{d}{dt} E_{\mathcal{C}(t)} &= 2\int_{\mathcal{C}(t)}\left(c^{-2} u_tu_{tt} + \nabla u_t\cdot \nabla u\right)dx - \int_{\partial \mathcal{C}(t)}\left(c^{-2}|u_t|^2 + |\nabla u|^2\right)|\nabla \phi|^{-1}dS\\
&=  2\int_{\mathcal{C}(t)}u_t\left(c^{-2} u_{tt} - \Delta u\right)dx + \int_{\partial\mathcal{C}(t)}u_t \partial_\nu u dS\\
&\quad- \int_{\partial \mathcal{C}(t)}\left(c^{-2}|u_t|^2 + |\nabla u|^2\right)|\nabla \phi|^{-1}dS\\
&\leq -2\int_{\mathcal{C}(t)}ac^{-2}u_t\partial^\alpha_tudx.
\end{aligned}
\]
Integrating in time over the interval $(0,\tau)$ we get
\[
E_{\mathcal{C}(\tau)} \leq E_{\mathcal{C}(0)} - 2\int_{\mathcal{C}(0)}a(x)c^{-2}(x)\left(\int^{\tau(x)}_{0}u_t(x,t)\int^t_0\frac{(t-s)^{-\alpha}}{\Gamma(1-\alpha)}u_t(x,s)dsdt\right)dx,
\]
where $\tau(x) = \sup\{t\in(0,\tau):(x,t)\in \mathcal{C}\}$. Since the last term on the right-hand side is nonpositive, we conclude that $E_{\mathcal{C}(0)}=0$ implies $E_{\mathcal{C}(\tau)}=0$, which means there is finite propagation speed.

For the last part of the lemma, we again assume enough regularity and then argue by density. Let's multiply \eqref{frac_PAT} by $c^{-2}u_t$ and integrate over $\cU$, where we recall $u=0$ on $(0,T)\times\partial \cU$. We then have
\[
\frac{d}{dt}E_\cU(t) = 2\int_{\cU} u_t(c^{-2}u_{tt} - \Delta u)dx = -2\int_\cU au_t(\partial^\alpha_tu)dx.
\]
Integrating in time and denoting $v = c^{-1}a^{1/2}u_t$, we obtain that 
\[
E_\cU(T) - E_\cU(0) = -2\int_{(0,T)\times \cU}\left(\frac{1}{\Gamma(1-\alpha)}\int^t_0\frac{v(x,s)}{(t-s)^\alpha}ds\right)v(x,t)dxdt \leq 0.
\]
\end{proof}

Let's precise now the energy spaces of initial conditions. 
Consider $A = c^2(x)\Delta$ endowed with null Dirichlet boundary conditions at the smooth boundary $\partial\Omega$ of some open and bounded region $\Omega\subset\RR^n$. The domain of the operator $A$ is given by
\[
D(A) = H^2(\Omega)\cap H^1_0(\Omega).
\]
We also set $D(A^0) = L^2(\Omega)$ and
\[
D(A^{\theta})= \{f\in L^2(\Omega):A^{\theta} f\in L^2(\Omega)\},\quad\theta\in(0,1),
\]
the respective domains of the fractional-order operators $A^{\theta}$. We define the energy spaces of initial conditions as
\[
\mathcal{H}^1_0(\Omega):= D(A^{1/2})\times D(A^0)\quad\text{and}\quad \mathcal{H}^2_0(\Omega):= D(A)\times D(A^{1/2}),
\]
ans for $s\in(0,1)$ their interpolation spaces 
\[
\begin{aligned}
\mathcal{H}^{1+s}_0(\Omega)= D(A^{(1+s)/2})\times D(A^{s/2}),
\end{aligned}
\]
which are continuously embedded in $H^{1+s}(\Omega)\times H^{s}(\Omega)$. In particular, the space $\mathcal{H}^1_0(\Omega)$ is topologically equivalent to the completion of $C^\infty_c(\Omega)\times C^\infty_c(\Omega)$ under the energy norm
\begin{equation}\label{def:energy_norm}
\|(f,g)\|^2_{H^1_0(\Omega)\times L^2(\Omega)}:= \int_{\Omega}(|\nabla f(x)|^2+|g(x)|^2)dx.
\end{equation}
For more details on interpolation, Sobolev spaces and the defintion of $A^{\theta}$ and $D(A^\theta)$, see Appendix \ref{Appx:Sobolev}.

\begin{theorem}\label{thm:regularity}
    Let $u$ be the unique solution to 
    \begin{equation}\label{IBVP_att}
        \left\{\begin{array}{ll}
        \square_{c,\alpha}u = F,& (0,T)\times\Omega\\
        u=0,&(0,T)\times\partial\Omega\\
        (u,\partial_tu)|_{t=0} = (f,g),& \Omega,
        \end{array}\right.
    \end{equation}
    with $(f,g)\in \cH^{1+s}_0(\Omega)$, 
    $F\in L^1([0,T];D(A^{s/2}))$
    for some $s\in(0,1)$.
    Then, 
    \[
    u\in C^1([0,T];D(A^{s/2}))\cap C([0,T];D(A^{(1+s)/2})),
    \]
    and the following continuous dependence holds
    \[
\begin{aligned}
\|(u,\partial_tu)\|_{H^{1+s}(\Omega)\times H^s(\Omega)}&\leq C\left(\|(f,g)\|_{\mathcal{H}^{1+s}_0(\Omega)}+\|F\|_{L^1([0,T];D(A^{s/2}))}\right).
\end{aligned}
\]
\end{theorem}
\begin{remark}
    Notice above that for $(f,g)$ and $F$ as in the statement of the theorem, they in particular satisfy that $(f,g)\in H^1_0(\Omega)\times L^2(\Omega)$ and $F\in L^1([0,T];L^2(\Omega))$, therefore, the existence and uniqueness of the solution in the energy space $H^1_0(\Omega)\times L^2(\Omega)$ follows from \cite{Ya2022}.
\end{remark}
\begin{proof}
Let's consider the solution operator for the unattenuated wave equation, 
\begin{equation}\label{def:sol_op_hom}
\UU(t):(f,g)\mapsto (u(t),\partial_tu(t)),\quad \text{for $t>0$}
\end{equation}
where $u$ solves
 \begin{equation}\label{eq:unatt_wave}
\left\{\begin{array}{ll}
\partial_t^2u - c^2\Delta u  = 0,& (0,T)\times\Omega \\
u=0,& (0,T)\times\partial\Omega\\
(u,\partial_tu)|_{t=0} = (f,g),& \Omega.
\end{array}\right.
\end{equation}
By interpolation, and standard regularity and higher regularity results, we have that the map $t\mapsto \UU(t)$ belongs to $C([0,T];\mathcal{B}(\cH^{1+s}_0(\Omega)))$, where for a given Banach space $X$ write $\mathcal{B}(X)$ to denote the Banach space of bounded linear operators from $X$ onto itself. 

Furthermore, the Duhamel's principle implies that the inhomogeneous solution operator
\[
\VV(t):(f,g,F)\mapsto (v(t),\partial_tv(t)),
\]
with $v$ solution to
 \begin{equation}\label{eq:unatt_wave}
\left\{\begin{array}{ll}
\partial_t^2v - c^2\Delta v  = F,& (0,T)\times\Omega \\
v=0,& (0,T)\times\partial\Omega\\
(v,\partial_tv)|_{t=0} = (f,g),& \Omega,
\end{array}\right.
\end{equation}
can be written in the form
\begin{equation}\label{def:sol_op_inhom}
\VV(t)[f,g,F] = \UU(t)[f,g] + \int^t_0\UU(t-s)[0,F(s)]ds,
\end{equation}
from where it is clear that for a fixed 
$F\in L^1([0,T],D(A^{s/2}))$, 
\[
t\mapsto \VV(t)[\cdot,\cdot,F]\in C([0,T];\mathcal{B}(\cH^{1+s}_0(\Omega))).
\]
For more details about integration on Banach spaces, we refer the reader to Appendix \ref{Appx:int_Banach}. Moreover, by interpolation it is clear that
\[
\|\VV(\cdot)[f,g,F]\|_{C([0,T];\mathcal{H}^{1+s}_0(\Omega))}\leq C\left(\|(f,g)\|_{\mathcal{H}^{1+s}_0(\Omega)}+\|F\|_{L^1([0,T];D(A^{s/2}))}\right).
\]

Following \cite{Ya2022}, we use the previous two operators along with Duhamel's principle to write the solution $\UU_\alpha(t)=(u(t),\partial_tu(t))$ to \eqref{IBVP_att} as
\begin{equation}\label{eq:op_Ualpha}
\UU_\alpha(t)=\VV(t)[f,g,F] + \int^t_0\int^s_0\UU(t-s)Q(s-\tau)\UU_\alpha(\tau)d\tau ds,
\end{equation}
where $Q(t)$ is defined as
\[
Q(t):=-\frac{t^{-\alpha}}{\Gamma(1-\alpha)}\left(\begin{matrix}0&0\\0&a(x)\end{matrix}\right),\quad \text{for}\quad t\in (0,T].
\]
It follows from Young's convolution inequality that the operation of convolution against the kernel $Q(t)$ maps continuously $L^1((0,T);\cH^{1+s}_0(\Omega))$ into itself.

Denoting by $
(\widetilde{\mathcal{G}}\UU_\alpha)(t)$ the double integral on the right-hand side of \eqref{eq:op_Ualpha} (for $f,g,F$ fixed) and $(\mathcal{G}\UU_\alpha)(t):= \VV(t)[f,g,F] + (\widetilde{\mathcal{G}}\UU_\alpha)(t)$, equation \eqref{eq:op_Ualpha} rewrites as $\UU_\alpha = \mathcal{G}\UU_\alpha$, which means that we are looking for a fixed point for the operator $\mathcal{G}$ in some subspace $X_0\subset C([0,T];\cH^{1+s}_0(\Omega))$. 

The Banach fixed point theorem demands to find a contracting iterate $\mathcal{G}^m$ for some $m\geq 1$. It follows directly from \cite{Ya2022} that there is a constant $C_1>0$ such that for all $m\in\NN$,
\[
\|\widetilde{\mathcal{G}}^m\VV\|_{C([0,T];\cH^1_0)}\leq \frac{C^m_1T^{m(2-\alpha)-1}}{\Gamma(m(2-\alpha)-1)}\|\VV\|_{C([0,T]:\cH^1_0)},\quad\forall \VV\in C([0,T];\cH^1_0),
\]
and similarly, there is another constant $C_2>0$ such that
\[
\|\widetilde{\mathcal{G}}^m\VV\|_{C([0,T];\cH^2_0)}\leq \frac{C^m_2T^{m(2-\alpha')-1}}{\Gamma(m(2-\alpha')-1)}\|\VV\|_{C([0,T];\cH^2_0)},\quad\forall \VV\in C([0,T];\cH^2_0),
\]
with $\alpha' = \frac{1+\alpha}{2}$. The constants $C_i$, $i=1,2$, depend on $\Omega, T, \alpha$ and $a$.

Denoting by $\delta_{1,m},\delta_{2,m}>0$ the respective constants on the right-hand side of the previous inequalities, by properties of the Gamma function one deduces that $\delta_{i,m}<1$ for $i=1,2$, provided that one takes $m\geq 1$ large enough.

Setting
\[
X_0 = [C([0,T];\cH^2_0),C([0,T];\cH^1_0)]_{1-s},
\]
which is a Banach space when equipped with the interpolation norm $\|\cdot\|_{X_0}$ (see Appendix \ref{Appx:Sobolev} for its definition), we get from the previous that
\[
\|\widetilde{\mathcal{G}}^m\VV\|_{X_0}\leq \delta_{2,m}^{s}\delta_{1,m}^{1-s}\|\VV\|_{X_0},\quad\forall \VV\in X_0.
\]
This implies that for any pair $\VV,\WW\in X_0$,
\[
\|\mathcal{G}^m\VV-\mathcal{G}^m\WW\|_{X_0}\leq \delta_{2,m}^{s}\delta_{1,m}^{1-s}\|\VV-\WW\|_{X_0},
\]
which means $\mathcal{G}^m$ is a contraction in $X_0$, and consequently, there is a fixed point $\UU\in X_0$ and only remains to show that
\[
X_0\subset C([0,T];\cH^{1+s}_0(\Omega)).
\]

A function $\UU(t,x)$ belonging to the interpolation space $X_0$ can be decomposed as $\UU = \UU_2 + \UU_1$ with $\UU_i\in C([0,T];\cH^{i}_0(\Omega))$, $i=1,2$. Moreover, for all $t\in[0,T]$ and $\rho>0$ we have
\[
\begin{aligned}
\|\UU_{2}(t)\|^2_{\cH^{2}_0(\Omega)}+\rho^2 \|\UU_{1}(t)\|^2_{\cH^{1}_0(\Omega)}\leq \|\UU_2\|^2_{C([0,T];\cH^{2}_0(\Omega))} +\rho^2\|\UU_1\|^2_{C([0,T];\cH^{1}_0(\Omega))},
\end{aligned}
\]
thus, taking square root and infimum with respect to the decomposition $\UU = \UU_2 + \UU_1$ we deduce the following inequality between the $K$-functionals for the respective pairs of spaces $(\cH^{2}_0(\Omega),\cH^{1}_0(\Omega)))$ and $(C([0,T];\cH^{2}_0(\Omega)),C([0,T];\cH^{1}_0(\Omega)))$:
\[
K(\rho,\UU(t))\leq K(\rho,\UU),\quad\forall t\in[0,T].
\]
The definitions of the K-functional and the interpolation norm are included in Appendix \ref{Appx:Sobolev}.

It follows directly then that $\|\UU(t)\|_{\cH^{1+s}_0(\Omega)}\leq \|\UU\|_{X_0}$ for all $t\in[0,T]$ and subsequently,
\[
\|\UU\|_{C([0,T];\cH^{1+s}_0(\Omega))}\leq \|\UU\|_{X_0}.
\]
The continuity inequality follows from the previous, the fixed-point formulation of equation \eqref{IBVP_att}, the continuous embedding of $\mathcal{H}^{1+s}_0(\Omega)$ into $H^{1+s}(\Omega)\times H^s(\Omega)$, and the fact that for $m$ large enough, $\|\tilde{\mathcal{G}}^m\|_{\mathcal{B}(X_0)}\leq \delta_m<1$ with $\delta_m\to 0$ as $m\to\infty$. Indeed, 
\[
\begin{aligned}
\|\UU\|_{X_0}=\|\mathcal{G}^m\UU\|_{X_0}
&\leq\sum^{m-1}_{k=0}\|\tilde{\mathcal{G}}\VV(\cdot)[f,g,F]\|_{X_0} + \|\tilde{\mathcal{G}}^m\UU\|_{X_0}\\
 &\leq C\left( \|\VV(\cdot)[f,g,F]\|_{X_0} + \delta_m\|\UU\|_{X_0}\right)\\
  &\leq C\left( \|\VV(\cdot)[f,g,F]\|_{C([0,T];\mathcal{H}^{1+s}_0(\Omega))} + \delta_m\|\UU\|_{X_0}\right),
\end{aligned}
\]
thus, by taking $m$ large enough we can absorb the last term on the right and and obtain
\[
\begin{aligned}
\|\UU\|_{H^{1+s}(\Omega)\times H^s(\Omega)}&\leq C \|\UU\|_{X_0}\\
&\leq C\|\VV(\cdot)[f,g,F]\|_{C([0,T];\mathcal{H}^{1+s}_0(\Omega))}\\
&\leq C\left(\|(f,g)\|_{\mathcal{H}^{1+s}_0(\Omega)}+\|F\|_{L^1([0,T];D(A^{s/2}))}\right).
\end{aligned}
\]
\end{proof}

\section{Uniqueness}
In this section, we generalize the uniqueness result presented in \cite{Ya2022} by allowing more general foliations by strictly convex surfaces. In \cite{Ya2022}, the authors address the constant sound speed case, and it is clear, as pointed out there, that the argument can be extended to variable speeds subject to the existence of a function with strictly convex level sets (for the sound speed metric $c^{-2}dx^2$). A common particular instance of this requirement consists in imposing the sound speed to satisfy an inequality of the form
\[
\frac{(x-x_0)\cdot\partial c(x)}{c(x)}<1,\quad\forall x\in\Omega.
\]
for some fixed $x_0\in\RR^n$. In geometric terms, the previous inequality is equivalent to requiring the Euclidean spheres $\{x:|x-x_0|=r\}$, with $r>0$, to be strictly convex under the sound speed metric.

The main result of this section states that provided the existence of an admissible foliation by strictly convex hypersurfaces, null boundary observations at $(0,T)\times\Gamma$ imply a vanishing initial condition $u_0$ in a domain of influence associated to $(0,T)\times\Gamma$ and the foliation.

Below, we define what we mean by an admissible foliation of $\Omega$ with respect to the observation region $\Gamma$. We point out that this is a more restrictive type of foliation than those considered previously in \cite{StUh2013}. The reason for considering this family of foliations lies in the fact that, due to the time-convolution term, we do not have a John's (or double-cone) unique continuation-type result at our disposal, and we are forced to move in small increments and via local unique continuation results, perpendicular to the foliation. This contrasts with standard unique continuation results where the zero-set typically advances across balls for the sound speed metric, which might not be strictly convex for large radius.

\begin{definition}\label{def:adm_foliation}
    Let $\Omega_1\Supset \Omega$ be a larger domain with a smooth boundary. 
    We say that a one-parameter family of smooth, compact, and oriented hypersurfaces $\Sigma(\Omega,\Gamma)=\{\Sigma_s\}_{s\in[\underline{s},\overline{s}]}$ is 
    an {\em admissible foliation of $(\Omega,\Gamma)$}
    if the following conditions are met: 
    \begin{itemize}
        \item[(a)] (uniformly continuous family) for every $\epsilon>0$ there is $\delta>0$ such that the Hausdorff distance $\text{dist}_H(\Sigma_s,\Sigma_{s'})<\epsilon$ whenever $|s- s'|<\delta$;
        \item[(b)] each $\Sigma_s$ divides $\Omega_1$ into two open connected parts which we denote as $\Sigma_{s}^{\text{int}}$ and $\Sigma_{s}^{\text{ext}}$, with the later region satisfying that its closure contains $\partial\Omega_1$. Moreover, $\Omega\subset \Sigma^{\text{int}}_{\underline{s}}$, $\overline{\Sigma_{s}^{\text{int}}}\subset\Sigma_{s'}^{\text{int}}$ for all $s>s'$, and $\Sigma_s\cap (\partial\Omega\setminus \Gamma) = \emptyset$ for all $s$;
        \item[(c)] (normal vector field and its integral curves) for any $s$ and $x\in\Sigma_s$ we let $\nu(x)$ be the interior unit normal vector, this is,  pointing towards $\Sigma_s^{\text{int}}$. This gives to each leaf $\Sigma_s$ a global orientation. The resulting vector field $\nu$ is assumed to be $C^2$ and we require that its integrals curves  passing through $\Omega\cap\Sigma_{\overline{s}}^{\text{ext}}$ intersect the observation region $\Gamma$ only once;

        \item[(d)] (flat coordinates) for every $s\in[\underline{s},\overline{s}]$ and at each point $x\in\Sigma_s$ there is a flat chart $(U,\varphi)$ containing $x$. In other words, there is a continuous function $\zeta(s')$ such that each leaf $\Sigma_{s'}$ intersecting $U$ takes the local form $\{x=(x',x^n) : x^n=\zeta(s')\}$.
        
        \item[(e)] (boundary normal coordinates) for every $s\in[\underline{s},\overline{s}]$ there is an $\epsilon_s>0$  such that boundary normal coordinates can be prescribed in a collar neighborhood of $\Sigma_s$ of the form $\{x\in\Omega_1:\dist(x,\Sigma_s)<\epsilon_s\}$;

        \item[(f)] (strict convexity) for all $s$, the surfaces $\Sigma_s$ are strictly convex with respect to the sound speed metric $c^{-2}dx^2$ ($c=1$ outside $\Omega$).
        In other words, for every $s\in[\underline{s},\overline{s}]$ there is $\kappa_s$ (principal curvature) such that in boundary normal coordinates, where the metric takes the form $g_{\alpha\beta}(x',x^n)dx^\alpha dx^\beta + (dx^ n)^2$,
        \[
        \partial_{x^n}g^{ij}\xi_i\xi_j\geq \kappa_s|\xi|^2_g,\quad\forall \xi\in\RR^n.
        \]
        More details about convexity of hypersurfaces can be found in Appendix \ref{Appdx:convexity}.
    \end{itemize} 
\end{definition}

We denote by $\dist(x,y)$ the distance function for the sound speed metric $c^{-2}dx^2$ and by $\dist_\nu(x,y)$ the length (in the sound speed metric) of the integral curve segment of $\nu$ connecting $x$ and $y$ (if exists). We will write by $\dist_\nu(x,\Sigma_s)$ the length of the segment of integral curve passing through $x$ and $\Sigma_s$, which is well defined for all $x\in \Omega\cap\Sigma^{\text{ext}}_{\overline{s}}$ and $C^2$ due to condition (c).

Next, we define a region of influence associated with the foliation and the observation set $(0,T)\times\Gamma$.

\begin{definition}\label{def:region_influence}
    Denoting 
    \[
    \begin{aligned}
    &s_0=\inf\{s\in[\underline{s},\overline{s}]:\Sigma_s\cap\overline{\Omega}\neq \emptyset\}\quad\text{and}\\
    &s_T=\sup\{s\in [\underline{s},\overline{s}]:\forall x\in \Sigma_s\cap \overline{\Omega},\; T>\dist_\nu(x,\Sigma_{s_0})\},
    \end{aligned}
    \]
    we call $D_{\Sigma,T} := \bigcup_{s\in[s_0,s_T)} \Sigma_s\cap\Omega$ the {\em domain of influence} associated with the foliation, and for the positive $C^2$-function 
    \[
\tau(x) := T - \dist_\nu(x,\Sigma_{s_0}),\quad x\in D_{\Sigma,T},
\]
we define the {\em region of influence} for the pair $(\Sigma,T)$ as
\[
\mathcal{D}_{\Sigma,T}:=\{(t,x)\in (0,T)\times D_{\Sigma,T}\;:\: 0<t<\tau(x)\}.
\]
\end{definition}
An typical example of a valid configuration of a domain, observation region and foliation, and some of its relevant components are illustrated in figure \ref{fig:foliation}.

\begin{figure}[htp]
    \centering
    \includegraphics[width=8cm]{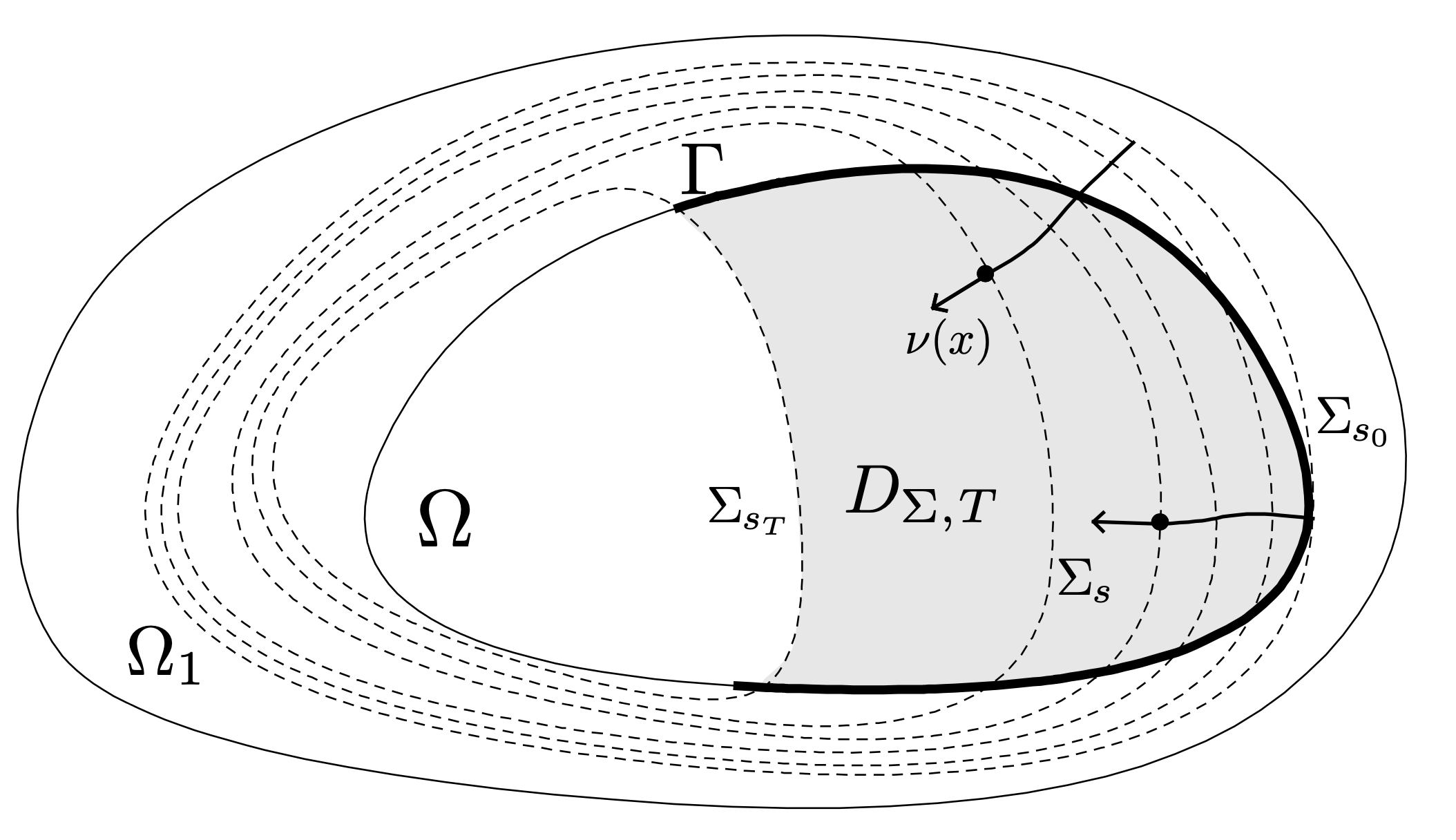}
    \caption{An example of an admissible foliation and its domain of influence. Two integral curve segments for the unit normal vector field $\nu$ are depicted, as well as the limiting leaves $\Sigma_{s_0}$ and $\Sigma_{s_T}$.}
    \label{fig:foliation}
\end{figure}

We provide a few remarks concerning some of the conditions imposed on the foliation.

\begin{remark}
Condition (b) implies that $\Gamma':=\bigcup_{s\in[s_0,s_T]} \Sigma_s\cap \partial\Omega$ is a closed subset of $\Gamma$.
\end{remark}

\begin{remark}
Regarding condition (c), the assumption on the intersection between the integral curves and the observation region $\Gamma$ is not intended to be optimal; it is introduced solely to simplify the exposition.
\end{remark}

\begin{remark}\label{rmk:flat_coor} ({\em uniform choice of parameters})
    Due to the compactness of the region $\overline{\mathcal{D}_{\Sigma,T}}$, there are positive constants $\epsilon_0,\delta_0,\overline{\kappa},\underline{\kappa}$ such that: at any point $x\in\Sigma_s$ with $s\in[s_0,s_T]$ we can prescribe flat coordinates in the open balls $B_{\delta_0}(x)$; for all $s\in[s_0,s_T]$ there are boundary normal coordinates in $\{x\in\overline{\mathcal{D}_{\Sigma,T}}:\dist(x,\Sigma_s)<\epsilon_0\}$ with respect to $\Sigma_s$; and finally, the principal curvatures $\kappa_s$ satisfy that $\underline{\kappa}\leq \kappa_s\leq \overline{\kappa}$ for all $s\in[s_0,s_T]$. We will assume, without loss of generality, that $\epsilon_0<\delta_0$.
\end{remark}

Prior to presenting the main result of this section and the associated uniqueness property for the inverse photoacoustic tomography problem, we introduce a final piece of notation. 

Let's set $\Omega_s = \Omega\cap \Sigma^{\text{int}}_s$ and $\Gamma_s = (\Gamma\cap\Sigma_s^{\text{int}})\cup(\Sigma_s\cap \overline{\Omega})$; for $s\in[s_0,s_T)$ and $x\in\mathcal{D}_{\Sigma,T}\cap \Omega_s$ close enough to $\Sigma_s$, we denote by $\pi_s(x)$ its projection onto $\Sigma_s$, that is, the closest point to $x$ in $\Sigma_s$ for the sound speed metric, and by $\pi^\nu_s(x)$ the projection of $x$ onto $\Sigma_s$ along the integral curves of $\nu$.

\begin{theorem}\label{thm:uniqueness_att}
    Let $\{\Sigma_s\}_{s\in[\underline{s},\overline{s}]}$ be an admissible foliation of $(\Omega,\Gamma)$ according to definition \ref{def:adm_foliation}, and let $\mathcal{D}_{\Sigma,T}$ be the region of influence associated to the pair $(\Sigma,T)$ for an observation time $T>0$. If $u$ is a solution to \eqref{frac_PAT} with $\supp(u_0)\subset \Omega$ and null Cauchy data on $(0,T)\times\Gamma$, then, $u=0$ in $\mathcal{D}_{\Sigma,T}$
    and in particular
    $\text{\rm supp}(u_0) \cap D_{\Sigma,T}=\emptyset$.
\end{theorem}

\begin{corollary}\label{cor:injectivity}
    Let $\Sigma(\Omega,\Gamma)=\{\Sigma_s\}_{s\in[\underline{s},\overline{s}]}$ be an admissible foliation according to definition \ref{def:adm_foliation} and $\mathcal{K}\subset\Sigma^{\text{ext}}_{\overline{s}}\cap\Omega$ a compact subdomain with smooth boundary. Provided that
    \[
    T>T_0(\mathcal{K},\Sigma):= \sup\{\dist_{\nu}(x,\Sigma_{s_0}):x\in\mathcal{K}\},
    \]
   then the observation map $\Lambda_\alpha$ is injective in $\mathcal{H}^1_0(\mathcal{K}):=\{f\in H^1(\Omega):\supp(f)\subseteq\mathcal{K}\}$. In the case of complete data, $\Gamma=\partial\Omega$, we might replace $\mathcal{K}$ with $\Omega$ and $T_0(\Omega,\Sigma)$ become the length of the largest integral curve segment contained in $\Omega$.
\end{corollary}
The proof of the theorem is based on a layer-stripping argument and boundary unique continuation for the attenuated wave equation. This last property follows from two types of Carleman estimates associated to the systems satisfied by $u(t,x)$ and the one by $\bu(t,x) := \int^t_0u(s,x)ds$,
namely,
 \begin{equation}\label{frac_PAT_int}
\left\{\begin{array}{l}
\partial_t^2\bu - c^2\Delta \bu + a(x)\partial^\alpha_t\bu = b(t,x)u_0(x),\quad \text{in}\:(0,T)\times\Omega,\\
(\bu,\partial_t\bu)|_{t=0} = (0,u_0),
\end{array}\right.
\end{equation}
with $b(t,x) = \frac{a(x)}{\Gamma(2-\alpha)}(t^{1-\alpha}-1)$. 
The main ingredients of the proof are the content of the next technical lemmas which we shall state and leave their proofs after we prove Theorem \ref{thm:uniqueness_att}.

\begin{lemma}\label{lemma1_thm5}
    For any $s\in[s_0,s_T)$, assume that $\supp(u_0)\subset\overline{\Omega}_s$, and $\bu$ satisfies \eqref{frac_PAT_int} in $(0,T)\times\Omega_s$ with null Cauchy data on the piece-wise smooth boundary  consisting of $(0,T)\times(\Sigma_s^{\text{int}}\cap \Gamma)$ and $\{(t,x):x\in\Sigma_s\cap\Omega,\; 0<t<\tilde{\tau}(x)\}$, for some positive $C^2$-function $\tilde{\tau}:\Sigma_s\to(0,T]$. 
    Then, there are $\epsilon_1,\epsilon_2>0$ such that $\bu$ (and also $u$) vanishes in the set
    \[
    \{(t,x)\in \RR_+\times\Omega_s\;:\: 0<\dist(x,\Sigma_{s})<\epsilon_1,\quad 0<t<\epsilon_2\}.
    \]
\end{lemma}
\begin{remark}
    If $\epsilon_0<\min_{x\in\Sigma_s\cap\overline{\Omega}}\tilde{\tau}(x)$ the values of $\epsilon_2$ and $\epsilon_1$ depend only on the constants $\epsilon_0$ and $\kappa_0$, otherwise, they depend also on $\min_{x\in\Sigma_s\cap\overline{\Omega}}\tilde{\tau}(x)$.
\end{remark}


\begin{lemma}\label{lemma2_thm5}
    For any $s\in[s_0,s_T)$ assume that $\supp(u_0)\subset\overline{\Omega}_s$ and $u$ satisfies \eqref{frac_PAT} in $(0,T)\times\Omega_s$, 
    with null Cauchy data on $\{(t,x):x\in\Sigma_s\cap\Omega,\; 0<t<\tilde{\tau}(x)\}$ for some positive $C^2$-function $\tilde{\tau}:\Sigma_s\to(0,T]$, and on $(0,T)\times(\Sigma_s^{int}\cap \Gamma)$.
    If there are $\epsilon_1,\epsilon_2>0$  such that $u=0$ in
     \[
    \{(t,x)\in \RR_+\times\Omega_s\;:\: 0<\dist(x,\Sigma_{s})<\epsilon_1,\quad 0<t<\epsilon_2\},
    \]
    then, $u=0$ in the set
    \[
    \{(t,x)\in \RR_+\times\Omega_s\;:\: 0<\dist(x,\Sigma_{s})<\epsilon_1,\quad 0<t<\tilde{\tau}(\pi_s(x))-\dist(x,\Sigma_{s})\}.
    \]
    In particular, there is a constant $C>0$, depending on the function $\tilde{\tau}$, the sound speed $c$ and the vector field $\nu$, such that $u=0$ in the smaller set 
    \[
    \{(t,x)\in \RR_+\times\Omega_{s}\;:\: 0<\dist(x,\Sigma_{s})<\epsilon_1,\quad 0<t<\tilde{\tau}(\pi^\nu_s(x))-\dist_\nu(x,\Sigma_{s})-C (\dist_\nu(x,\Sigma_{s}))^2\}.
    \]
\end{lemma}
The proof of the second part of the previous lemma utilizes the following approximations, which roughly say that both distances, $\dist(x,\Sigma_s)$ and $\dist_\nu(x,\Sigma_s)$, are similar for $x$ in a sufficiently small vicinity of $\Sigma_s$.
\begin{lemma}\label{lemma3_thm5}
    There are constants $C_1,C_2, C_3>0$, depending on the metric $c^{-2}dx^2$, such that, for any $s\in[s_0,s_T]$ and $x\in\Omega_s\cap\{\dist(x,\Sigma_s)<\epsilon_0\}$:
    \begin{itemize}
        \item[i.] $\dist_\nu(x,\Sigma_s)\leq C_1\dist(x,\Sigma_s)$;
        \item[ii.] $|\dist_\nu(x,\Sigma_s)-\dist(x,\Sigma_s)|\leq C_2(\dist_\nu(x,\Sigma_s))^2$;
        \item[iii.] $\dist(\pi_s(x),\pi^\nu_s(x))\leq C_3(\dist_\nu(x,\Sigma_s))^2$.
    \end{itemize}
\end{lemma}

\begin{myproof}[Proof of Theorem \ref{thm:uniqueness_att}]
Pick an arbitrary $(t_0,x_0)\in \mathcal{D}_{\Sigma,T}$ and let $s\in(s_0,s_{T})$ be such that $x_0\in\Sigma_s$. 
    Let's consider an $\epsilon\leq \epsilon_0$ whose value will be chosen later. By the uniform continuity of the leaves of the foliation, for any choice of $\epsilon$ there is $\delta>0$ such that if $|s-s'|<\delta$, $\dist_H(\Sigma_s,\Sigma_{s'})<\epsilon$. 
Let $N\in \mathbb{N}$ be large enough such that $\delta = \frac{2}{N+1}(s-s_0)$ satisfies the previous (for a fixed $\epsilon$) and define $s_j = s_0+j\frac{\delta}{2}$ for $j=0,\dots,N+1$, so in particular $s_{N+1}=s$.
Based on Lemma \ref{lemma3_thm5} (part {\em i.}), by choosing $N$ even larger if necessary we make sure that for all $j=1,\dots,N+1$,
\begin{equation}\label{eq:dist_nu}
\max_{x\in\Sigma_{s_j}\cap\overline{\Omega}}\dist_\nu(x,\Sigma_{s_{j-1}})<\epsilon.
\end{equation}
Then, due to \eqref{eq:dist_nu} and the definition of $s_T$ (see the paragraph below Remark \ref{rmk:flat_coor}), we have that for all $x\in\Sigma^{\text{ext}}_{s_{N+1}}\cap\Omega$ the function
\begin{equation}\label{eq:g_tilde}
\tilde{g}(x):=(\dist_\nu(x,\Sigma_{s_{N}}))^2+\sum^{N}_{k=1}(\dist_\nu(\pi^\nu_{s_k}(x),\Sigma_{s_{k-1}}))^2,
\end{equation}
satisfies
\begin{equation}\label{ineq:tilde_g}
\tilde{g}(x)<\epsilon \dist_\nu(x,\Sigma_{s_0})< \epsilon T.
\end{equation}
Then, for any but fixed large constant $M>0$, by choosing $\epsilon$ smaller if necessary (and consequently, increasing the value of $N$),
and because $\tau(x)$ has a positive minimum value in the closure of $\Sigma^{\text{ext}}_{s_{N+1}}\cap\Omega$, we might assume without loss of generality that
\begin{equation}\label{ineq:t_0}
t_0<\tau(x_0) - M\epsilon T.
\end{equation}
We will later pick a specific constant $M$ and choose the other parameters accordingly.

Recall that $\supp(u_0)\cap\partial\Omega=\emptyset$. By finite propagation speed and shrinking $\epsilon$ if necessary, we have that $u=0$ in
\[
\{(t,x)\in \RR_+\times\Omega_{s_0}\;:\: 0<\dist(x,\Sigma_{s_0})<\epsilon,\quad 0<t<\epsilon\}.
\]
Notice that $\Omega_{s_0}=\Omega$ and $\Gamma_{s_0}=\Gamma$, thus by hypothesis, $u$ has null Cauchy data on 
\[
(0,T)\times\Gamma\supset \{(t,x):x\in\Gamma_{s_0},\;0<t<\tau(x)\}.
\]
According to Lemma \ref{lemma2_thm5} (where we take $\tilde{\tau}(x)=\tau(x)$) we have that $u=0$ in
\[
\{(t,x)\in \RR_+\times\Omega_{s_0}\;:\: 0<\dist(x,\Sigma_{s_0})<\epsilon,\quad 0<t<\tau(\pi_{s_0}(x))- \dist(x,\Sigma_{s_0})\},
\]
and in addition, due to Lemma \ref{lemma3_thm5}, $\pi_{s_0}(x)$ and $\pi_{s_0}^\nu(x)$, and $\dist(x,\Sigma_{s_0})$ and $\dist_\nu(x,\Sigma_{s_0})$ are close to each other, and in fact there is a uniform constant $C_0>0$ (depending on $\tau$, the sound speed $c$ and the vector field $\nu$) for which
\[
\tau(x)-(\tau(\pi_{s_0}(x))- \dist(x,\Sigma_{s_0}))<C_0(\dist_\nu(x,\Sigma_{s_0}))^2,
\]
since $\tau(x) = \tau(\pi_{s_0}^\nu(x)) - \dist_\nu(x,\Sigma_{s_0})$. This implies that $u=0$ in the smaller set
\[
\{(t,x)\in \RR_+\times\Omega_{s_0}\;:\: 0<\dist(x,\Sigma_{s_0})<\epsilon,\quad 0<t<\tau(x)-C_0(\dist_\nu(x,\Sigma_{s_0}))^2\}.
\]
It is clear that $\bu=\int^t_0u(s)ds$ (solution to \eqref{frac_PAT_int}) must also vanish there. Moreover, since the previous set contains
\[
\{(t,x):\RR_+\times \Sigma_{s_1}\cap\Omega:\; 0<t<\tilde{\tau}(x)\}\cup \left((0,T)\times(\Sigma^{\text{int}}_{s_1}\cap\Gamma)\right)
\]
with $s_1=s_0+\frac{\delta}{2}$ and $\tilde{\tau}(x) = \tau(x) - C_0(\dist_\nu(x,\Sigma_{s_0}))^2$ (this is because $\dist_H(\Sigma_{s_1},\Sigma_{s_0})<\epsilon$),
we deduce that $\bu$ has null Cauchy data in this surface.

In what follows, $N$ and therefore $\epsilon$ are chosen so that \eqref{ineq:t_0} is satisfied with the (so far arbitrary) constant $M$ chosen to be $C_0$. Starting now from $\Gamma_{s_1}$, we can apply Lemmas \ref{lemma1_thm5} and \ref{lemma2_thm5} to obtain that $u=0$ in 
\[
\{(t,x)\in \RR_+\times\Omega_{s_1}\;:\: 0<\dist(x,\Sigma_{s_1})<\epsilon,\quad 0<t<\tilde{\tau}(x)-C_0(\dist_\nu(x,\Sigma_{s_1}))^2\},
\]
and in particular, defining $\tilde{\tau}(x)$ in $\Gamma_{s_2}$ as 
\[
\tilde{\tau}(x) = \tau(x)-C_0(\dist_\nu(\pi^\nu_{s_1}(x),\Sigma_{s_0}))^2-C_0(\dist_\nu(x,\Sigma_{s_1}))^2 ,
\]
we obtain that $\bu$ has now null Cauchy data on 
\[
\{(t,x):\RR_+\times \Sigma_{s_2}\cap\Omega:\; 0<t<\tilde{\tau}(x)\}\cap \left((0,T)\times(\Sigma^{\text{int}}_{s_2}\cap\Gamma)\right)
\]
This allows us to continue iterating this process. We do this $N-1$ times and deduce that $u=0$ in 
\begin{equation}\label{null_set}
\bigcup_{0\leq j\leq N}\{(t,x)\in \RR_+\times\Omega_{s_j}\;:\: 0<\dist(x,\Sigma_{s_j})<\epsilon,\quad 0<t<\tilde{\tau}_j(x)\}
\end{equation}
where $\tilde{\tau}_0(x)=\tau(x)$ and 
\[
\tilde{\tau}_j(x) = \tau(x)-C_0\left((\dist(x,\Sigma_{s_{j}}))^2+\sum^{j}_{k=1}(\dist(\pi^\nu_{s_k}(x),\Sigma_{s_{k-1}}))^2\right)
\]
for $j\geq 1$.

We conclude the proof by recalling that: $x_0\in\Sigma_{s_{N+1}}=\Sigma_{s}$; the definitions of $\{s_j\}_{j=1}^{N+1}$ and $\delta$; and \eqref{eq:g_tilde}-\eqref{ineq:tilde_g}; hence, by taking $\epsilon>0$ small enough (and therefore, increasing $N$) the inequality \eqref{ineq:t_0} holds and this subsequently implies
\[
\begin{aligned}
t_0<\tau(x_0)-C_0\epsilon T <\tau(x_0)-C_0\tilde{g}(x_0)=\tilde{\tau}_{N}(x_0).
\end{aligned}
\]
Since in addition $0<\dist(x_0,\Sigma_{S_N})<\epsilon$, we conclude that $(t_0,x_0)$ belongs to the open set \eqref{null_set} and consequently $u=0$ in a neighborhood of $(t_0,x_0)$.
\end{myproof}

\begin{myproof}[Proof of lemma \ref{lemma1_thm5}]
The proof is based on the construction of a local Carleman estimates for the operator $\square_{c,\alpha}$, which hold thanks to the strict convexity of $\Sigma_s\cap \overline{\Omega}$. The compactness of the aforementioned hypersurface allows us to glue together several of those local inequalities and then conclude that $\overline{u}$ has to vanish near $\Sigma_s\cap \overline{\Omega}$, for small times. 

We begin by taking an arbitrary $y\in \Sigma_{s}\cap\overline{\Omega}$ and considering boundary normal coordinates $(x',x^n)$ centered at this point, where $x^n$ corresponds to the signed distance to $\Sigma_s$ so that $y$, $\Sigma_{s}$ and $\Sigma_{s}^{\text{int}}$ are characterized respectively by $(0,0)$, $x^n=0$ and $x^n>0$. 
We consider also the weight function
\begin{equation}\label{Carleman_weight}
\varphi_y(t,x) = (R-x^n)^2-\beta t^2 - r^2,
\end{equation}
with $R,r$ and $\beta\in(0,1)$ chosen as follows. By associating to $\varphi_y$ the set
\[
Q_s(y) = \{(t,x)\in \RR_+\times\Omega_s: \varphi_y(t,x)>0,\; |x'|<\eta\}
\]
we chose $R,r$ and $\delta>0$ so that the projection of $Q_s(y)$ to the spatial variable, namely,
\[
Q_s(y)|_{t=0}:=\{x\in \Omega_s: \varphi_y(0,x)>0,\; |x'|<\eta\}
\]
lies within the boundary normal  local chart. In particular, we need $R-r\leq \epsilon_0$ with the value of $r$ depending on the upper bound of the principal curvatures of the hypersurfaces $\Sigma_s\cap\Omega$ (see Remark \ref{rmk:flat_coor}). We also take $R>0$ and $\beta\in(0,1)$ such that 
\[
\partial Q_s(y)\cap \left((0,\infty)\times\Gamma_s\right)\subset \{(t,x): 0<t<\tilde{\tau}(x),\; x\in\Gamma_s\},
\]
this is, if 
\[
t_{max}:=\max\{t>0:\exists x\in \Gamma_s\text{ such that } (t,x)\in\partial Q_s(y)\}=\sqrt{\beta^{-1}(R^2-r^2)}
\]
we require $ t_{max}<\min_{x\in\Sigma_s\cap\overline{\Omega}}\tilde{\tau}(x)$. 
This holds for any $\beta\in(0,1)$ and
\[
R^2<r^2+\beta\min_{x\in\Sigma_s\cap\overline{\Omega}}\tilde{\tau}(x)^2,
\]
and consequently we arrive at the following condition: for any $\beta\in(0,1)$ and for $r$ as before, we take $R>0$ satisfying
\[
R<\min\left\{r+\epsilon_0,\sqrt{r^2 + \beta\min_{x\in\Sigma_s\cap\overline{\Omega}}\tilde{\tau}(x)^2}\right\}.
\]
A more restrictive one would be
\[
R<\min\left\{\sqrt{r^2+\epsilon_0^2},\sqrt{r^2 + \beta\min_{x\in\Sigma_s\cap\overline{\Omega}}\tilde{\tau}(x)^2}\right\},
\]
thus, in the case of $\epsilon_0<\min_{x\in\Sigma_s\cap\overline{\Omega}}\tilde{\tau}(x)$ there is $\beta\in(0,1)$ for which the previous inequality still holds and we take $R<r+\epsilon_0$, otherwise, we might choose any $\beta\in(0,1)$ and take $R<\sqrt{r^2 + \beta\min_{x\in\Sigma_s\cap\overline{\Omega}}\tilde{\tau}(x)^2}$.

The strict convexity of $\Sigma_s\cap\overline{\Omega}$ gives rise to the next local Carleman inequality (for more details on its construction, see \cite{AP2018} and also \cite{KT2004,BDU2007}). There is a constant $C>0$ and $\tau_0>0$ such that for any $v\in C^\infty(\RR_+\times\RR^n)$ with support contained in $\RR_+\times(B^{n-1}_\eta(0)\times\RR)$ and $v(0,x)=0$, with $\eta$ as in the definition of the set $Q_s(y)$, and any $\tau\geq \tau_0$, 
\[
\begin{aligned}
&\tau\int_{Q_s(y)|_{t=0}}e^{2\tau\varphi_y(0)}|v_t(0)|^2dx+
\tau\int_{Q_s(y)}e^{2\tau\varphi_y}(\tau^2|v|^2 + |v_t|^2+|v_x|^2_g)dxdt\\
&\leq C\int_{Q_s(y)}e^{2\tau\varphi_y}|\square_c v|^2dxdt +C\int_{\Phi_s(y)}\left(\langle X_1\nabla v,\nabla v\rangle+\langle X_2,\nabla v\rangle v + X_3|v|^2\right)dS,
\end{aligned}
\]
where $dS$ denotes the surface measure on $\Phi_s(y)=\partial Q_s(y)\cap\{\varphi_y=0\}$, and the matrix-function
$X_1(t,x)$, the vector-function $X_2(t,x)$, and the scalar-function $X_3(t,x)$ are  continuous and depend on $\varphi_y$ but not on $\tau$. Notice there are only boundary terms at $\Phi_s(y)$ due to our assumption of null Cauchy data, and the support of $v$.

The extension of the local estimate to $\square_{c,\alpha}$ follows from the next inequality proven in 
\cite{Ya2022},
\[
\|e^{\tau\varphi_y} \partial^\alpha_tv\|^2_{L^2(Q_s(y))} \leq C_{T_0,\alpha}\|e^{\tau\varphi_y}\partial_tv\|^2_{L^2(Q_s(y))},\quad C_{T_0,\alpha} = \frac{\max\{1,T^{1-\alpha}\}}{\min_{\theta\in(1,2)}\Gamma(\theta)},
\]
which is a consequence of Young's inequality, the integrability of $t^{-\alpha}$ near zero, and the fact that the weight function $\varphi$ decreases for positive time. Absorbing terms with the left hand side (taking $\tau>0$ larger if necessary) we get
\[
\begin{aligned}
&\tau\int_{Q_s(y)|_{t=0}}e^{2\tau\varphi_y(0)}|v_t(0)|^2dx+\tau\int_{Q_s(y)}e^{2\tau\varphi_y}(\tau^2|v|^2 + |v_t|^2+|v_x|^2_g)dxdt\\
&\leq C\int_{Q_s(y)}e^{2\tau\varphi_y}|\square_{c,\alpha} v|^2dxdt +C\int_{\Phi_s(y)}\left(\langle X_1\nabla v,\nabla v\rangle+\langle X_2,\nabla v\rangle v + X_3|v|^2\right)dS.
\end{aligned}
\]
By the compactness of $\Sigma_s\cap\overline{\Omega}$, we can cover it with a finite collection of sets $Q_s(y)$, each of them contained in a boundary normal local chart. We then define a weight function $\varphi$ on a neighborhood of $\Sigma_s\cap\overline{\Omega}$ by taking it to be of the form \eqref{Carleman_weight} on each of those patches, for a uniform choice of parameters $R,r$ and $\beta$. A partition of unity subordinate to the previous finite covering gives us that for any $v\in C^\infty(Q_s)$ with null Cauchy data on 
\[
\{(t,x):\RR_+\times \Sigma_{s}\cap\Omega:\; 0<t<\tilde{\tau}(x)\}\cap \left((0,T)\times(\Sigma^{\text{int}}_{s}\cap\Gamma)\right),
\]
and for sufficiently large $\tau$,
\begin{equation}\label{ineq:Carleman}
\begin{aligned}
&\tau\int_{Q_s|_{t=0}}e^{2\tau\varphi(0)}|v_t(0)|^2dx+\tau\int_{Q_s}e^{2\tau\varphi}(\tau^2|v|^2 + |v_t|^2+|v_x|^2_g)dxdt\\
&\leq C\int_{Q_s}e^{2\tau\varphi}|\square_{c,\alpha} v|^2dxdt + C\int_{\Phi_s}\left(\langle X_1\nabla v,\nabla v\rangle+\langle X_2,\nabla v\rangle v + X_3|v|^2\right)dS,
\end{aligned}
\end{equation}
where
\[
Q_s = \{(t,x)\in\RR_+\times\Omega_s:\varphi(t,x)>0\},\quad Q_s|_{t=0} = \{x\in\Omega_s:\varphi(0,x)>0\}
\]
and
\[
\Phi_s = \left(\{\varphi=0\}\cap (\RR_+\times\Omega_s)\right).
\]
By density (where we use Lemma 2.1 in \cite{StUh2013} to approximate near the corners of $Q_s$) 
the previous inequality holds for all $v\in H^2(Q_s)$ with vanishing Cauchy data on $\{(t,x):\RR_+\times \Sigma_{s}\cap\Omega:\; 0<t<\tilde{\tau}(x)\}\cup \left((0,T)\times(\Sigma^{\text{int}}_{s}\cap\Gamma)\right)$.\\

We now apply the Carleman estimates \eqref{ineq:Carleman} to $v=\overline{u}$ and use that $\varphi(t)\leq \varphi(0)$ to get
\[
\begin{aligned}
&\tau\int_{Q_s|_{\{t=0\}}}e^{2\tau\varphi(0)}|u_0|^2 dxdt +\tau^3\int_{Q_s}e^{2\tau\varphi}|\overline{u}|^2 dxdt\\
&\leq C\int_{\Phi_s}\left(\langle X_1\nabla \overline{u},\nabla \overline{u}\rangle+\langle X_2,\nabla \overline{u}\rangle v + X_3|\overline{u}|^2\right)dS,
\end{aligned}
\]
for all $\tau\gg 1$, and where the right-hand side is independent of $\tau$. Therefore, the previous inequality leads us to conclude that $\bu$ must vanish in $Q_s$, and in particular $u_0=0$ in $Q_s|_{\{t=0\}}$. \\

We end the proof by noticing the following concerning the choice of parameters $R,r$ and $\beta$. If $\epsilon_0<\min_{x\in\Sigma_s\cap\overline{\Omega}}\tilde{\tau}(x)$ we take $R=\sqrt{r^2+\epsilon_0^2/2}$, therefore, we might choose $\epsilon_1=\frac{1}{2}(\sqrt{r^2+\epsilon_0^2/2}-r)<R-r$. In the other case, this is, if $\epsilon_0\geq \min_{x\in\Sigma_s\cap\overline{\Omega}}\tilde{\tau}(x)$ we choose
\[
R=\sqrt{r^2 + \frac{\beta}{2}\min_{x\in\Sigma_s\cap\overline{\Omega}}\tilde{\tau}(x)^2},
\]
and subsequently take
\[
\epsilon_1=\frac{1}{2}\left(\sqrt{r^2+\frac{\beta}{2}\min_{x\in\Sigma_s\cap\overline{\Omega}}\tilde{\tau}(x)^2}-r\right)<R-r.
\]
With the previous choice of parameters and the result obtained from the Carleman estimate, we conclude that it is possible to find a sufficiently small $\epsilon_2$ such that 
    \[
    \{(t,x)\in \RR_+\times\Omega_s\;:\: 0<\dist(x,\Sigma_{s})<\epsilon_1,\quad 0<t<\epsilon_2\}\subset Q_s.
    \]
where $\bar{u}=0$.
\end{myproof}

\begin{myproof}[Proof of lemma \ref{lemma2_thm5}]
As in the previous lemma, we let $x^n$ be the signed distance to $\Sigma_s$. 
We define the function $\psi$ in  boundary normal local coordinates as
\[
\psi(t,x) = (\epsilon_1-x^n)(\tilde{\tau}(x',0)-t-x^n).
\]
Its levels sets $\{\psi(t,x)=\gamma\}_{\gamma>0}$ exhaust the region 
\[
\mathcal{Q}=\{(t,x)\in\RR_+\times\Omega_s:0<x^n<\epsilon_1,\; 0<t<\tilde{\tau}(\pi_s(x))-x^n\},
\]
where we recall that $\pi_s(x)$ stands for the projection of $x\in \Omega_s\cap\{0<x^n<\epsilon_1\}$ onto $\Sigma_s$. It is clear that for some $\gamma_0>0$, 
\[
\{\psi(t,x)>\gamma_0\}\subset \{(t,x):0<x^n<\epsilon_1,\; 0<t<\epsilon_2\}
\]
thus, by Lemma \ref{lemma1_thm5} we have $u=0$ in this super-level set. Let 
\[
\gamma_1 = \inf\{\gamma>0:u=0\text{ in }\{\psi>\gamma\}\}.
\]
If $\gamma_1=0$ we obtain the first assertion of the lemma, otherwise, we take an arbitrary point $(t_0,x_0)\in \{\psi=\gamma_1\}$ and our task is to show that $u$ has to vanish in a neighborhood of this point, consequently, in a neighborhood of the whole level set $\{\psi=\gamma_1\}\cap \overline{\mathcal{Q}}$ contradicting the definition of $\gamma_1$ and concluding that $\gamma_1=0$.

Recall we are working within 
a boundary normal local chart (with respect to $\Sigma_s$) containing $(t_0,x_0)\in \{\psi=\gamma_1\}$. In those coordinates, we define the weight function
\[
\varphi(t,x) = \psi(t,x)- T_1[\psi](t_0,x',x^n_0) - \frac{\delta}{2}|x-x_0|^2,
\]
for some $\delta>0$ to be chosen later, and where $T_1[\psi](t,x',x^n)$ stands for the first order Taylor expansion of $\psi$ around $(t_0,x_0)$, thus
\[
T_1[\psi](t_0,x',x^n_0)=\psi(t_0,x_0) + \sum^{n-1}_{j=1}\partial_{x^j}\psi(t_0,x_0)(x^j-x_0^j).
\]
Our main objective is to show that $u$ vanishes in a set of the form
$\{\varphi>\eta\}\cap\{\psi\leq \psi(t_0,x_0)\}$ which, for $\eta<0$ sufficiently small in magnitude, contains $(t_0,x_0)$ and lies within a small neighborhood of $(t_0,x_0)$. We can see this last statement by noticing that there are constants $c_1,c_2>0$ such that for any  $(\theta,\xi)\in(\psi'(t_0,x_0))^{\perp}$, $|\theta|\leq c_1|\xi|$ and
\[
\langle (\psi-\varphi)''(\theta,\xi),(\theta,\xi)\rangle = \delta|\xi|^2>c_2|(\theta,\xi)|^2.
\]
To show unique continuation, we intend to use a Carleman estimate requiring the strong pseudoconvexity condition in the subset $\{(t,x,0,\xi):t>0,x\in\RR^n,\xi\in\RR^n\}\subset T^*\RR^{n+1}$. For the precise definition of pseudoconvexity in a subset of the cotangent bundle we refer the reader to \cite{Ta1999}. 
In fact, we first notice that at $(t_0,x_0)$, 
\[
\varphi_{x'}(t_0,x_0)=0,\quad \varphi_{x^n}(t_0,x_0)=-a,\quad\varphi_t(t_0,x_0)=-b\quad\text{with}\quad a>b>0,
\]
for $1\leq i,j\leq n-1$,
\[
\varphi_{x^ix^j}(t_0,x_0)=-\delta\delta_{ij},\quad \varphi_{x^ix^n}(t_0,x_0) = -\partial_{x^i}\tilde{\tau}(x'_0),\quad \varphi_{tx^i}(t_0,x_0) = 0,
\]
and
\[
\varphi_{x^nx^n}(t_0,x_0) = 2-\delta,\quad \varphi_{tx^n}(t_0,x_0) = 1.
\]
Then, and since the principal symbol of $\square_c$, $p(t,x,\theta,\xi)  = |\xi|_g^2 - \theta^2$, is elliptic in $\{(t,x,0,\xi):t>0,x\in\RR^n,\xi\in\RR^n\}$, we only need to verify the inequality $\frac{1}{i\tau}\{\overline{p}_\varphi,p_\varphi\}(t_0,x_0,0,\xi)>0$ for all $\xi\neq 0$ such that $p_\varphi(t_0,x_0,0,\xi)=0$, where $\{\cdot,\cdot\}$ stands for the Poisson bracket and
\[
p_\varphi(t,x,\theta,\xi):=p(t,x,\theta+i\tau\varphi_t,\xi+i\tau\varphi_x).
\]
Notice that $p_\varphi(t_0,x_0,0,\xi)=0$ implies $\xi_n = 0$ and $|\xi|^2_g=\tau^2(a^2-b^2)$. Computing the bracket and evaluating at $(t_0,x_0,0,\xi)$ leads to
\[
\begin{aligned}
&\frac{1}{i\tau}\{\overline{p}_\varphi,p_\varphi\}(t_0,x_0,0,\xi) = \frac{2}{\tau}\{\mathfrak{Re}p_\varphi,\mathfrak{Im}p_\varphi\}\\
&=8\tau^2\left(a^2(2-\delta) - b^2 - ab + \frac{a}{4\tau^2}\sum^{n-1}_{i,j}(\partial_{x^n}g^{ij})\xi_i\xi_j+\frac{\delta}{\tau^2}\sum^{n-1}_{k=1}\Big(\sum^{n-1}_{i=1}g^{ik}\xi_i\Big)^2 \right)\\
&\geq 16\tau^2(a^2-ab + a(a^2-b^2)\frac{\underline{\kappa}}{8} - \delta M) 
\end{aligned}
\]
for some 
\[
M>\frac{a^2}{2}+\frac{1}{2\tau^2}\sum^{n-1}_{k=1}\Big(\sum^{n-1}_{i=1}g^{ik}\xi_i\Big)^2,
\]
and with $\underline{\kappa}$ the constant from the uniform strict convexity condition for the foliation. The desired inequality is obtained by taking $\delta>0$ small enough.

Because $\varphi$ is strongly pseudoconvex in the set $\{(t,x,0,\xi):t>0,x\in\RR^n,\xi\in\RR^n\}$ with respect to $\square_c$ at $(t_0,x_0)$, there exist constants $\zeta,d,C,\tau_0>0$ such that for each $\epsilon>0$ and $\tau\geq \tau_0$ the following Carleman estimate holds (see \cite{Ta1999})
\[
\tau^{-1}\|e^{-\frac{\epsilon}{2\tau}|D_t|^2}e^{\tau\varphi}v\|^2_{(2,\tau)}\leq C\left(\|e^{-\frac{\epsilon}{2\tau}|D_t|^2}e^{\tau\varphi}\square_cv\|^2 +e^{-d\epsilon\tau}\|e^{\tau\varphi}\square_cv\|^2 + e^{-d\epsilon\tau}\|e^{\tau\varphi}v\|^2_{(1,\tau)}\right)
\]
for every $v\in C^\infty_c(B(t_0,x_0;\zeta))$ with $B(t_0,x_0;\zeta)\subset\RR^{n+1}$ the ball of radius $\zeta$ and center $(t_0,x_0)$.

The extension of the previous inequality to $\square_{c,\alpha}$ is obtained in a similar fashion as in \cite{BDU2007}, where the authors treat the case of a time-convolution perturbation of the (Euclidean) wave operator with a regular kernel. The main ingredients needed to obtain the extension are the time-convolution character of the lower order term (i.e. $a\partial^\alpha_t v$) and the linearity in time of the weight function $\varphi$. These conditions imply a suitable commutation identity that we explain next (see \eqref{eq:commutation}). 

Denoting $k(t)=\frac{H(t)}{\Gamma(1-\alpha)t^\alpha}$, for a compactly supported function $v(t)$ with support in $\RR_+$ we might write $\partial^\alpha_tv = k*v_t$ and we have
\[
e^{\tau\varphi}(k*v_t) = (e^{\tau\varphi_tt}k)*(e^{\tau\varphi}v)_t - \tau \varphi_t(e^{\tau\varphi_tt}k)*(e^{\tau\varphi}v)
\]
since $\varphi(t,x)-\varphi(s,x) = \varphi_t(t,x)(t-s)$ with $\varphi_t$ a negative constant. Writing $E=e^{-\frac{\epsilon}{2\tau}|D_t|^2}$, which we can also recast as the convolution operator
\[
Ef(t) = \left(\frac{\tau}{2\pi\epsilon}\right)^{1/2}\int e^{-\frac{\tau|t-s|^2}{2\epsilon}}f(s)ds,
\]
we see that
\[
Ee^{\tau\varphi}(k*v_t) = E(e^{\tau\varphi_tt}k)*(e^{\tau\varphi}v)_t - \tau \varphi_tE(e^{\tau\varphi_tt}k)*(e^{\tau\varphi}v).
\]
Then, $E$ and $(e^{\tau\varphi_tt}k)*$ commute thus
\begin{equation}\label{eq:commutation}
\begin{aligned}
Ee^{\tau\varphi}k*v_t&= (e^{\tau\varphi_tt}k)*E(e^{\tau\varphi}v)_t - \tau \varphi_t(e^{\tau\varphi_tt}k)*E(e^{\tau\varphi}v)\\
&= (e^{\tau\varphi_tt}k)*(Ee^{\tau\varphi}v)_t - \tau \varphi_t(e^{\tau\varphi_tt}k)*E(e^{\tau\varphi}v),
\end{aligned}
\end{equation}
where we have also used the commutation $E\partial_t = \partial_tE$ for compactly supported $f$ (which follows from integration by parts).

Applying Young's convolution inequality and using the integrability of the kernel $e^{\tau\varphi_tt}k$ (where recall $\varphi_t<0$ in the support of $v$) we deduce the estimate
\[
\|e^{-\frac{\epsilon}{2\tau}|D_t|^2}e^{\tau\varphi}\partial^\alpha_tv\|^2\leq C\|e^{-\frac{\epsilon}{2\tau}|D_t|^2}e^{\tau\varphi}v\|_{(1,\tau)}^2.
\]
In consequence, and for different constants $C$ on each line, the following estimates hold for all $v\in C^\infty_c(B(t_0,x_0;\zeta))$,
\[
\begin{aligned}
&\tau^{-1}\|e^{-\frac{\epsilon}{2\tau}|D_t|^2}e^{\tau\varphi}v\|^2_{(2,\tau)}\\
&\leq C\left(\|e^{-\frac{\epsilon}{2\tau}|D_t|^2}e^{\tau\varphi}\square_{c,\alpha}v\|^2 + \|e^{-\frac{\epsilon}{2\tau}|D_t|^2}e^{\tau\varphi}\partial^\alpha_tv\|^2+e^{-d\epsilon\tau}\|e^{\tau\varphi}\square_cv\|^2 + e^{-d\epsilon\tau}\|e^{\tau\varphi}v\|^2_{(1,\tau)}\right)\\
&\leq C\left(\|e^{-\frac{\epsilon}{2\tau}|D_t|^2}e^{\tau\varphi}\square_{c,\alpha}v\|^2 + \|e^{-\frac{\epsilon}{2\tau}|D_t|^2}e^{\tau\varphi}v\|_{(1,\tau)}^2+e^{-d\epsilon\tau}\|e^{\tau\varphi}\square_cv\|^2 + e^{-d\epsilon\tau}\|e^{\tau\varphi}v\|^2_{(1,\tau)}\right)\\
&\leq C\left(\|e^{-\frac{\epsilon}{2\tau}|D_t|^2}e^{\tau\varphi}\square_{c,\alpha}v\|^2 +e^{-d\epsilon\tau}\|e^{\tau\varphi}\square_cv\|^2 + e^{-d\epsilon\tau}\|e^{\tau\varphi}v\|^2_{(1,\tau)}\right).
\end{aligned}
\]
The last inequality follows after absorbing terms  with the LHS by choosing $\tau\geq \tau_0>0$ larger if necessary.

By taking $v=\chi u$ with $\chi$ a smooth cut-off with support contained in a sufficiently small neighborhood of $(t_0,x_0)$, and noticing that $\supp(v)\subset\{\varphi(t,x)>\ell_1\}$ for some $\ell_1<0$, we get
\[
\|e^{-\frac{\epsilon}{2\tau}|D_t|^2}e^{\tau\varphi}\chi u\|^2\leq Ce^{\ell_2\epsilon\tau}
\]
for some other constants $\ell_2<0$ and $C>0$ depending on $\|\chi u\|_{H^1(Q_0)}$.

The conclusion that $u=0$ in a neighborhood of $(t_0,x_0)$ is a consequence of the modified Carleman method introduced by Tataru in \cite{Ta1994} (see also \cite{LernerBook}).
\end{myproof}

\begin{myproof}[Proof of lemma \ref{lemma3_thm5}]
{\em i.} Let $x_0\in\Omega_s\cap\{\dist(x,\Sigma_s)<\epsilon_0\}$ be arbitrary. In a flat chart with coordinates $x=(x',x^n)$, there is $c_0 \in\RR$ such that $\Sigma_s$ is characterized by $x^n=c_0$ and $\Sigma_s^{\text{int}}$ by $x^n>c_0$. Moreover, the vector field $\nu$ takes the form $\nu(x)=(0,\nu^n(x))$. 

Let $y_0 = (x'_0,c_0)$ and consider the curve $\gamma_\nu(r) = y_0+r(x_0-y_0)$, which up to re-parametrization is an integral curve for the vector field $\nu$ connecting $y_0$ and $x_0$. In particular, since the length of a curve is invariant under change of parametrization, we have that 
\begin{equation}\label{ineq:dist_nu}
\begin{aligned}
\dist_\nu(x_0,\Sigma_s)&=L_g(\gamma_\nu(r)|_{r\in[0,1]})\\
&= \displaystyle\int^1_0|\dot{\gamma}_\nu(r)|_gdr\\
&\leq C\int^1_0|\dot{\gamma}_\nu(r)|dr\leq C|x_0-y_0|
\end{aligned}
\end{equation}
for some constant $C$ depending on the local chart, and where $|\cdot|_g$ and $|\cdot|$ stand respectively for the norm associated to the sound speed and the euclidean metric.

Let $z_0=\pi_s(x_0)\in\Sigma_s$ the projection of $x_0$ onto $\Sigma_s$ and denote by $\gamma:[0,\theta_0]\to\RR^n$ the unit speed geodesic connecting $x_0=\gamma(\theta_0)$ and $z_0=\gamma(0)$, therefore, $\theta_0 = \dist(x_0,z_0)=\dist(x_0,\Sigma_s)$. Taylor expanding $\gamma(\theta)$ around $z_0$ and evaluating at $\theta_0$ gives
\[
x_0 = z_0 + \dist(x_0,z_0)\dot{\gamma}(0) + O(\dist(x_0,z_0)^2).
\]
Denoting $\hat{e}_n=(0,\dots,0,1)$, we have
\[
y_0+|x_0-y_0|\hat{e}_n = z_0 + \dist(x_0,z_0)\dot{\gamma}(0) + O(\dist(x_0,z_0)^2),
\]
but since $y_0$ and $z_0$ both belong to $\Sigma_s$, restricting our attention to the $n$-th coordinate we see that
\[
|x_0-y_0| = \dist(x_0,z_0)(\dot{\gamma}(0)\cdot\hat{e}_n) + O(\dist(x_0,z_0)^2),
\]
therefore
\[
|x_0-y_0|\leq C\dist(x_0,z_0)
\]
for some positive constant $C$ depending on $y_0$. Bringing together this inequality and \eqref{ineq:dist_nu} gives
\[
\dist_\nu(x_0,\Sigma_s)\leq C\dist(x_0,\Sigma_s)\leq C\epsilon_0.
\]
Finally, the compactness of $\overline{D_{\Sigma,T}}$ allows us to choose a uniform constant $C$ such that the previous holds for any $s\in[s_0,s_T]$ and $x_0\in\Omega_s\cap\{\dist(x,\Sigma_s)<\epsilon_0\}$.\\

\noindent{\em ii. and iii.} Take an arbitrary $x\in\Omega_s\cap\{\dist(x,\Sigma_s)<\epsilon_0\}$. We will now consider boundary normal coordinates with respect to $\Sigma_s$.

Let $y_0\in\Sigma_s$ be the projections of $x_0$ along the integral curves of $\nu$ to $\Sigma_s$, and denote by $\gamma_\nu:[0,\theta_0]\to\RR^n$ the integral curve of the unitary vector field $\nu$ such that $\gamma_\nu(0)=y_0\in \Sigma_s$ and $\gamma_\nu(\theta_0)=x_0$. Then, $\theta_0=\dist_\nu(x_0,\Sigma_s)$.

On the other hand, we denote by $z_0$ the closest point to $x_0$ in $\Sigma_s$, thus, they are connected by a geodesic which in the boundary normal coordinates takes the form $\gamma(\theta) = z_0+\theta \hat{e}_n$.

Applying a Taylor expansion to $\gamma_\nu(\theta)$ we get
\[
x_0 = y_0 + \theta_0\hat{e}_n + O(\theta_0^2),
\]
from where we deduce
\[
z_0+\dist(x_0,\Sigma_s)\hat{e}_n=y_0 + \dist_\nu(x_0,\Sigma_s)\hat{e}_n + O(\dist_\nu(x_0,\Sigma_s)^2).
\]
In particular, the $n$-th coordinate gives
\[
\dist(x_0,\Sigma_s)=\dist_\nu(x_0,\Sigma_s) + O(\dist_\nu(x_0,\Sigma_s)^2),
\]
from where a local version of $ii$ follows, while the equality in the other coordinates implies
\[
|\pi_s(x_0)-\pi^\nu_s(x_0)| = O(\dist_\nu(x_0,\Sigma_s)^2),
\]
and subsequently, a local version of $iii.$ 

We conclude the proof by invoking once again the compactness of $\overline{D_{\Sigma,T}}$, which allows us to choose uniform constants such that $ii$ and $iii$ hold for any $s\in[s_0,s_T]$ and $x_0\in\Omega_s\cap\{\dist(x,\Sigma_s)<\epsilon_0\}$.\\
\end{myproof}

\section{Stability}

For a fixed compact subdomain $\mathcal{K}\subset\Omega$ with smooth boundary, we consider initial conditions in
\[
H^1_0(\mathcal{K}):=\{f\in H^1(\RR^n):\supp(f)\subseteq\mathcal{K}\}.
\]
The technical result upon which the main part of the stability theorem follows consists in showing that for $v$ solution to the unattenuated system
 \begin{equation}\label{unatt_PAT}
\left\{\begin{array}{l}
\partial_t^2v - c^2\Delta v = 0,\\
(v,\partial_tv)|_{t=0} = (v_0,0),
\end{array}\right.
\end{equation}
the fractional derivative $ \partial^\alpha_tv$ depends continuously on $v_0$, in terms of its $H^\alpha(\RR^n)$-norm. This is the content of Theorem \ref{thm:reg_frac} below.
The stability of the inverse problem, stated in Theorem \ref{thm:stability}, is then a consequence of the well-posedness of the fractional forward problem and a standard compactness-uniqueness argument.\\

Below, we denote by $C^{1-\alpha}([0,T];L^2(\RR^n))$  the Banach space of all $(1-\alpha)-$H\"older continuous functions with respect to time, with values in $L^2(\RR^n)$, and equipped with the norm
    \[
    \|v\|_{C^{1-\alpha}([0,T];L^2(\RR^n))}:=\sup_{t\in[0,T]}\|v(\cdot,t)\|_{L^2(\RR^n)} + \sup_{t,\tau\in[0,T],t\neq\tau}\frac{\|v(\cdot,t)-v(\cdot,\tau)\|_{L^2(\RR^n)}}{|t-\tau|^{1-\alpha}}.
    \]

\begin{theorem}\label{thm:reg_frac}
    Let $v$ be the solution to \eqref{unatt_PAT} for an initial condition $v_0\in H^1_0(\mathcal{K})$. Then, 
    \[
    \partial^\alpha_tv\in C^{1-\alpha}([0,T];L^2(\RR^n))\cap C((0,T];H^{1-\alpha}(\RR^n))
    \cap L^1((0,T);H^{1-\alpha}(\RR^n)),
    \]
    and for each $t\in[0,T]$, $\partial^\alpha_tv(t)$ has compact support in $\RR^n$. Moreover, for each $\delta\in(0,T)$ there is a constant $C>0$ such that
    \[
    \|\partial^\alpha_tv\|_{L^1((0,T);L^2(\RR^n))}+\|\partial^\alpha_tv\|_{C((\delta,T);L^2(\RR^n))}\leq C\|v_0\|_{H^\alpha(\RR^n)},\quad \forall v_0\in H^1_0(\mathcal{K}),
    \]
    and
    \[
    \|\partial^\alpha_tv\|_{L^1((0,T);H^{1-\alpha}(\RR^n))}+\|\partial^\alpha_tv\|_{C((\delta,T);H^{1-\alpha}(\RR^n))}\leq C\|v_0\|_{H^1(\RR^n)},\quad \forall v_0\in H^1_0(\mathcal{K}).
    \]
\end{theorem}

\begin{proof}
The finite propagation speed property of the wave equation implies the compactness of support of $\partial_t^\alpha v(t)$. Moreover, since $v_0\in H^1_0(\mathcal{K})$, standard regularity properties lead to $\partial_tv\in C([0,T];L^2(\RR^n))$, from where we can deduce that 
$\partial^\alpha_tv\in C^{1-\alpha}([0,T];L^2(\RR^n))$. In fact, since we can write $\partial^\alpha_tv = \cI_{1-\alpha}(\partial_tv)$ with the fractional integral operator $\cI_{1-\alpha}$ defined below, we apply the next property of fractional integrals to $\partial_tv$.
\begin{lemma}\label{lemma:cont_frac_int}
For $\alpha\in(0,1)$ the operator
\[
\cI_{1-\alpha}u(x,t) := \frac{1}{\Gamma(1-\alpha)}\int^t_0\frac{u(x,s)}{(t-s)^{\alpha}}ds
\]
is a continuous map from $C([0,T];L^2(\RR^n))$ into $C^{1-\alpha}([0,T];L^2(\RR^n))$. In addition, for any $r\in \RR$, it is a continuous map from $L^1([0,T];H^r(\RR^n))$ into itself.
\end{lemma}

The rest of the proof, namely, the regularity in $C((0,T];H^{1-\alpha}(\RR^n))$ and $L^1((0,T);H^{1-\alpha}(\RR^n))$, and the continuous dependence estimate, follows directly from the next lemma in the specific cases $m=\alpha,1$.
\begin{lemma}\label{lemma:comp_FIO_Caputo}
    For each $\delta\in(0,T)$ and real exponent $m\geq 0$ there exists a constant $C(\delta,m)>0$ such that
    \[
    \|\partial^\alpha_tv\|_{L^1((0,T);H^{m-\alpha}(\RR^n))}+\|\partial^\alpha_tv\|_{C((\delta,T];H^{m-\alpha}(\RR^n))}\leq C\|v_0\|_{H^m(\RR^n)},\quad \forall v_0\in H^1_0(\mathcal{K}),
    \]
    where $v$ is the unique solution to \eqref{unatt_PAT}.
\end{lemma}

\end{proof}

\begin{myproof}[Proof of Lemma \ref{lemma:cont_frac_int}]
1. According to Minkowski's integral inequality we obtain
\[
\begin{aligned}
\|\cI_{1-\alpha}u(\cdot,t)\|_{L^2}&\leq \frac{1}{\Gamma(1-\alpha)}\int^t_0\frac{\|u(\cdot,s)\|_{L^2}}{(t-s)^{\alpha}}ds\\
&\leq \left(\int^t_0\frac{(t-s)^{-\alpha}}{\Gamma(1-\alpha)}dt\right)\|u\|_{C^0(L^2)}\\
&\leq \frac{t^{1-\alpha}}{\Gamma(2-\alpha)}\|u\|_{C^0(L^2)}
\end{aligned},
\]
therefore
\[
\sup_{t\in[0,T]}\|\cI_{1-\alpha}u(t)\|_{L^2}\leq \frac{T^{1-\alpha}}{\Gamma(2-\alpha)}\|u\|_{C^0(L^2)}.
\]
 Similarly, noticing first that for $t>\tau$ and for almost every $x\in \RR^n$,
\[
\begin{aligned}
&|\cI_{1-\alpha}u(x,t)-\cI_{1-\alpha}u(x,\tau)|\\
&\leq \int^t_\tau\frac{(t-s)^{-\alpha}}{\Gamma(1-\alpha)}|u(x,s)|ds + \int^\tau_0\frac{(t-s)^{-\alpha} - (\tau-s)^{-\alpha}}{\Gamma(1-\alpha)}|u(x,s)|ds,
\end{aligned}
\]
then, by Minkowski's integral inequality and the triangle inequality, we get 
\[
\begin{aligned}
&\|\cI_{1-\alpha}u(\cdot,t)-\cI_{1-\alpha}u(\cdot,\tau)\|_{L^2}\\
&\leq \|u\|_{C^0(L^2)}\left(\int^t_\tau\frac{(t-s)^{-\alpha}}{\Gamma(1-\alpha)}ds + \int^\tau_0\frac{|(t-s)^{-\alpha} - (\tau-s)^{-\alpha}|}{\Gamma(1-\alpha)}ds\right),
\end{aligned}
\]
and subsequently,
\[
\begin{aligned}
&\|\cI_{1-\alpha}u(\cdot,t)-\cI_{1-\alpha}u(\cdot,\tau)\|_{L^2}^2\\
&\leq \frac{2\|u\|^2_{C^0(L^2)}}{\Gamma(2-\alpha)^2}\left((t-\tau)^{2(1-\alpha)} + \left((t-\tau)^{(1-\alpha)}+\tau^{1-\alpha} - t^{1-\alpha} \right)^2\right)\\
&\leq \frac{4(t-\tau)^{2(1-\alpha)}}{\Gamma(2-\alpha)^2}\|u\|^2_{C^0(L^2)}.
\end{aligned}
\]
Taking supremum with respect to $t,\tau\in[0,T]$, with $t\neq \tau$, we deduce there is a constant $C_{T,\alpha>0}$ such that
\[
\|\cI_{1-\alpha}u\|_{C^{1-\alpha}(L^2)}\leq C_{T,\alpha}\|u\|_{C^0(L^2)}.
\]

2. Applying Fubini to pass the Fourier Transform inside the fractional integral, and also Minkowski's integral inequality and Youngs's convolution inequality, we obtain
\[
\begin{aligned}
\|\cI_{1-\alpha}u\|_{L^1(H^r)}& = \int^T_0\left(\int_{\RR^n}(1+|\xi|^2)^{r}\left(\frac{1}{\Gamma(1-\alpha)}\int^t_0\frac{\hat{u}(\xi,s)}{(t-s)^{1-\alpha}}ds\right)^2d\xi\right)^{1/2}dt\\
&\leq \frac{1}{\Gamma(1-\alpha)}\int^T_0\int^t_0\frac{\|u(\cdot,s)\|_{H^r}}{(t-s)^{\alpha}}dsdt\\
&\leq \left(\sup_{t\in[0,T]}\int^t_0\frac{(t-s)^{-\alpha}}{\Gamma(1-\alpha)}ds\right)\|u\|_{L^1(H^r)}\\
&\leq \frac{T^{1-\alpha}}{\Gamma(2-\alpha)}\|u\|_{L^1(H^r)}.
\end{aligned}
\]
\end{myproof}

\begin{myproof}[Proof of Lemma \ref{lemma:comp_FIO_Caputo}] 
Given the initial data $v_0$ in \eqref{unatt_PAT}, it is well known that $v$ can be represented
by a Fourier Integral Operator (FIO) acting on $v_0$, this is, up to an smooth function we can write $v$ as a finite sum of Fourier integrals of the form
\[
\frac{1}{(2\pi)^n}\int e^{i\phi(t,x,\xi)}A(t,x,\xi)\hat{v}_0(\xi)d\xi,
\]
with $\phi$ a phase function which is smooth and positively homogeneous of order one in the Fourier variable $\xi$, and $A$ a classical symbol of order zero, this is, it has an asymptotic expansion $\sum_{j\geq 0}A_j(t,x,\xi)$ with each $A_j(t,x,\xi)\in S^{-j}(\RR\times\RR^n\times \RR^n)$ positively homogeneous of order $-j$ in $\xi$. We refer the reader to \cite{TrevesII,GS1994} for more details on the geometric optics construction that leads to this representation. 

Since we are interested in local regularity, we will work near some arbitrary $(t_1,x_1)\in(0,\infty)\times\RR^n$ and microlocalize ---i.e., we localize in phase-space--- around a null-bicharacteristic for the symbol $p(t,x,\tau,\xi) = c^2(x)|\xi|^2-\tau^2$ and passing through $(t_1,x_1,\tau_1,\xi_1)\in T^*\RR^{n+1}$ satisfying $p(t_1,x_1,\tau_1,\xi_1)=0$. We do this last part by means of Egorov's theorem and
replacing $v_0$ with a microlocalization of it, this is, we apply a zero-order pseudodifferential operator with symbol $\chi(x,\xi)$, homogeneous of order zero in $\xi$ and  support contained in a conic neighborhood of some $(x_0,\xi_0)\in T^*\RR^{n}$. This phase-space point is such that $(0,x_0,\tau_0,\xi_0)\in p^{-1}(\{0\})$ generates a null-bicharacteristic passing through $(t_1,x_1,\tau_1,\xi_1)$. In the following, we assume that such microlocalization has already been performed and we simply assume that the wavefront set of $v_0$ is contained in a conic neighborhood of $(x_0,\xi_0)\in T^*\RR^{n}$. Then, 
\begin{equation}\label{eq:fourier_int}
v(t,x) = \sum_{\pm}\frac{1}{(2\pi)^n}\int e^{i\phi_{\pm}(t,x,\xi)}A^{\pm}(t,x,\xi)\hat{v}_0(\xi)d\xi
\end{equation}
is a microlocal representation of the actual solution in the vicinity of the null-bicharacteristic emanating from $(0,x_0;\tau_0,\xi_0)\in T^*\RR^{n+1}$, near the point $(t_1,x_1)$. The equality is therefore understood modulo a smooth function. 

It is known that for $v_0\in H^m_{comp}(\RR^n)$ ---the space of compactly supported $H^m$-functions--- $v\in C([0,T];H^m_{comp}(\RR^n))\cap C^1([0,T];H^{m-1}_{comp}(\RR^n))$, thus a priori we have 
\[
\partial^\alpha_tv \in C([0,T];H^{m-1}_{comp}(\RR^n)).
\]

We denote by $v_+$ and $v_-$ the respective terms in \eqref{eq:fourier_int}. Their phase functions $\phi_\pm$ are associated with the positive and negative sound speed in the sense they respectively solve the eikonal equations
\[
\partial_t\phi_\pm = \pm c(x)|\partial_x\phi_\pm|.
\]
In general, these equations can not be solved globally in the whole interval $[0,T]$, however,  the Cauchy problem is locally solvable for any initial time (see for instance \cite{TrevesII}), hence by compactness of $[0,T]$ we only require to do so a finite number of times, this is, by solving Cauchy problems with initial conditions at times $0=t_0<t_1<t_2<...<t_l<T$ for some $l\geq 1$. At $t=0$ we set
\[
\phi_\pm=x\cdot\xi,\quad A_{0}^\pm = 1/2,\quad A_{j}^\pm=0 \quad\text{for}\quad j\geq 1.
\]
and find $\phi$ and $A$ for a small interval of time containing $t_1$. The restriction of these functions to $t_1$ gives us initial conditions for the next Cauchy problem for $\phi$ and $A$, which we solve in a small interval of time containing $t_2$. We repeat this process until we construct the parametrix in a neighborhood of a null-bicharacteristic up to $t=T$.

The intervals where the eikonal equation is solvable form a finite covering of $[0,T]$ and we complete it to a covering of all $\RR$. Recalling Remark \ref{rmk:partition_unity}, without loss of generality, we take a partition of unity $\{h_j\}_{j\geq 0}$ subordinate to it and further assume that $h_0$ is the only one with support containing a neighborhood of $t=0$. We will use this partition of unity to apply the convolution in Definition \ref{def:conv}, and in light of Remark \ref{rmk:Caputo_der} and our choice of partition of unity, we will treat the cases $j=0$ and $j\geq 1$ separately.

In addition, we will perform the computations solely on the component $v_+$ in \eqref{eq:fourier_int} since for the other one is analogous. That said, we will omit the sub- and super-script $``+"$ from now on. 

We are interested only in the leading order terms, or in other words, the most singular component of $\partial^\alpha_t v$. We will then work with an oscillatory integral representation and keep track of the terms having the highest asymptotic behavior in $\xi$ (the Fourier variable). We will see that the leading order is $|\xi|^\alpha$ (implying a loss of $\alpha$-derivatives in space), thus for an initial source $v_0\in H^m(\RR^n)$, $m\geq 0$, the leading order term will have $H^{m-\alpha}$-regularity in space.

{\em 1.} Let's assume $j\geq 1$. Recalling the Fourier integral representation of the fractional derivative in \eqref{eq:Fourier_Caputo_a}-\eqref{eq:Fourier_Caputo_b} and Remarks \ref{rmk:partition_unity} and \ref{rmk:Caputo_der}, if $v_0(x)\in H^m(\RR^n)$ for some $m\geq 0$, we have
\[
t\mapsto \frac{1}{2\pi}\int\int^\infty_0 e^{i(t-s)\omega}(i\omega)^\alpha h_j(s)v_0(x)dsd\omega \;\in\; C^\infty([0,\infty);H^m(\RR^n)),
\]
thus, modulo a $C^\infty([0,\infty);H^m(\RR^n))$-function,
\[
(\partial^\alpha_tv)_{j} = \frac{1}{2\pi}\int\int e^{i(t-s)\omega}(i\omega)^\alpha h_j(s)v(x,s)dsd\omega.
\]
Notice that, due to $h_j$ having support away from the origin, we are allowed to treat the integral with respect to $s$ as if it were on the whole real line.

We also do not lose generality by assuming $v_0\in C^\infty_c(\RR^n)$ since otherwise the next computations hold for $\hat{v}_0(\xi)$ replaced by $\tilde{\chi}(\epsilon\xi)\hat{v}_0(\xi)$, with $\tilde{\chi}\in\mathcal{S}(\RR^n)$ such that $\tilde{\chi}(0)=1$
and taking the limit as $\epsilon\to 0$ in the sense of distributions.

By using the local representation of $v$ as a Fourier integral we obtain
\[
\begin{aligned}
(\partial^\alpha_tv)_{j}
&= \frac{1}{(2\pi)^{n+1}}\int\int \int e^{i(t-s)\omega}(i\omega)^\alpha h_j(s)e^{i\phi(s,x,\xi)}A(s,x,\xi)\hat{v}_0(\xi)d\xi dsd\omega
\end{aligned}
\]
modulo a $C^\infty([0,\infty); H^m(\RR^n))$-function, where the integrals for the measures $dsd\omega$ and $d\xi$ are absolutely convergent (since $j\geq 1$). 
Applying Fubini to the previous equality we get
\[
(\partial^\alpha_tv)_{j} = \frac{1}{(2\pi)^{n+1}}\int \cI_j(t,x,\xi)\hat{v}_0(\xi)d\xi 
\]
with 
\[
\begin{aligned}
\cI_j(t,x,\xi) = \int\int e^{i(t-s)\omega+i\phi(s,x,\xi)}(i\omega)^\alpha h_j(s)A(s,x,\xi) dsd\omega.
\end{aligned}
\]
Let's set $\lambda=|\xi|$ and use the homogeneity of $\phi$ with respect to $\xi$ to write $\psi(t,x,\xi)=\phi(s,x,\hat{\xi})$, where $\hat{\xi} = \xi/|\xi|$, hence $\phi(s,x,\xi) = \lambda\psi(s,x,\xi)$. As it is customary when dealing with oscillatory integrals, we may assume the integration in $\xi$ happens away from the origin since we can always cut off a neighborhood of it and this only brings a smooth error. Finally, to alleviate notation we will occasionally omit the dependence on $x$ and $\xi$, which will be assumed fixed from now on, and we will write $g'(s)=\partial_sg(s)$ to denote differentiation with respect to the $s$ variable. \\

{\em 2.} Since we are working with an arbitrary $j\geq 1$, for this part we will omit the subscript and denote by $\cI(t,\lambda)$ and $h(s)$ any of the integrals $\cI_j(t,x,\xi)$ and function $h_j(s)$ respectively. We then analyze
\[
\cI(t,\lambda) = \int \int  e^{i(t-s)\omega+i\lambda\psi(s)}(i\omega)^\alpha h(s)A(s)dsd\omega.
\]
After the change of variables $\omega=\lambda\theta$ and the Taylor expansion of $\psi$,
\begin{equation}\label{eq:psi_Taylor}
\psi(s)= \psi(t)+(s-t)\Psi(t,s),\quad \Psi(t,s)=\int^1_0\psi'(\rho s + (1-\rho)t)d\rho,
\end{equation}
the integral rewrites as
\[
\cI(t,\lambda) = \lambda^{\alpha+1}e^{i\lambda\psi(t)}\int \int  e^{i\lambda(t-s)(\theta+\Psi(t,s))}(i\theta)^\alpha h(s)A(s)dsd\theta.
\]
We now take a smooth cut-off $\chi_1(\theta)\in C^\infty_c(\RR^n,[0,1])$ satisfying
\[
\chi_1(\theta)=1\quad \text{for}\quad |\theta|\leq 1\quad\text{and}\quad\chi_1(\theta)=0\quad\text{for}\quad|\theta|\geq 2,
\]
and split $\cI(t,\lambda) = I_1(t,\lambda)+I_2(t,\lambda)$ with
\[
I_1(t,\lambda) = \lambda^{\alpha+1}e^{i\lambda\psi(t)}\int \int  e^{i\lambda(t-s)(\theta+\Psi(t,s))}(i\theta)^\alpha \chi_1(\theta)h(s)A(s)dsd\theta.
\]

{\em 2.1.} To estimate $I_1$ we apply Fubini and the change of variables $\eta = \theta+\Psi(t,s)$ to get
\[
I_1(t,\lambda) = \lambda^{\alpha+1}e^{i\lambda\psi(t)}\int \int  e^{i\lambda(t-s)\eta}b_1(t,s,\eta,x,\xi)d\eta ds
\]
where
\[
b_1(t,s,\eta,x,\xi) = i^\alpha(\eta-\Psi(t,s,x,\xi))^\alpha \chi_1(\eta-\Psi(t,s,x,\xi))h(s)A(s,x,\xi).
\]
Applying the stationary phase method with critical point $(s,\eta)=(t,0)$, we deduce that as a function of $(x,\xi)$ and for some constant $B_0\in \CC$,
\[
\int \int  e^{i\lambda(t-s)\eta}b_1(t,s,\eta,x,\xi)d\eta ds =  B_0\lambda^{-1}b_1(t,t,0,x,\xi) \quad\text{mod}\quad S^{-2}(\RR^n_x\times\RR^n_\xi)
\]
since
\[
b_1(t,t,0,x,\xi) = (-i\psi'(t,x,\xi))^\alpha\chi_1(-\psi'(t,x,\xi))h(t)A(t,x,\xi)
\]
is a symbol of order zero depending smoothly on $t\geq 0$. In consequence, we obtain that
\[
I_1(t,\lambda) = e^{i\phi(t,x,\xi)}\beta_1(t,x,\xi)
\]
with $\beta_1(t,\cdot,\cdot)\in S^\alpha(\RR^n_x\times\RR^n_\xi)$ depending smoothly on $t\geq 0$.\\

{\em 2.2.} For $I_2(t,\lambda)$ we use the fact that
\begin{equation}\label{exp_identity}
-\frac{1}{i\lambda\theta}\partial_s(e^{i\lambda(t-s)\theta}) = e^{i\lambda(t-s)\theta},
\end{equation}
and integrate by parts twice to obtain
\[
I_2(t,\lambda) = \lambda^{\alpha-1}e^{i\lambda\psi(t)}\int \int  e^{i\lambda(t-s)\theta}(i\theta)^{\alpha-2} (1-\chi_1(\theta))\partial^2_s\left(e^{i\lambda\Psi(t,s)}h(s)A(s)\right)dsd\theta.
\]
Both integrals are absolutely convergent, thus, by differentiating and applying Fubini we get
\[
I_2(t,\lambda) = \lambda^{\alpha+1}e^{i\lambda\psi(t)}\int \int  e^{i\lambda(t-s)(\theta+\Psi(t,s))}\tilde{b}_2(t,s,\theta,x,\xi)d\theta ds,
\]
with
\[
\tilde{b}_2(t,s,\theta,x,\xi) = (i\theta)^{\alpha-2} (1-\chi_1(\theta))(-|\psi'(s)|^2 + 2i\lambda^{-1}\psi'(s)A'(s) + \lambda^{-2}A''(s)).
\]
Using once again the change of variables $\eta = \theta+\Psi(t,s)$, the previous rewrites as
\[
I_2(t,\lambda) = \lambda^{\alpha+1}e^{i\lambda\psi(t)}\int \int  e^{i\lambda(t-s)\eta}b_2(t,s,\eta,x,\xi)d\eta ds
\]
where
\[
b_2(t,s,\eta,x,\xi) = \tilde{b}_2(t,s,\eta-\Psi(t,s),x,\xi).
\]
Consequently, and similarly as in the previous case, by the stationary phase method with respect to the parameter $\lambda=|\xi|$ we get
\[
I_2(t,\lambda) = e^{i\phi(t,x,\xi)}\beta_2(t,x,\xi)
\]
for a symbol $\beta_2(t,\cdot,\cdot)\in S^\alpha(\RR^n_x\times\RR^n_\xi)$ depending smoothly on $t\geq 0$.\\

{\em 2.3.} We conclude from the previous that for each $j\geq 1$ there is a symbol $\beta_j(t,\cdot,\cdot)\in S^\alpha_{1,0}(\RR^n_x\times\RR^n_\xi)$, depending smoothly on $t\geq 0$, such that
\[
(\partial^\alpha_tv)_{j} = \frac{1}{(2\pi)^{n+1}}\int e^{i\phi(t,x,\xi)}\beta_j(t,x,\xi)\hat{v}_0(\xi)d\xi .
\]

{\em 3.} It remains to address the case $j=0$. We will estimate $(\partial^\alpha_tv)_{0,\epsilon}(t)$ for $t>0$ and we begin by writing
\[
\begin{aligned}
(\partial^\alpha_tv)_{0,\epsilon}(t) &=\frac{1}{2\pi}\int\int^\infty_0 e^{i(t-s)\omega}(i\omega)^\alpha h_0(s)(v(s)-v_0)\chi(\epsilon \omega)dsd\omega\\
&= \frac{1}{2\pi}\int\int e^{i(t-s)\omega}(i\omega)^\alpha h_0(s)(v(s)-v_0)\chi(\epsilon \omega)dsd\omega\\
&\quad - \frac{1}{2\pi}\int\int^0_{-\infty} e^{i(t-s)\omega}(i\omega)^\alpha h_0(s)(v(s)-v_0)\chi(\epsilon \omega)dsd\omega\\
&=:\mathcal{J}_1(\epsilon) - \mathcal{J}_2(\epsilon).
\end{aligned}
\]
The previous is justified by the fact that all the functions involved above are defined for negative times. \\

{\em 3.1.} The term $\cJ_1(\epsilon)$ integrates the variable $s$ in the whole real line so it is equal to
\[
\frac{1}{2\pi}\int \int e^{i(t-s)\omega}(i\omega)^\alpha h_0(s)v(s)\chi(\epsilon \omega)dsd\omega -  
\mathfrak{F}^{-1}\left[(i\omega)^\alpha \mathfrak{F}[h_0(s)](\omega)\chi(\epsilon \omega)\right](t)v_0.
\]
Taking the limit $\epsilon\to 0$, by dominate convergence we obtain
\[
\frac{1}{2\pi}\int \int e^{i(t-s)\omega}(i\omega)^\alpha h_0(s)v(s)dsd\omega -  
\mathfrak{F}^{-1}\left[(i\omega)^\alpha \mathfrak{F}[h_0(s)](\omega)\right](t)v_0,
\]
where $\mathfrak{F}[h_0(s)](\omega)\in \mathcal{S}(\RR)$ hence the second term belongs to $C^\infty([0,\infty);H^m(\RR^n))$. On the other hand, the first one is estimated analogously as for the cases $j\geq 1$, and we conclude by the stationary phase method that modulo a $C^\infty([0,\infty);H^m(\RR^n))$ function,
\[
\lim_{\epsilon\to 0}\cJ_1(\epsilon) = \frac{1}{(2\pi)^n}\int e^{i\phi(t,x,\xi)}\gamma_1(t,x,\xi)\hat{v}_0(\xi)d\xi
\]
for some symbol $\gamma_1(t,\cdot,\cdot)\in S^{\alpha}(\RR^n_x\times\RR^n_\xi)$ depending smoothly on $t\geq 0$.\\

{\em 3.2.} We finally estimate $\cJ_2(\epsilon)$. Fubini's theorem and the Fourier integral representation of $v$ give (modulo a smooth function)
\[
\begin{aligned}
\cJ_2(\epsilon) &= \frac{1}{(2\pi)^{n+1}}\int\int\int^0_{-\infty} e^{i(t-s)\omega}(i\omega)^\alpha h_0(s)(e^{i\lambda\psi(s)}A(s)-e^{i\lambda\psi(0)}A(0))\chi(\epsilon \omega)\hat{v}_0dsd\omega d\xi\\
&= \frac{1}{(2\pi)^{n+1}}\int j_{\epsilon}(t,x,\xi)\hat{v}_0(\xi)d\xi,
\end{aligned}
\]
where after the change of variables $\omega=\lambda\theta$, the kernel $j_{\epsilon}(t,x,\xi)$ rewrites as
\[
j_{\epsilon} =\lambda^{\alpha+1}\int\int^0_{-\infty} e^{i\lambda(t-s)\theta}(i\theta)^{\alpha}h_0(s)(e^{i\lambda\psi(s)}A(s)-e^{i\lambda\psi(0)}A(0))\chi(\epsilon \lambda\theta)dsd\theta,
\]
Using once again the identity \eqref{exp_identity} along with integration by parts, we get
\[
\begin{aligned}
j_{\epsilon} =
&\lambda^{\alpha}\int\int^0_{-\infty} e^{i\lambda(t-s)\theta}(i\theta)^{\alpha-1}\partial_s\left(h_0(s)(e^{i\lambda\psi(s)}A(s)-e^{i\lambda\psi(0)}A(0))\right)\chi(\epsilon \lambda\theta)dsd\theta\\
=&i\lambda^{\alpha+1}\int\int^0_{-\infty} e^{i\lambda((t-s)\theta+\psi(s))}(i\theta)^{\alpha-1}B_0(s)\chi(\epsilon \lambda\theta)dsd\theta\\
&+i\lambda^{\alpha}\int\int^0_{-\infty} e^{i\lambda((t-s)\theta+\psi(s))}(i\theta)^{\alpha-1}C_0(s)\chi(\epsilon \lambda\theta)dsd\theta\\
&+ i\lambda^{\alpha}\int\int^0_{-\infty} e^{i\lambda(t-s)\theta}(i\theta)^{\alpha-1}h'_0(s)(e^{i\lambda\psi(s)}A(s)-e^{i\lambda\psi(0)}A(0))\chi(\epsilon \lambda\theta)dsd\theta\\
&=: j_{\epsilon,1} + j_{\epsilon,2} + j_{\epsilon,3},
\end{aligned}
\]
where we denote $B_0(s)=h_0(s)\psi'(s)A(s)$ and $C_0(s)=h_0(s)A'(s)$.

Recalling that $h_0(s)=1$ in a neighborhood of $s=0$, the last of the three terms rewrites as
\[
j_{\epsilon,3}=i\lambda^{\alpha}\int\int e^{i\lambda(t-s)\theta}(i\theta)^{\alpha-1}h'_0(s)(e^{i\lambda\psi(s)}A(s)-e^{i\lambda\psi(0)}A(0))\chi(\epsilon \lambda\theta)dsd\theta,
\]
and following computations similar to previous cases (for the other functions $h_j$, $j\geq 1$), the stationary phase method and dominate convergence allows us to express the limit as $\epsilon\to 0$ of this integral in the form $e^{i\phi(t,x,\xi)}d_3(t,x,\xi)$, for some symbol $d_3(t,\cdot,\cdot)\in S^\alpha(\RR^n_x\times\RR^n_\xi)$ depending smoothly on $t\geq 0$.

Regarding
\[
j_{\epsilon,1}=i\lambda^{\alpha+1}\int\int^0_{-\infty} e^{i\lambda((t-s)\theta+\psi(s))}(i\theta)^{\alpha-1}B_0(s)\chi(\epsilon \lambda\theta)dsd\theta,
\]
where $B_0(s)=h_0(s)\psi'(s)A(s)$, let's consider a cut-off $\chi_2\in C^\infty_c(\RR^n,[0,1])$ satisfying that
\[
\chi_2(\theta)=1\quad \text{for}\quad \frac{1}{C_{1}}\leq |\theta|\leq C_1\quad\text{and}\quad \text{supp}(\chi_2)\subset\{\frac{1}{2C_{1}}\leq |\theta|\leq 2C_1\},
\]
with
\[
C_1 > \sup\{|\psi'(s,x,\zeta)|^{-1}+|\psi'(s,x,\zeta)|:s\in\text{supp}(h_0),\;x\in\text{supp}(v_0),\zeta\in\mathbb{S}^{n-1}\}.
\]
In particular, there is $\delta>0$ such that $|\theta - \psi'(s)|>\delta$ in $\text{supp}(1-\chi_2)$. 
We then split $j_{\epsilon,1} = j_{\epsilon,1,1}+j_{\epsilon,1,2}$ with
\[
j_{\epsilon,1,1} =i\lambda^{\alpha+1}\int\int^0_{-\infty} e^{i\lambda((t-s)\theta+\psi(s))}(i\theta)^{\alpha-1}B_0(s)\chi(\epsilon \lambda\theta)(1-\chi_2(\theta))dsd\theta
\]
and 
\[
j_{\epsilon,1,2} =i\lambda^{\alpha+1}\int\int^0_{-\infty} e^{i\lambda((t-s)\theta+\psi(s))}(i\theta)^{\alpha-1}B_0(s)\chi(\epsilon \lambda\theta)\chi_2(\theta)dsd\theta.
\]
Introducing the differential operator $Lf(s) = \frac{-1}{i(\theta-\psi'(s))}\partial_s f(s)$, we use the identity
\[
e^{i\lambda((t-s)\theta + \psi)} = \frac{1}{\lambda}L\left(e^{i\lambda((t-s)\theta + \psi)}\right)
\]
and integration by parts to get
\[
\begin{aligned}
j_{\epsilon,1,1} = &-i\lambda^{\alpha}e^{i\lambda((t-s)\theta+\psi(0))}(i\theta)^{\alpha-1}\chi(\epsilon \lambda\theta)(1-\chi_2(\theta))\frac{B_0(0,\theta) }{i(\theta-\psi'(0))}d\theta \\
&+i\lambda^{\alpha}\int\int^0_{-\infty} e^{i\lambda((t-s)\theta+\psi(s))}(i\theta)^{\alpha}\chi(\epsilon \lambda\theta)(1-\chi_2(\theta))B_1(s,\theta)dsd\theta,
\end{aligned}
\]
with $B_1 = L^t(B_0)$. By defining $B_N = (L^t)^NB_0$, for $N\geq 1$, an iteration of the previous gives
\[
\begin{aligned}
j_{\epsilon,1,1} &= -\sum_{l=0}^{N-1}\lambda^{\alpha-l}\int e^{i\lambda( t\theta + \psi(0))}(i\theta)^{\alpha-1}\chi(\epsilon \lambda\theta)(1-\chi_2(\theta))\frac{B_l(0,\theta) }{i(\theta-\psi'(0))}d\theta \\
&\quad+  \lambda^{\alpha-N}\int\int^0_{-\infty} e^{i\lambda((t-s)\theta+\psi(s))}(i\theta)^{\alpha-1}\chi(\epsilon \lambda\theta)(1-\chi_2(\theta))B_{N}(s,\theta)dsd\theta.
\end{aligned}
\]
Recalling that $\lambda\psi(0)=x\cdot\xi$ and noticing that all the integrals in $\theta$ are absolutely convergent, in the limit as $\epsilon\to 0$ we obtain by dominate convergence,
\begin{equation}\label{eq:j_1}
\begin{aligned}
j_{1,1}:=\lim_{\epsilon\to 0}j_{\epsilon,1,1}&=e^{ix\cdot\xi}\sum_{l=0}^{N-1}\lambda^{\alpha-l}\int e^{i\lambda t\theta }(i\theta)^{\alpha-1}(\chi_2(\theta)-1)\frac{B_l(0,\theta) }{i(\theta-\psi'(0))}d\theta \\
&\quad+  \lambda^{\alpha-N}\int\int^0_{-\infty} e^{i\lambda((t-s)\theta+\psi(s))}(i\theta)^{\alpha-1}(1-\chi_2(\theta))B_{N}(s,\theta)dsd\theta.
\end{aligned}
\end{equation}
Let's rewrite the previous in the following form
\begin{equation}\label{eq:asym_exp_j1}
e^{-ix\cdot\xi}j_{1,1}(t,x,\xi) = \sum_{l=1}^{N-1}a_{l}(t,x,\xi) + S_{N}(t,x,\xi)
\end{equation}
where
\[
a_{l}(t,x,\xi)=|\xi|^{\alpha-l}\int e^{i|\xi|t\theta }(i\theta)^{\alpha-1}(\chi_2(\theta)-1)\frac{B_l(0,\theta,x,\xi) }{i(\theta-\psi'(0,x,\xi))}d\theta
\]
and
\[
S_{N}(x,\xi) = |\xi|^{\alpha-N}\int\int^0_{-\infty} e^{i(|\xi|(t-s)\theta+\phi(s,x,\xi)-x\cdot\xi)}(i\theta)^{\alpha-1}(1-\chi_2(\theta))B_{N}(s,\theta,x,\xi)dsd\theta.
\]
It is clear from \eqref{eq:j_1} with $N=1$ that, since the phase functions involved are homogeneous of order 1 in $\xi$, all the derivatives $\partial^\beta_x\partial^\gamma_\xi j_{1,1}(t,x,\xi)$ are of tempered growth. Moreover, for any compact $K\subset\RR^n$ there is a constant $C_{K,N}>0$ such that
\[
|S_{N}(t,x,\xi)|\leq C_{K,N}(1+|\xi|)^{\alpha - N},\quad\forall (t,x,\xi)\in [0,T]\times K\times\RR^n.
\]
To conclude that \eqref{eq:asym_exp_j1} is indeed an asymptotic sum of symbols (depending smoothly on $t$), we verify that $a_l\in S^{\alpha-l}$, for all $l\geq 0$ (this is true after adding a cut-off supported away from $\xi=0$). This follows from the fact that $\psi$ is smooth with respect to $(x,\hat{\xi})$, with $\hat{\xi}=\xi/|\xi|$ smooth away from the origin and homogeneous of order zero, and also because the functions $B_l$ involve the symbol $A\in S^0$. For a complete statement of the property used here, we refer the reader, for instance, to \cite[\S 1]{GS1994}. 

We have shown that $e^{-ix\cdot\xi}j_{1,1}(t,\cdot,\cdot) \in S^\alpha(\RR^n_x\times\RR^n_\xi)$ depending smoothly on $t\geq 0$, and moreover, there is a $\Psi$DO of order $\alpha$, say $Q_1$, such that
\[
\lim_{\epsilon\to 0}\frac{1}{(2\pi)^{n+1}}\int j_{\epsilon,1,1}(t,x,\xi)\hat{v}_0(\xi)d\xi = Q_1[v_0].
\]

We now estimate $j_{\epsilon,1,2}$. Since the function $\chi_2(\theta)$ is compactly supported it makes the integral without $\chi(\epsilon\lambda\theta)$ to be absolutely convergent, thus, we can take the limit as $\epsilon\to 0$ and work with
\[
j_{1,2} := \lim_{\epsilon\to 0}j_{\epsilon,1,2}=i\lambda^{\alpha+1}\int\int^0_{-\infty} e^{i\lambda((t-s)\theta+\psi(s))}B_0(s)(i\theta)^{\alpha-1}\chi_2(\theta)dsd\theta.
\]
For $s<0<t$, we use
\[
e^{i\lambda((t-s)\theta + \psi)} = \frac{1}{i\lambda(t-s)}\partial_\theta\left(e^{i\lambda((t-s)\theta + \psi)}\right)
\]
and integrate by parts once to get
\[
j_{1,2}=-i\lambda^{\alpha}\int\int^0_{-\infty} e^{i\lambda((t-s)\theta+\psi(s))}\frac{B_0(s)}{i(t-s)}\partial_\theta\left((i\theta)^{\alpha-1}\chi_2(\theta)\right)dsd\theta.
\]
Let's denote $g(\theta) = \partial_\theta\left((i\theta)^{\alpha-1}\chi_2(\theta)\right)$. By using \eqref{eq:psi_Taylor}, Fubini's theorem, and the change of variables $\eta = \theta + \Psi(t,s)$, we get $j_{1,2}= e^{i\phi(t,x,\xi)}d_{1,2}(t,x,\xi)$ with
\[
d_{1,2}(t,x,\xi)=-|\xi|^{\alpha}\int^0_{-\infty}\int e^{i|\xi|(t-s)\eta}\frac{B_0(s,x,\xi)}{(t-s)}g(\eta-\Psi(t,s,x,\xi))d\eta ds.
\]
We notice that the previous integrals are absolutely convergent and integrate over compact regions. Since $A(s,x,\xi)\in S^0(\RR^n_x\times\RR^n_\xi)$ and $\Psi(t,s,x,\xi)$ is positive homogeneous of order zero in $\xi$, it is then simple to verify that $d_{1,2}(t,\cdot,\cdot)\in S^\alpha(\RR^n_x\times\RR^n_\xi)$ and depends smoothly on $t>0$, with an integrable singularity in $t=0$.

Similar to what was previously done, in the limit as $\epsilon\to 0$, the integral 
\[
j_{\epsilon,2}=i\lambda^{\alpha}\int\int^0_{-\infty} e^{i\lambda((t-s)\theta+\psi(s))}(i\theta)^{\alpha-1}C_0(s)\chi(\epsilon \lambda\theta)dsd\theta
\]
can be expressed as
\[
e^{ix\cdot\xi}d_{2,1}(t,x,\xi) + e^{i\phi(t,x,\xi)}d_{2,2}(t,x,\xi)
\]
for symbols $d_{2,1},d_{2,2}\in S^{\alpha-1}$, the first one smooth for $t\geq 0$, while the second one smooth for $t>0$ and with an integrable singularity at $t=0$.\\

We conclude that
\[
\mathcal{J}_2(0) = \lim_{\epsilon\to 0}\mathcal{J}_2(\epsilon) = Q[v_0](t,x) + \frac{1}{(2\pi)^{n+1}}\int e^{i\phi(t,x,\xi)}\gamma_2(t,x,\xi)\hat{v}_0(\xi)d\xi
\]
for some $\Psi$DO $Q$ of order $\alpha$ and $\gamma_2(t,\cdot,\cdot)\in S^\alpha(\RR^n_x\times\RR^n_\xi)$ depending smoothly for all $t>0$ and with an integrable singularity at $t=0$. Moreover, for $T>0$ such that $\supp(h_0)\subset[-T,T]$, there is some constant $C>0$ for which the next inequality holds,
\[
\|\mathcal{J}_2(0)\|_{L^1([0,T];H^{m-\alpha}(\Omega))} \leq C\|v_0\|_{H^m(\RR^n)}\left(1+\int^T_0\int^0_{-T}\frac{1}{(t-s)}ds\right).
\]
with last integral bounded because
\[
\int^T_0\int^0_{-T}\frac{1}{(t-s)}ds dt=\int^T_0\ln\left(1+\frac{T}{t}\right)dt = T\int^\infty_2\frac{\ln(z)}{(z-1)^2}dz<\infty.
\]
Notice the second equality above is obtained after the change of variables $z=1+\frac{T}{t}$.\\

{\em 4.} Summarizing all the previous computations, we have proven that there are symbols $\beta_j(t,\cdot,\cdot)$, $q(t,\cdot,\cdot)\in S^\alpha(\RR^n_x\times\RR^n_\xi)$, $j\geq 0$, with $q,\beta_j$ depending smoothly on $t\geq 0$ for $j\geq 1$, but $\beta_0$ only depending smoothly on $t>0$ and with an integrable singularity at $t=0^+$, such that, for any $v_0\in H^m_{comp}(\RR^n)$,
\[
(\partial^\alpha_tv)_{j} = \frac{1}{(2\pi)^{n}}\int e^{i\phi(t,x,\xi)}\beta_j(t,x,\xi)\hat{v}_0(\xi)d\xi\quad \text{for}\quad j\geq 1,
\]
and
\[
\begin{aligned}
(\partial^\alpha_tv)_{0} &= \frac{1}{(2\pi)^{n}}\int e^{i\phi(t,x,\xi)}\beta_0(t,x,\xi)\hat{v}_0(\xi)d\xi \\
&\quad + \frac{1}{(2\pi)^{n}}\int e^{ix\cdot\xi}q(t,x,\xi)\hat{v}_0(\xi)d\xi.
\end{aligned}
\]
The mapping properties of FIO's and $\Psi$DO's imply that for each $t>0$ and $j\geq 0$, $(\partial^\alpha_tv)_{j}(t,\cdot)\in H^{m-\alpha}(\RR^n)$. We also deduce from Remark \ref{rmk:def_conv} and Definition \ref{def:Caputo} (of the Caputo derivative), that for a fixed $T>0$ there is $N_T>0$ and a constant $C_T>0$ such that
\[
\|(\partial^\alpha_tv)_{0}\|_{L^1([0,T];H^{m-\alpha}(\RR^n))} +\sum_{j=1}^{N_T}\|(\partial^\alpha_tv)_{j}\|_{C([0,T];H^{m-\alpha}(\RR^n))} \leq C_T\|v_0\|_{H^{m}(\RR^n)}.
\]
\end{myproof}

The next lemma gives microlocal stability for the inverse problem, since the unattenuated observation map is known to be stable (under appropriate geometric conditions). The stability property will then follow by requiring the injectivity of the (attenuated) observation map.

\begin{lemma}\label{lemma:observation_maps}
    Let $\Lambda_{0}:H^1_0(\mathcal{K})\to H^1((0,T)\times\Gamma)$ be the measurement operator for the unattenuated case (i.e., $a\equiv 0$). Then, there is a constant $C>0$ such that
    \[
    \|\Lambda_{0} u_0\|_{H^1((0,T)\times\Gamma)}\leq C\left(\|\Lambda_\alpha u_0\|_{H^1((0,T)\times\Gamma)} + \|u_0\|_{H^{\alpha}(\RR^n)}\right),\quad\forall u_0\in H^1_0(\mathcal{K}).
    \]
\end{lemma}
\begin{proof}

Let $u$ be the solution to \eqref{frac_PAT} and $v$ to \eqref{unatt_PAT} with initial condition $v_0=u_0$. We see that $w=v-u$ satisfies 
\begin{equation}\label{eq:w}
\square_{c,\alpha} w = a\partial^\alpha_tv,\quad (w,\partial_tw)|_{t=0} =(0,\frac{a}{\Gamma(2-\alpha)}u_0),
\end{equation}
therefore, the well-posedness of the forward problem in Theorem \ref{thm:regularity} gives us the inequality
\[
\|(w,\partial_tw)\|_{L^\infty([0,T];\mathcal{H}^{1}_0(\Omega))}\leq C\left(\|u_0\|_{L^2(\Omega)}+\|a\partial^\alpha_tv\|_{L^1([0,T];L^2(\Omega)))}\right),
\]
but because $a(x)$ is compactly supported within $\Omega$,  
we can replace the $L^1([0,T];L^2(\Omega))$-norm by the norm in $L^1([0,T];L^2(\RR^n))$. 
From Theorem \ref{thm:reg_frac}, we then deduce the inequality
\begin{equation}\label{energy_est_w}
\|(w,\partial_tw)\|_{L^\infty([0,T];\mathcal{H}^{1}_0(\Omega))}\leq 
C\|u_0\|_{H^\alpha(\Omega)}.
\end{equation}

We now focus on the trace of the error function $w$, in particular, in determining the regularity of $Rw= w|_{(0,T)\times\partial\Omega}$ where we denote by $R$ the boundary trace operator that is well defined since $w\in H^1(\Omega)$. 

Let's recall the fixed point formulation of equation $\square_{c,\alpha} w = a\partial^\alpha_tv$ from section \ref{sec:forward_pblm}, 
\[
\begin{aligned}
w(t) = &U_1(t)[0,\frac{a}{\Gamma(2-\alpha)}u_0] + \int^t_0U_1(t-s)[0,a\partial^\alpha_tv(s)]ds\\&+\int^t_0\int^s_0U_1(t-s)Q(s-r)(w(r),\partial_rw(r))^tdrds,
\end{aligned}
\]
where $U_1 = \Pi_1\UU$, with $\Pi_1$ the projection to the first component and $\UU$ the semigroup associated with the (unattenuated) wave equation, i.e. the initial source-to-solution map. The integral above is understood in the Bochner integral sense (more details of this can be found in  Appendix \ref{Appx:int_Banach}).

We apply the trace operators and obtain
\begin{equation}\label{Duhamel}
\begin{aligned}
Rw(t) =& RU_1(t)[0,\frac{a}{\Gamma(2-\alpha)}u_0] + \int^t_0RU_1(t-s)[0,a\partial^\alpha_tv(s)]ds\\
&+\int^t_0RU_1(t-s)\int^s_0Q(s-r)(w(r),\partial_rw(r))^tdrds,
\end{aligned}
\end{equation}
where we notice that, due to finite propagation speed, there is $\delta>0$ for which $RU_1(t)[f,g]=0$ for all $t\in (0,\delta)$ and any pair of initial conditions in $H^1_0(\mathcal{K})\times L^2(\mathcal{K})$.

Following computations from Appendix \ref{Appx:int_Banach}, the $H^1((0,T)\times\partial\Omega)$-norms of the integral terms are bounded by
\[
\begin{aligned}
&\|F_1\|_{L^1((0,T);L^2(\Omega))}+\|F_1\|_{L^\infty((\delta,T);L^2(\Omega))}\\
&+\|F_2\|_{L^1((0,T);L^2(\Omega))}+\|F_2\|_{L^\infty((\delta,T);L^2(\Omega))}
\end{aligned}
\]
with $F_1(s)=a\partial^\alpha_tv(s)$ and 
\[
F_2(s)=\int^s_0Q(s-r)(w(r),\partial_rw(r))^tdr,
\]
and further bounded from above by a positive constant times
\[
\begin{aligned}
&\|\partial^\alpha_tv\|_{L^1((0,T);L^2(\Omega))}+\|\partial^\alpha_tv\|_{L^\infty((\delta,T);L^2(\Omega))}\\
&+\|\partial_tw\|_{L^1((0,T);L^2(\Omega))}+\|\partial_tw\|_{L^\infty((\delta,T);L^2(\Omega))}.
\end{aligned}
\]
By recalling the estimate from Theorem \ref{thm:reg_frac} and \eqref{energy_est_w}, we obtain a final upper bound of a positive constant times $\|u_0\|_{H^{\alpha}(\Omega)}$.
In other words, by denoting $Ku_0=Rw|_{(0,T)\times\Gamma}$ 
, for any $u_0\in H^1_0(\mathcal{K})$ we have
\[
\begin{aligned}
&\|Ku_0\|_{H^{1}((0,T)\times \Gamma)}
\leq C\|u_0\|_{H^{\alpha}(\Omega)}.
\end{aligned}
\]
The proof is then complete by noticing that $\Lambda_\alpha u_0 = Ru$, $\Lambda_{0}u_0 = Rv$ and $v = u + w$, hence $K u_0 = \Lambda_0 u_0 - \Lambda_{\alpha}u_0$.
\end{proof}

\begin{definition}
    For a compact $\mathcal{K}\subset\Omega$ and $\Gamma\subset\partial\Omega$ relatively open. We say the pair $(\mathcal{K},\Gamma)$ satisfies the visibility condition of singularities if all unit speed geodesics issued from $\mathcal{K}$ reach  $\Gamma$ non-tangentially and in finite time. 
    We also define $T_1(\mathcal{K},\Gamma)$ as the maximum of the exit times of unit speed geodesic issued from $\mathcal{K}$ and exiting $\Omega$ through $\Gamma$.
\end{definition}

\begin{remark}\label{rmk:stab_comp_data}
    In the case of complete data, i.e. $\Gamma=\partial\Omega$, we might replace the previous condition with the manifold $(\Omega,c^{-2}dx^2)$ being non-trapping and $\partial\Omega$ strictly convex. In this case, the observation time $T$ has to be larger than $T_1(\Omega)$ defined as half the length of the longest geodesic segment contained in $\Omega$. We also take initial conditions in $H^1_0(\Omega)$.
\end{remark}

\begin{theorem}\label{thm:stability}
    Under the visibility condition of singularities and the foliation condition of theorem \ref{thm:uniqueness_att}, if $T>T_1(\mathcal{K},\Gamma)$, then there is $C>0$ such that
    \[
    \|u_0\|_{H^1(\Omega)}\leq C\|\Lambda_\alpha u_0\|_{H^1((0,T)\times\Gamma)},\quad \forall u_0\in H^1_0(\mathcal{K}).
    \]
\end{theorem}
\begin{proof}
We start from the result of lemma \ref{lemma:observation_maps}, 
\[
\|\Lambda_{0} u_0\|_{H^1((0,T)\times\Gamma)}\leq C\left(\|\Lambda_\alpha u_0\|_{H^1((0,T)\times\Gamma)}+\|u_0\|_{H^{\alpha}(\Omega)}\right).
\]
The visibility condition implies stability for the unattenuated inverse problem, thus $\|u_0\|_{H^1(\Omega)}\leq C\|\Lambda_{0} u_0\|_{H^1((0,T)\times\Gamma)}$ for some constant $C>0$, which gives
\[
\|u_0\|_{H^1(\Omega)}\leq C\left(\|\Lambda_\alpha u_0\|_{H^1((0,T)\times\Gamma)}+\|u_0\|_{H^{\alpha}(\Omega)}\right).
\]
for a possibly different constant $C>0$.

Finally, the foliation condition implies the injectivity of $\Lambda_\alpha$ (as stated in Theorem \ref{thm:uniqueness_att}), and added to the compactness of the injection $H^1_0(\Omega)\to H^{\alpha}_0(\Omega)$, a standard compactness-uniqueness argument leads to the desired inequality.
\end{proof} 

\section{Reconstruction}

We adapt the dissipative sharp time reversal approach introduced by the authors in \cite{AP2018} to the setting of fractional memory terms. 
Let's begin by defining the time-reversed system and operator based on the following observation.
\begin{lemma}\label{lemma:adjoint}
For any pair of functions $f,g$ such that $\partial_tf,\partial_tg\in L^\infty((0,T);L^2(\Omega))$, then
\[
\langle a\partial^\alpha_tf,c^{-2}\partial_tg\rangle = \langle c^{-2}\partial_tf,aI^{1-\alpha}_T\partial_tg\rangle
\]
where $\langle\cdot,\cdot\rangle$ is the inner product in $L^2((0,T)\times\Omega)$ and $I^{1-\alpha}_T$ is a time-reversed fractional integral operator defined as
\[
I^{1-\alpha}_Tu(t) := \frac{1}{\Gamma(1-\alpha)}\int^T_t\frac{u(s)}{(s-t)^\alpha}ds.
\]
\end{lemma}
\begin{proof}
Direct computations employing Fubini's theorem give
\[
\begin{aligned}
\langle a\partial^\alpha_tf,c^{-2}\partial_tg\rangle &= \int_\Omega\int^T_0c^{-2}\overline{g'(x,t)}\left(\frac{a(x)}{\Gamma(1-\alpha)}\int^t_0\frac{f'(x,s)}{(t-s)^\alpha}ds\right)dtdx\\
&=\int_\Omega\int^T_0\int^T_s\frac{a(x)c^{-2}\overline{g'(x,t)}f'(x,s)}{\Gamma(1-\alpha)(t-s)^\alpha}dtdsdx\\
&= \int_\Omega\int^T_0 c^{-2}f'(x,s)\left(\frac{a(x)}{\Gamma(1-\alpha)}\int^T_s\frac{\overline{g'(x,t)}}{(t-s)^\alpha}dt\right)dsdx.
\end{aligned}
\]
\end{proof}

We define the Time-Reversal operator as the map $A_\alpha:h\mapsto (v,\partial_tv)|_{t=0}$, that takes the boundary data $h\in H^1((0,T)\times\partial\Omega)$ to the restriction at time $t=0$ of the solution $v$ to the final boundary value system:
\[
\left\{\begin{array}{l}
\partial_t^2v - c^2\Delta v - a(x)I^{1-\alpha}_T\partial_tv =0,\\
(v,\partial_tv)|_{t=T} = (\phi,0),\\
v|_{(0,T)\times\partial\Omega} = h.
\end{array}\right.
\]
Here $\phi$ is the harmonic extension of $h(\cdot,T)$ to $\Omega$. 

For any $u_0\in H^1_0(\Omega)$
we take $h=\Lambda_\alpha u_0$. The error function $w=u-v$ then solves
\[
\left\{\begin{array}{l}
\partial_t^2w - c^2\Delta w = - a(x)\partial^{\alpha}_tu- a(x)I^{1-\alpha}_T\partial_tv,\\
(\bu,\partial_t\bu)|_{t=T} = (u(T)-\phi,\partial_tu(T)),\\
w|_{(0,T)\times\partial\Omega} = 0,
\end{array}\right.
\]
therefore, standard energy computations along with Lemma \ref{lemma:adjoint} for the real-valued functions $u,v$ give
\[
\begin{aligned}
E_\Omega(w,T)-E_\Omega(w,0) &= -2\langle a\partial^{\alpha}_tu,c^{-2}\partial_tw\rangle - 2\langle aI^{1-\alpha}_T\partial_tv,c^{-2}\partial_tw \rangle\\
&= -2\langle a\partial^{\alpha}_tu,c^{-2}\partial_tu\rangle +2\langle a\partial^{\alpha}_tu,c^{-2}\partial_tv\rangle\\
&\quad - 2\langle aI^{1-\alpha}_T\partial_tv,c^{-2}\partial_tu \rangle + 2\langle aI^{1-\alpha}_T\partial_tv,c^{-2}\partial_tv \rangle\\
&= -2\langle a\partial^{\alpha}_tu,c^{-2}\partial_tu\rangle + 2\langle a\partial^\alpha_tv,c^{-2}\partial_tv \rangle.
\end{aligned}
\]
Recalling the definition of the energy norm \eqref{def:energy_norm}, we see that
\[
\begin{aligned}
E_\Omega(w,0) &=\|(w(0),\partial_tw(0))\|^2_{H^1_0(\Omega)\times L^2(\Omega)}\quad\text{and}\\ E_\Omega(w,T)&=\|(u(T)-\phi,\partial_tu(T))\|^2_{H^1_0(\Omega)\times L^2(\Omega)},
\end{aligned}
\]
therefore, it follows from Lemma \ref{lemma:positivity} and the identity
\[
\|f-\phi\|^2_{H^1_0(\Omega)}=\|f\|^2_{H^1_0(\Omega)}-\|\phi\|^2_{H^1_0(\Omega)},\quad\forall f\in H^1_0(\Omega),
\]
that
\[
\begin{aligned}
\|(w(0),\partial_tw(0))\|^2_{H^1_0(\Omega)\times L^2(\Omega)}&= E_\Omega(w,T) + 2\langle a\partial^{\alpha}_tu,c^{-2}\partial_tu\rangle-2\langle a\partial^\alpha_tv,c^{-2}\partial_tv \rangle\\
&\leq \|(u(T)-\phi,\partial_tu(T))\|^2_{H^1_0(\Omega)\times L^2(\Omega)} + 2\langle a\partial^{\alpha}_tu,c^{-2}\partial_tu\rangle\\
&\leq \|(u(T),\partial_tu(T))\|^2_{H^1_0(\Omega)\times L^2(\Omega)} + 2\langle a\partial^{\alpha}_tu,c^{-2}\partial_tu\rangle\\
&= E_{\RR^n}(u,T)- E_{\RR^n\setminus\Omega}(u,T)+ 2\langle a\partial^{\alpha}_tu,c^{-2}\partial_tu\rangle\\
&= E_{\RR^n}(u,0)- E_{\RR^n\setminus\Omega}(u,T)\\
&= \|(u_0,-\textstyle\frac{a}{\Gamma(2-\alpha)}u_0))\|^2_{H^1_0(\Omega)\times L^2(\Omega)}- E_{\RR^n\setminus\Omega}(u,T),
\end{aligned}
\]
 where we have used Theorem \ref{thm:energy_frac_wave} (dissipation of energy). Then, by Poincare's inequality we obtain
\begin{equation}\label{ineq:K1}
    \|w(0)\|^2_{H^1_0(\Omega)}\leq \left(1+\left(\textstyle\frac{C_P\|a\|_{C(\RR^n)}}{\Gamma(2-\alpha)}\right)^2\right)\|u_0\|^2_{H^1_0(\Omega)} - E_{\RR^n\setminus\Omega}(u,T),
\end{equation}
for some positive constant $C_P$ depending on $\Omega$.
Defining the error operator 
\[
K := \text{Id} - \Pi_1A_\alpha\Lambda_\alpha,\quad Ku_0 =w(0),
\]
where $\Pi_1$ stands for the projection to the first component, the reconstruction formula is a consequence of $K$ being a contraction in $H^1_0(\Omega)$. Assuming the damping coefficient is small enough, the previous statement will follow from the next observability-type inequality.

\begin{lemma}\label{lemma:obs_ineq}
Under the same hypothesis of the stability theorem and assuming $\Gamma=\partial\Omega$, there is a constant $C_0>0$ such that for all $u_0\in H^1_0(\Omega)$,
\[
\|u_0\|^2_{H^1_0(\Omega)}
\leq C_0 E_{\RR^n\setminus\Omega}(u,T),
\] 
with $u$ solution to \eqref{frac_PAT}. The constant $C_0>0$ is uniform with respect to the damping coefficient $a(x)$ in a small ball in $C_c(\Omega)$ centered at zero.
\end{lemma}
\begin{proof} 
    Let $v$ be the solution to the unattenuated problem \eqref{unatt_PAT} with initial condition $(u_0,0)$ (do not confuse it with the function $v$ above, solution to a time-reversed system). We know the exterior observability inequality holds for $v$ as it was proven in \cite{Palacios2016} thus, by letting $w=v-u$, we get
    \[
    \|u_0\|^2_{H^1_0(\Omega)}\leq C E_{\RR^n\setminus\Omega}(v,T)\leq CE_{\RR^n\setminus\Omega}(u,T) + CE_{\RR^n\setminus\Omega}(w,T).
    \]
    Recalling that $w$ solves the system \eqref{eq:w}, we can recast it in integral form as
    \[
    (w(t),\partial_tw(t)) = \VV(t)[0,-\textstyle\frac{au_0}{\Gamma(2-\alpha)},a\partial^\alpha_tv]
    \]
    with $\VV$ the solution operator for the inhomogeneous wave equation, with null Dirichlet conditions on a sufficiently large domain $\Omega'\supset \Omega$ (see \eqref{def:sol_op_hom}-\eqref{def:sol_op_inhom}) depending on the final time $T>0$. It is clear that the map $a\mapsto (w(t),\partial_tw(t))$ is continuous from $C(\RR^n)$ to $H^1(\RR^n)\times L^2(\RR^n)$, and moreover, by the inequalities proven in Theorem \ref{thm:reg_frac} we have
    \[
    \|(w(T),\partial_tw(T))\|_{H^1(\RR^n)\times L^2(\RR^n)}\leq C\|a\|_{C(\RR^n)}\|u_0\|_{H^1_0(\Omega)}.
    \]
    In consequence, we obtain the inequality
    \[
    \|u_0\|^2_{H^1_0(\Omega)}\leq CE_{\RR^n\setminus\Omega}(u,T) + C\|a\|^2_{C(\RR^n)}\|u_0\|^2_{H^1_0(\Omega)}.
    \]
    By restricting the damping coefficient to a ball in $C_c(\Omega)$ centered at zero and of sufficiently small radius $\delta_0>0$, we can absorb the last term with the left-hand side and conclude that
    \[
    \|u_0\|^2_{H^1_0(\Omega)}\leq C_0E_{\RR^n\setminus\Omega}(u,T),\quad \forall  u_0\in H^1_0(\Omega),
    \]
    for a constant $C_0>0$, uniform with respect to the damping coefficients within the ball.
    \end{proof}

Bringing \eqref{ineq:K1} and the previous lemma, we deduce there is $\delta_0>0$ such that for any damping coefficient $a\in C^\infty_c(\Omega)$ with $\|a\|_{C(\RR^n)}\leq \delta_0$, the next inequality holds 
\[
    \|Ku_0\|^2_{H^1_0(\Omega)}\leq \left(1+\left(\textstyle\frac{C_P\|a\|_{C(\RR^n)}}{\Gamma(2-\alpha)}\right)^2-C_0^{-1}\right)\|u_0\|^2_{H^1_0(\Omega)},\quad \forall u_0\in H^1_0(\Omega),
\]
with $C_0>0$ the constant from Lemma \ref{lemma:obs_ineq}. It is clear then that by restricting the previous to damping coefficients $\|a\|_{C(\RR^n)}<\delta$, such that $C_0^{-1/2}>\frac{C_P\delta}{\Gamma(2-\alpha)}$, we get that $K$ is a strict contraction in $H^1_0(\Omega)$. We have proven the following theorem.
\begin{theorem}\label{thm:reconstruction}
    Let $(\Omega,c^{-2}dx)$ be non-trapping and $\partial\Omega$ strictly convex. Assume there is an admissible foliation $\Sigma(\Omega,\partial\Omega)$, and the observation time $T>0$ is such that $T>\max\{T_0(\Omega,\Sigma),T_1(\Omega)\}$ (see Remark \ref{rmk:stab_comp_data}). Then, there is $\delta>0$ so that for any damping coefficient $a\in C^\infty_c(\Omega)$ with $\|a\|_{C(\RR^n)}<\delta$,
    the error operator $K = \text{\rm Id} - \Pi_1A_\alpha\Lambda_\alpha$  is a strict contraction in $H_0^1(\Omega)$ and we have the following Neumann series reconstruction formula for the fractionally attenuated photoacoustic problem:
    \[
    u_0 = \sum^\infty_{m=0}K^mA_\alpha h,\quad h = \Lambda_\alpha u_0,\quad\forall u_0\in H_0^1(\Omega).
    \]
\end{theorem}

\section*{Acknowledgement}
The work of B. Palacios was partially supported by Agencia Nacional de Investigaci\'on y Desarrollo (ANID), Grant FONDECYT Iniciaci\'on N$^\circ$1251207. The work of S. Acosta was partially supported by NIH award 1R15EB035359-01A1.
S. Acosta would like to thank the support and research-oriented environment provided by Texas Children's Hospital.

\begin{appendix}
    \section{Interpolation and Sobolev Spaces}\label{Appx:Sobolev}
    
    We mainly follow  \cite{LM_v1} (Chapters 1.7, 1.9, 1.11 and 1.15) and \cite{Ta2007}, but other useful references are \cite{Bergh2012, McL2000}. With regard to intermediate derivatives, we follow \cite{LM_v1} Chapter 1.2. \\

{\noindent \em Real interpolation.} Let's consider two separable Hilbert spaces $X\subset Y$ with dense and continuous injection. The $K$-functional associated with the pair $(X,Y)$ is defined as
\[
K(\rho,u;X,Y):= \inf_{u_0+u_1=u}\left(\|u_0\|_X^2+\rho^2\|u_1\|_Y^2\right)^{1/2},\quad u_0\in X,\; u_1\in Y,
\]
for every $u\in X+Y$ and for all $\rho>0$. We define the interpolation space $[X,Y]_{\theta}$ as
\[
[X,Y]_{\theta}= \{u\;|\; u\in Y,\; \rho^{-(\theta+1/2)}K(\rho,u;X,Y)\in L^2(0,\infty)\}
\]
with norm
\begin{equation}\label{def:int_norm}
u\mapsto \left(\|u\|_X^2 + \int^\infty_0\rho^{-(2\theta+1)}K(\rho,u;X,Y)^2d\rho\right)^{1/2}.
\end{equation}

\noindent{\em Sobolev Spaces.} For any $s\in\RR$, the Sobolev spaces $H^s(\RR^n)$ are defined via the Fourier transform as follows:
\[
H^s(\RR^n)=\{u\in\mathcal{S}'(\RR^n): \|u\|^2_{H^s}:=\int (1+|\xi|^2)^s|\hat{u}|^2d\xi<\infty\}
\]
with $\mathcal{S}'(\RR^n)$ the Fr\'echet space of tempered distributions. 

Given a bounded open subset $\Omega\subset\RR^n$ with smooth boundary and $m$ a positive integer, we define
\[
H^m(\Omega) = \{u|_{\Omega}:\in H^m(\RR^n)\}
\]
with the norm
\[
\||u|\|_{H^m(\Omega)} := \inf\{\|U\|_{H^m}:u=U|_{\Omega},\; U\in H^m(\RR^n)\}.
\]
It can be characterized as the space of functions $u\in L^2(\Omega)$ with derivatives $\partial^{\alpha}_xu\in L^2(\Omega)$ for all multi-index $|\alpha|\leq m$, and with equivalent norm
\[
\|u\|_{H^m(\Omega)} := \left(\sum_{|\alpha|\leq m}\|\partial^{\alpha}_xu\|^2_{L^2(\Omega)} \right)^{1/2}.
\]

For any $s\in (0,1)$ we define the fractional order Sobolev spaces $H^s(\Omega)$ as the following interpolation space (for instance, following the K-method). Writing $H^0(\Omega)=L^2(\Omega)$, then
\[
H^s(\Omega) :=[H^1(\Omega),H^0(\Omega)]_{\theta},\quad \theta=1-s, 
\]
with its respective interpolation norms. They can be characterized as 
\[
H^s(\Omega) = \{u|_{\Omega}:u\in H^s(\RR^n)\}
\]
with the equivalent norm
\[
\||u|\|_{H^s(\Omega)} := \inf\{\|U\|_{H^s}:u=U|_{\Omega},\; U\in H^s(\RR^n)\}.
\]
We will also need the spaces
\[
H^{1+s}(\Omega):= [H^2(\Omega),H^0(\Omega)]_{\frac{1-s}{2}}=[H^2(\Omega),H^1(\Omega)]_{1-s},\quad s\in(0,1).
\] 

The subspaces $H^s_0(\Omega)$ are defined as the completions of $C^\infty_c(\Omega)$ under the $H^s(\Omega)$-norms. For $0\leq s\leq 1/2$, $H^s_0(\Omega)= H^s(\Omega)$, while for $1/2<s\leq 1$, $H^s_0(\Omega)$ is strictly contained in $H^s(\Omega)$ and can be characterized as
\[
H^s_0(\Omega) = \{u\in H^s(\Omega):u|_{\partial\Omega}=0\},\quad 1/2<s\leq 1.
\]
It is also possible to obtain them as interpolation spaces, except for the case $s=1/2$. In fact,
\[
[H^1_0(\Omega),L^2(\Omega)]_{1-s}=H^s_0(\Omega) ,\quad\text{for } s\in(0,1),\; s\neq 1/2,
\]
with equivalent norms, 
while for $s=1/2$ one has
\[
[H^1_0(\Omega),L^2(\Omega)]_{1/2} =: H^{1/2}_{00}(\Omega),
\]
which is a strict subset of $H^{1/2}_0(\Omega)$ with continuous inclusion.

For a bounded open set with smooth boundary, $\Omega\subset\RR^n$, and $T>0$, the space $H^1([0,T]\times\partial\Omega)$ is defined as
\[
\{u\in L^2([0,T];H^1(\partial\Omega)):\partial_tu\in L^2([0,T];L^2(\partial\Omega))\},
\]
where the trace spaces $H^s(\partial\Omega)$ for any $s\in\RR$ are defined via finite partitions of unity (due to the compactness of $\partial\Omega$).

\bigskip

\noindent{\em Interpolation spaces of intermediate derivatives.} Set $A = c^2(x)\Delta$ ---with $c(x)$ a smooth function bounded from below by a positive constant--- endowed with null Dirichlet boundary conditions at $\partial\Omega$, thus, its domain is
\[
D(A) := H^2(\Omega)\cap H^1_0(\Omega).
\]
Using its spectral decomposition, we define for $\theta\in(0,1)$ the operators $A^{\theta}f = \sum_{j\geq 0}\lambda_j^\theta(f,\phi_j)_{L^2}$ where $\{\lambda_j\}$ corresponds to the spectrum of $A$, and $\{\phi_j\}$ respective eigenfunctions forming an orthonormal basis of $L^2(\Omega)$. We will use the notation $D(A^0) := L^2(\Omega)$.

\begin{remark}\label{rmk:general_XY_spaces}
    For two separable Hilbert spaces $X\subset Y$, with $X$ dense and continuously injected in $Y$, there is always a self-adjoint and positive operator $A$ in $Y$ with domain $X$, hence $D(A^0)=Y$ and $D(A)=X$.
\end{remark}

Instead of using the K-method, it is possible to alternatively define  interpolation spaces of intermediate derivatives as
\[
[D(A),D(A^0)]_\theta := D(A^{1-\theta}),
\]
with $D(A^{1-\theta})$ corresponding to the domain of $A^{1-\theta}$, this is,
\[
D(A^{1-\theta}) := \{f\in L^2(\Omega):A^{1-\theta} f\in L^2(\Omega)\},
\]
and they are endowed with the norm 
\[
f\mapsto \left(\|f\|^2_{D(A^0)}+\|A^{1-\theta}f\|_{D(A^0)}^2\right)^{1/2},
\]
which turns out to be equivalent to the one defined in \eqref{def:int_norm}. 
In particular, we have that $[D(A),D(A^0)]_{1/2}=D(A^{1/2})$ is topologically equivalent to $H^1_0(\Omega)$. 

For evolution equations one might consider energy spaces of initial conditions given by
\[
\mathcal{H}^1_0(\Omega):= D(A^{1/2})\times D(A^0)\quad\text{and}\quad \mathcal{H}^2_0(\Omega):= D(A)\times D(A^{1/2}).
\]
An application of the Reiteration Theorem (see \cite[Chapter 1, Theorem 6.1]{LM_v1}) gives, for $s\in(0,1)$, that their interpolation spaces are
\[
\begin{aligned}
\mathcal{H}^{1+s}_0(\Omega)&:=[\mathcal{H}^2_0(\Omega),\mathcal{H}^1_0(\Omega)]_{1-s} \\
&= [D(A),D(A^0)]_{(1-s)/2}\times[D(A),D(A^0)]_{1-s/2} \\
&= D(A^{(1+s)/2})\times D(A^{s/2}).
\end{aligned}
\]
Recalling the definition of fractional Sobolev spaces as interpolation spaces, namely
\[
H^{s}(\Omega)=[H^2(\Omega),H^0(\Omega)]_\theta,\quad s=2(1-\theta),
\]
because $D(A)$ is continuously embedded in $H^2(\Omega)$ and $D(A^0) = H^0(\Omega)$, by the interpolation Theorem (see \cite[Chapter 1, Theorem 5.1]{LM_v1}) we find that $\mathcal{H}^{1+s}_0(\Omega)$ is continuously embedded in $H^{1+s}(\Omega)\times H^{s}(\Omega)$.\\

    \section{Integration on Banach spaces}\label{Appx:int_Banach}
    We present a brief introduction to the Bochner integral for vector-valued functions. 
    More details can be found in \cite{DU1977}

    Let $X$ be an abstract Banach space and $(\Omega,\Sigma,\mu)$ a finite measure space. A vector-valued function $F:\Omega\to X$ is said to be {\em (strongly) $\mu$-measurable} if there exists a sequence of {\em simple functions}, this is, functions of the form $f_n(t) = \sum_{i=1}^n\chi_{E_i}(t)u_i$, with $u_i\in X$, $E_i\in\Sigma$ measurable subsets of $\Omega$, and where $\chi_{E_i}(t)$ stands for the characteristic function in $E_i$, such that
    \[
    \lim_{n\to\infty}\|f_n(t)-f(t)\|_X\ = 0\quad \text{for almost every $t\in \Omega$.}
    \]
    A $\mu$-measurable function $f:\Omega\to X$ is said to be {\em Bochner integrable} if there exists a sequence of simple functions $(f_n)$ such that
    \[
    \lim_{n\to\infty}\int_\Omega\|f_n-f\|_Xd\mu = 0,
    \]
    in which case we define the integral of $f$ over some $E\in\Sigma$ as
    \[
    \int_E fd\mu := \lim_{n\to\infty} \int_E f_nd\mu,
    \]
    where for $f_n=\sum_{i=1}^n\chi_{E_i}(t)u_i$, its integral is given by
    \[
    \int_\Omega f_nd\mu = \sum_{i=1}^n\mu(E_i\cap E)u_i.
    \]
    We have the following characterization of integrability.
    \begin{theorem}
    Given a $\mu$-measurable function $f$, we have that $f$ is Bochner integrable if and only if $\int_\Omega\|f\|_Xd\mu<\infty$.
    \end{theorem}

    Some properties of the Bochner integral that will be useful for us are the following.
    \begin{theorem}\label{prop:Bochner}
        If $f$ is $\mu$-Bochner integrable, then
        \begin{itemize}
        \item[(i)] $\lim_{\mu(E)\to 0}\int_Efd\mu = 0;$
        \item[(ii)] $\left\|\int_Efd\mu\right\|_X\leq \int_E \left\|f\right\|_Xd\mu$ for all $E\in\Sigma$;
        \item[(iii)] if $(E_n)$ is a sequence of pairwise disjoint members of $\Sigma$ and $E=\bigcup^\infty_{n=1}E_n$, then
        \[
        \int_{E}fd\mu = \sum^\infty_{n=1}\int_{E_n}fd\mu,
        \]
        where the sum on the right is absolutely convergent.
        \end{itemize}
    \end{theorem}
    We denote by $\mathcal{B}(X)$ the Banach space of bounded linear operators from $X$ to itself. Take $R\in \mathcal{B}(X)$. By definition of integrability, it is clear that for $f:\Omega\to X$ integrable, $Rf$ is integrable as well and moreover
    \[
    R\int_Efd\mu = \int_ERfd\mu.
    \]
    In fact, the same property holds for $R$ a closed linear operator from $X$ to some other Banach space $Y$, provided $Rf$ is integrable. See \cite{DU1977} for a proof of this. A particularly useful example is the closed operator $\partial_x:H^1_0(\Omega)\to L^2(\Omega)$.

    For the sake of this paper, we are interested in $\Omega = [0,T]\subset\RR$ and $\mu$ the Lebesgue measure in $\RR$. Let's consider also a continuous one-parameter family linear operators $R(s):X\to X$ for $s\in[0,T]$ (e.g. the semigroup associated to the wave equation). By continuity of $R(t)$ and the definition of measurability, $R(t)x$ is $\mu$-measurable for any $x\in X$, and if $f:[0,T]\to X$ is measurable then $R(t)f(t)$ is as well. 
    
    To preserve integrability we need, for instance, uniform boundedness, this is, we assume there is $M>0$ such that their operator norms satisfy $\|R(s)\|<M$ for all $s$. Then, $R(t)f(t)$ is also Bochner integrable (in fact, we only need $\|R(t)\|\|f\|_X$ to be integrable in $[0,T]$).

    Let $\UU(t)$ be a uniformly bounded strongly continuous semigroup in $X$ (for instance, take the semigroup associated to the wave equation in a domain $\Omega$ with null Dirichlet condition and $X=H^1_0(\Omega)\times L^2(\Omega)$). For any $\FF(t)\in L^1([0,T];X)$ (i.e. Bochner integrable) the following integral is well defined for every $t\in[0,T]$ and
    \[
    t\mapsto w(t)=\int^t_0\UU(t-s)\FF(s)ds \in C([0,T];X).
    \]

    Let's consider a bounded subdomain $\omega$, strictly contained in $\Omega$ and with smooth boundary $\partial\omega$. We call $R$ the trace operator on $\partial\omega$ which is known to be continuous from $H^1(\RR^n)$ to $H^{1/2}(\partial\omega)$.  
    Denoting by $U_i(t) = \pi_i\UU(t)$, $i=1,2$, the projection to the first and second component respectively, and taking $\FF(t) = [0,F(t)]$ with $F(t)\in L^1([0,T];L^2(\Omega))$ and $\supp(F)\subset\omega$, we have for each $t\in[0,T]$,
    \[
    Rw(t) = \int^t_0RU_1(t-s)\FF(s)ds,
    \]
    and moreover $Rw(t)\in C([0,T];H^{1/2}(\partial\Omega))$ (we also used that the projection operator $\pi_1:X\to H^1_0(\Omega)$ is continuous). Via a Fourier integral representation of solutions to the wave equations, it can be proven that $RU_1(t-s)\FF(s)\in H^1((0,T)\times\partial\omega)$. See appendix \ref{Appx:Sobolev} for a definition of this space.
    
    Let's notice that for any $t\in(0,T)$ and $h$ sufficiently small,
    \[
    \begin{aligned}
        &\frac{1}{h}(Rw(t+h)-Rw(t)) \\
        &= \frac{1}{h}\int^{t+h}_0RU_1(t+h-s)\FF(s)ds-\frac{1}{h}\int^t_0RU_1(t-s)\FF(s)ds\\
        &= \frac{1}{h}\int^{t+h}_tRU_1(t-s)\FF(s)ds+\int^{t}_0R\frac{1}{h}\big(U_1(t+h-s)\FF(s)-U_1(t-s)\FF(s)\big)ds.
    \end{aligned}
    \]
    Taking the limit as $h\to 0$, by \cite[Theorem 9 in \S II.2]{DU1977}, we obtain
    \[
    \begin{aligned}
    \partial_tRw(t) &= RU_1(0)\FF(t) + \int^t_0RU_2(t-s)\FF(s)ds\\
    &= \int^t_0RU_2(t-s)\FF(s)ds,
    \end{aligned}
    \]
    where we used that $\UU(0)=\text{Id}$ and $\supp(F)=\omega_0\subset\omega$.
    
    We next show that for some $\delta>0$ and $M>0$, the map $F\mapsto Rw$ is continuous from 
    \[
    \mathcal{L}^1_{\delta,M}:= L^1((0,T);L^2(\omega_0))\cap L^\infty((\delta,T);L^2(\omega_0))\to H^1((0,T)\times\partial\omega).
    \]
    Instead of using Lebesgue integration in manifolds we follow a Bochner integral approach. Let's introduce
    \[
    g_1(t) = \int^T_0RU_1(t-s)\FF(s)ds,
    \]
    and notice that $s \mapsto RU_1(\cdot-s)\FF(s)\in H^1((0,T)\times\partial\omega)$ is Bochner integrable in $[0,T]$, hence,
    \[
    \begin{aligned}
    \|g_1\|_{H^1((0,T)\times\partial\omega)}&\leq \int^T_0\|RU_1(\cdot-s)\FF(s)\|_{H^1((0,T)\times\partial\omega)}ds\\
    &\leq C\int^T_0\|\FF(s)\|_{H^1_0(\omega)\times L^2(\omega))}ds\\
    &= C\|\FF\|_{L^1((0,T);H^1_0(\omega)
    \times L^2(\omega))}.
    \end{aligned}
    \]
    It remains to estimate
    \[
    g_2(t) = \int^T_{t}RU_1(t-s)\FF(s)ds.
    \]
    since $Rw = g_1-g_2$. Let's first notice that
    \[
    \begin{aligned}
    \partial_tg_2(t) &= -RU_1(0)\FF(t) + \int^T_tRU_2(t-s)\FF(s)ds = \int^T_tRU_2(t-s)\FF(s)ds.
    \end{aligned}
    \]
    Secondly, we see that thanks to the finite propagation speed property there is $\delta>0$ such that $R\UU(t-s)\FF(s)=0$ for $s\in[t,t+\delta]$, thus
    \[
    g_2(t) = \int^T_{t+\delta}RU_1(t-s)\FF(s)ds\quad\text{and}\quad  \partial_tg_2(t)=\int^T_{t+\delta}RU_2(t-s)\FF(s)ds.
    \]
    We have
    \[
    \begin{aligned}
        \|g_2\|^2_{H^1((0,T)\times\partial\omega)} &=  \|g_2\|^2_{L^2((0,T)\times\partial\omega)} + \|\partial_tg_2\|^2_{L^2((0,T)\times\partial\omega)} + \|\partial_xg_2\|^2_{L^2((0,T)\times\partial\omega)},
    \end{aligned}
    \]
    thus, by Theorem \ref{prop:Bochner} and Jensen's inequality
    \[
    \begin{aligned}
        \|g_2\|^2_{H^1((0,T)\times\partial\omega)} &\lesssim \int^T_0\left(\int^T_{t+\delta}\|RU_1(t-s)\FF(s)\|_{L^2(\partial\omega)}ds\right)^{2}dt\\
        &\quad +  \int^T_0\left(\int^T_{t+\delta}\|RU_2(t-s)\FF(s)\|_{L^2(\partial\omega)}ds\right)^{2}dt\\
        &\quad +  \int^T_0\left(\int^T_{t+\delta}\|\partial_xRU_1(t-s)\FF(s)\|_{L^2(\partial\omega)}ds\right)^{2}dt\\
        &\lesssim \int^T_0\int^T_{t+\delta}\|RU_1(t-s)\FF(s)\|^2_{L^2(\partial\omega)}dsdt\\
        &\quad +  \int^T_0\int^T_{t+\delta}\|RU_2(t-s)\FF(s)\|^2_{L^2(\partial\omega)}dsdt\\
        &\quad +  \int^T_0\int^T_{t+\delta}\|\partial_xRU_1(t-s)\FF(s)\|^2_{L^2(\partial\omega)}dsdt\\
    \end{aligned}
    \]
    By Fubini's theorem and then the change of variables $r=s-t$, we get
    \[
    \begin{aligned}
        \|g_2\|^2_{H^1((0,T)\times\partial\omega)}
        &\lesssim \int^T_{\delta}\int^{s-\delta}_{0}\|RU_1(t-s)\FF(s)\|^2_{H^1(\partial\omega)}dtds\\
        &\quad +  \int^T_{\delta}\int^{s-\delta}_{0}\|RU_2(t-s)\FF(s)\|^2_{L^2(\partial\omega)}dtds\\
        &\lesssim \int^T_{\delta}\int^{s}_{\delta}\|RU_1(-r)\FF(s)\|^2_{H^1(\partial\omega)}drds\\
        &\quad +  \int^T_{\delta}\int^{s}_{\delta}\|RU_2(-r)\FF(s)\|^2_{L^2(\partial\omega)}drds\\
        &\lesssim \int^T_{\delta}\int^{T}_{0}\|RU_1(-r)\FF(s)\|^2_{H^1(\partial\omega)}drds\\
        &\quad +  \int^T_{\delta}\int^{T}_{0}\|RU_2(-r)\FF(s)\|^2_{L^2(\partial\omega)}drds.
    \end{aligned}
    \]
    We now use the time-symmetry of the (unattenuated) wave semigroup for initial conditions of the form $[0,F]$, namely
    \[
    U_1(-t)[0,F] = -U_1(t)[0,F]\quad\text{and}\quad U_2(-t)[0,F] = U_2(t)[0,F],
    \]
    which implies that
    \[
    \begin{aligned}
        \|g_2\|^2_{H^1((0,T)\times\partial\omega)}
        &\lesssim \int^T_{\delta}\int^{T}_{0}\|RU_1(r)\FF(s)\|^2_{H^1(\partial\omega)}drds\\
        &\quad +  \int^T_{\delta}\int^{T}_{0}\|RU_2(r)\FF(s)\|^2_{L^2(\partial\omega)}drds\\
        &= \int^T_{\delta}\|RU_1(\cdot)\FF(s)\|^2_{H^1((0,T)\times\partial\omega)}ds\\
        &\lesssim \|\FF\|^2_{L^2((\delta,T);H^1_0(\omega)\times L^2(\omega))}.
    \end{aligned}
    \]
    We have proven the inequality
    \[
    \|Rw\|_{H^1((0,T)\times\partial\omega)}\lesssim \|\FF\|_{L^1((0,T);H^1_0(\omega)\times L^2(\omega))}+ \|\FF\|_{L^2((\delta,T);H^1_0(\omega)\times L^2(\omega))}.
    \]
    In the particular case of $\FF= (0,F)$ with $F\in \mathcal{L}^1_\delta$, then
    \[
    \|g_1\|_{H^1((0,T)\times\partial\omega)}\lesssim  \|F\|_{L^1((0,T); L^2(\omega))}
    \]
    and
    \[
    \|g_2\|^2_{H^1((0,T)\times\partial\omega)}\lesssim C\|F\|_{L^2((\delta,T);L^2(\omega))}\lesssim C\|F\|_{L^\infty((\delta,T);L^2(\omega))}\|F\|_{L^1((0,T);L^2(\omega))},
    \]
    for a constant depending on the parameter $M$. Consequently, for some other constant $C_\delta>0$ depending on $\delta$ and $M$,
    \[
    \begin{aligned}
    \|Rw\|_{H^1((0,T)\times\partial\omega)}  &\leq \|g_1\|_{H^1((0,T)\times\partial\omega)} +\|g_2\|_{H^1((0,T)\times\partial\omega)} \\
    &\lesssim \left(\|F\|^{1/2}_{L^1((0,T);L^2(\omega))}+\|F\|^{1/2}_{L^\infty((\delta,T);L^2(\omega))}\right)\|F\|^{1/2}_{L^1((0,T);L^2(\omega))}
    \\
    &\lesssim \|F\|_{L^1((0,T);L^2(\omega))}+\|F\|_{L^\infty((\delta,T);L^2(\omega))}.
    \end{aligned}
    \]
One can take $\delta=0$ if we apriori know that $F\in L^\infty((0,T);L^2(\omega))$.

\section{Strict convexity of Riemannian hypersurfaces}\label{Appdx:convexity}
\end{appendix}
We give a brief summary of relevant ideas concerning the strict convexity of Riemannian hypersurfaces, and provide a derivation of the convexity condition in terms of cotangent vectors. For more details, we refer the reader to the book \cite{LeeBook}.
    
Let us recall that a Riemannian hypersurface $(M,g)\subset (\widetilde{M},\widetilde{g})$ ---with the latter the ambient manifold of dimension $n$--- is strictly convex with respect to a unit normal vector field $N$ along $M$, if its scalar second fundamental form satisfies $h(X,X)<0$, for all $X\in TM$. 

The scalar second fundamental form $h$ is the covariant symmetric 2-tensor field defined by
\[
h(X,Y) = \langle \Pi(X,Y),N\rangle
\]
for all smooth vector fields $X,Y\in \mathfrak{X}(M)$, where $\Pi(X,Y)$ stands for the second fundamental form of $M$. Equivalently we might write
\[
\Pi(X,Y)=h(X,Y)N.
\]
By homogeneity of $h$, we have $h(X,X) = |X|_g^2h(\hat{X},\hat{X})$ with $\hat{X}=X/|X|_g\in SM$, the sphere bundle of $M$, where $-h(\hat{X},\hat{X})$ reaches a minimum. Consequently, the condition $h(X,X)<0$ is equivalent to the existence of $\kappa>0$ such that $-h(X,X)\geq \kappa|X|_g$.

Denoting by $\widetilde{\nabla}$ the Levi-Civita connection of $\widetilde{M}$, we have that
\[
\Pi(X,Y) = \left(\widetilde{\nabla}_XY\right)^\perp,
\]
where $^\perp$ stands for the normal projection of $T\widetilde{M}|_M$ onto $NM$, the normal bundle of $M$, and on the right hand side $X,Y$ represent any extension of $X,Y\in\mathfrak{X}(M)$ to an open set of $\widetilde{M}$. We recall that the local representation of $\widetilde{\nabla}_XY$ is
\[
\widetilde{\nabla}_XY = (X(Y^k) + X^iY^j\Gamma^k_{ij})E_k
\]
for a given smooth local frame $(E_i)$. The connection coefficients $\Gamma^k_{ij}$ are the Christoffel symbols of the metric $\widetilde{g}$.

Let's assume now that $M$ is orientable. In boundary normal coordinates on $M$, this is, in a collar neighborhood of $M$ (in $\widetilde{M}$) we  write $x=(x',x^n)$ with $x'=(x^1,\dots,x^{n-1})$ local coordinates on $M$ and $x^n$ representing the signed distance to $M$, the ambient metric $\widetilde{g}$ takes the form
\[
g_{\alpha,\beta}dx^\alpha dx^\beta + (dx^n)^2
\]
with $1\leq \alpha,\beta\leq n-1$. We take $N=(0,\dots,0,-1)$ as the exterior unit normal vector field, thus, for any smooth vector field $X$ in $M$ (understood as $X=X^\alpha E_\alpha+0\cdot E_n$ in $T\widetilde{M}$),
\[
\Pi(X,X)=h(X,X)N = X^\alpha X^\beta\Gamma^n_{\alpha\beta}E_n,
\]
where in our coordinates $E_n=-N$. Then, $h(X,X) = -X^\alpha X^\beta\Gamma^n_{\alpha\beta}$.
According to our choice of normal vector field, $M$ is strictly convex if there is $\kappa>0$ such that
\[
X^\alpha X^\beta\Gamma^n_{\alpha\beta} \geq \kappa |X|^2_g.
\]
In boundary normal coordinates the Christoffel symbols satisfy $\Gamma^n_{\alpha,\beta} = -\frac{1}{2}\partial_ng_{\alpha\beta}$, therefore, the previous conditions rewrites as
\[
-\frac{1}{2}(\partial_ng_{\alpha\beta}) X^\alpha X^\beta \geq \kappa |X|^2_g.
\]
In terms of covectors and the inverse of the metric, $\tilde{g}^{-1}=g^{\alpha,\beta}dx^\alpha dx^\beta + (dx^n)^2$, we obtain the equivalent condition
\[
-\frac{1}{2}(\partial_ng^{\alpha\beta}(x)) \xi_\alpha \xi_\beta \geq \kappa |\xi|^2_g,\quad\forall (x,\xi)\in T^*M.
\]
Indeed, let's recall the natural isomorphism between tangent and cotangent bundles, $X^\alpha=g^{\alpha\beta}\xi_\beta$. On one hand, we have that $|X|^2_g=g^{\alpha\beta}\xi_\alpha\xi_\beta=|\xi|^2_g$, and on the other,
\[
\begin{aligned}
    \Pi(X,X)&=\Pi(g^{-1}\xi,g^{-1}\xi)\\
    &=g^{\alpha\gamma'}\xi_{\gamma'}g^{\beta\gamma}\xi_{\gamma}\left(-\frac{1}{2}\partial_n g_{\alpha\beta}\right).
\end{aligned}
\]
By differentiating the relation $g_{\alpha\beta}g^{\beta\gamma} = \delta_{\alpha\gamma}$ we get
\[
(\partial_ng_{\alpha\beta})g^{\beta\gamma} = -g_{\alpha\beta}(\partial_n g^{\beta\gamma}),
\]
from where it follows
\[
\begin{aligned}
    \Pi(X,X)&=-\frac{1}{2}g^{\alpha\gamma'}\xi_{\gamma'}\left(-g_{\alpha\beta}(\partial_n g^{\beta\gamma})\right)\xi_{\gamma}\\
    &=\frac{1}{2}\left(g^{\alpha\gamma'}g_{\alpha\beta}(\partial_n g^{\beta\gamma})\right)\xi_{\gamma}\xi_{\gamma'}\\
    &=\frac{1}{2}(\partial_n g^{\gamma'\gamma})\xi_{\gamma}\xi_{\gamma'}.
\end{aligned}
\]

\section{Symbol classes, FIO's and $\Psi$DO's}\label{Appdx:microlocal}

For any $m\in\RR$, and $n\geq 1$ integer, 
we define the {\em symbols class} $S^m(\RR^n)$ as the set of all smooth functions $a(x,\xi)$ satisfying that for any compact $K\subset \RR^n$, and any multi-indices $\alpha,\beta$, there is a constant $C_{K,\alpha,\beta}>0$ such that
\[
|\partial^\alpha_{x}\partial^\beta_\xi a(x,\xi)|\leq C_{K,\alpha,\beta}(1+|\xi|)^{m-|\beta|},\quad \forall (x,\xi)\in K\times\RR^n.
\]
We say that $a(x,\xi)$ is a {\em symbol of order $m$}. The set $S^m(\RR^n)$ becomes a Frech\'et space when augmented with the seminorms corresponding to the best constants in the inequalities above.

We say that the symbol $a(x,\xi)$ is {\em classical} if it has an asymptotic expansion of the form
\[
a(x,\xi)\sim \sum^\infty_{j=0}a_{m-j}(x,\xi),
\]
where $a_{m-j}$ are smooth and positively homogeneous in $\xi$ of order $m-j$ for $|\xi|>1$ (i.e., $a(x,\lambda\xi)=\lambda^ma(x,\xi)$ for $\lambda>0$), thus elements of $S^{m-j}(\RR^n)$; the sign $\sim$ means that
\[
a(x,\xi)-\sum^M_{j=0}a_{m-j}(x,\xi)\in S^{m-M-1}(\RR^n),\quad \forall M\geq 0.
\]

A real-valued function $\phi(x,\xi)\in C^\infty(\RR^n\times\dot{\RR}^n)$ (with $\dot{\RR}^n:=\RR^n\setminus\{0\}$) is called a phase function if it is positively homogeneous in $\xi$ of order 1, and $(\partial_x\phi,\partial_\xi\phi)\neq 0$ for all $x,\xi$. An operator $A$ defined by
\[
Af(x) = \frac{1}{(2\pi)^n}\int e^{i\phi(x,\xi)}a(x,\xi)\hat{f}(\xi)d\xi,\quad \forall f\in C^\infty_c(\RR^n)
\]
with $a(x,\xi)$ a symbol of order $m$, is said to be a Fourier Integral Operator of order $m$. If the symbol is classical, we call $a_m(x,\xi)$ the principal symbol of the operator. In the particular case of $\phi(x,\xi)=x\cdot\xi$ we say that $A$ (or equivalently $a(x,D)$) is a Pseudodifferential Operator ($\Psi$DO) of order $m$.

By interpreting the previous oscillatory integral in the sense of distributions, an FIO extends to a continuous map from $\mathcal{E}'(\RR^n)$ to $\mathcal{D}'(\RR^n)$ ---the spaces of compactly supported distributions and distributions, respectively. If in addition, the phase function is nondegenerate and the FIO is associated to a local canonical graph, it continuously maps $H^s_{comp}(\RR^n)$ to $H^{s-m}_{loc}(\RR^n)$. This is, for instance, the case of all $\Psi$DO's. For definitions of nondegeneracy and local cannonical graph we refer the reader to \cite{TrevesII}.


\begin{thebibliography}{99}

\small 

\bibitem{AM2016} Acosta, S. \& Montalto, C., 2016. Photoacoustic imaging taking into account thermodynamic attenuation. Inverse Problems, 32(11), 115001.

\bibitem{AP2018} Acosta, S. \& Palacios, B., 2018. Thermoacoustic tomography for an integro-differential wave equation modeling attenuation. Journal of Differential Equations, 264(3), pp.1984-2010.

\bibitem{Ammari2011} Ammari, H., Bretin, E., Garnier, J. \& Wahab, A., 2011. Time reversal in attenuating acoustic media. Contemporary Mathematics, 548, 151-163.

\bibitem{Bergh2012} Bergh, J. \& L\"ofstr\"om, J., 2012. Interpolation spaces: an introduction (Vol. 223). Springer Science \& Business Media.

\bibitem{BDU2007} Bukhgeim, A., Dyatlov, G. \& Uhlmann, G., 2007. Unique continuation for hyperbolic equations with memory. Journal of Inverse and Ill-posed Problems, 15(6), 587-598. 

\bibitem{CO2016} Chervova, O. \& Oksanen, L., 2016. Time reversal method with stabilizing boundary conditions for photoacoustic tomography. Inverse Problems, 32(12), 125004.

\bibitem{DU1977} Diestel, J. \& Uhl, J.J., 1977. Vector Measures, AMS.

\bibitem{Eg1988} Eggermont, P. P. B., 1988. On Galerkin methods for Abel-type integral equations. SIAM Journal on Numerical Analysis, 25(5), 1093-25. doi:https://doi.org/10.1137/0725063 

\bibitem{EHK2020} Eller, M., Hoskins, P. \& Kunyansky, L., 2020. Microlocally accurate solution of the inverse source problem of thermoacoustic tomography. Inverse Problemas, 36(8), 085012.

\bibitem{EK2021} Eller, M. \& Kunyansky, L., 2021.Parametrix for the inverse source problem of thermoacoustic tomography with reduced data. Inverse Problemas, 37(4), 045003.

\bibitem{Finch-Patch-Rakesh-2004} Finch, D., \& Patch, S. K., 2004. Determining a function from its mean values over a family of spheres. SIAM journal on mathematical analysis, 35(5), 1213-1240.

\bibitem{Gelfand_book} Gelfand, I.M. \& Shilov, G.E., 2016. Generalized Functions, Volume 1.

\bibitem{Go1991} Gordon, R., 1991. Riemann integration in Banach spaces. The Rocky Mountain Journal of Mathematics, 21(3), 923-949.

\bibitem{GS1994} Grigis, A. \& Sj\"ostrand, J., 1994. Microlocal analysis for differential operators: an introduction (Vol. 196). Cambridge university press.

\bibitem{HaKN2017} Haltmeier, M., Kowar, R. \& Nguyen, L. V., 2017. Iterative methods for photoacoustic tomography in attenuating acoustic media. Inverse Problems, 33(11), 115009.

\bibitem{HaN2019} Haltmeier, M. \& Nguyen, L.V., 2019 Reconstruction  Algorithms for Photoacoustic Tomography in Heterogeneous Damping Media. J Math Imaging Vis 61, 1007-1021. https://doi.org/10.1007/s10851-019-00879-y


\bibitem{Homan} Homan, A.,  2013. Multi-wave imaging in attenuating media. Inverse Problems and Imaging, 7(4): 1235-1250. doi: 10.3934/ipi.2013.7.1235	

\bibitem{HKN2008} Hristova, Y., Kuchment, P., \& Nguyen, L. 2008. Reconstruction and time reversal in thermoacoustic tomography in acoustically homogeneous and inhomogeneous media. Inverse problems, 24(5), 055006.

\bibitem{Ya2022} Huang, X., Kian, Y., Soccorsi, E. \& Yamamoto, M. (2023). Determination of source or initial values for acoustic equations with a time-fractional attenuation.  Analysis and Applications, Vol. 21, No. 05, pp. 1105-1130. 

\bibitem{KaS2013} Kalimeris, K., \& Scherzer, O., 2013. Photoacoustic imaging in attenuating acoustic media based on strongly causal models. Mathematical Methods in the Applied Sciences, 36(16), 2254-2264.

\bibitem{KaR2021} Kaltenbacher, B. \& Rundell, W., 2021. Some inverse problems for wave equations with fractional derivative attenuation. Inverse Problems, 37(4), 045002.

\bibitem{KT2004} Klibanov, M. V. \& Timonov, A. A., 2004. Carleman estimates for coefficient inverse problems and numerical applications (Vol. 46). Walter de Gruyter.

\bibitem{KS2011} Kowar, R. \& Scherzer, O., 2011. Attenuation models in photoacoustics. In Mathematical modeling in biomedical imaging II: Optical, ultrasound, and opto-acoustic tomographies (pp. 85-130). Berlin, Heidelberg: Springer Berlin Heidelberg.


\bibitem{LeeBook} Lee, J. M., 2006. Riemannian manifolds: an introduction to curvature (Vol. 176). Springer Science \& Business Media.

\bibitem{LernerBook} Lerner, N., 2019. Carleman inequalities: an introduction and more (Vol. 353). Springer.

\bibitem{LiLiu} Li, L. \& Liu, J.G., 2018. A generalized definition of Caputo derivatives and its application to fractional ODEs. SIAM Journal on Mathematical Analysis, 50(3), pp.2867-2900.


\bibitem{LM_v1} Lions, J. L. \& Magenes, E., 2012. Non-homogeneous boundary value problems and applications: Vol. 1 (Vol. 181). Springer Science \& Business Media.

\bibitem{McL2000} McLean, W. C. H., 2000. Strongly elliptic systems and boundary integral equations. Cambridge university press.

\bibitem{Palacios2016}  Palacios, B., 2016. Reconstruction for multi-wave imaging in attenuating media with large damping coefficient, Inverse Problems, 32, 125008, 15 pp.

\bibitem{Palacios2022}  Palacios, B., 2022. Photoacoustic tomography in attenuating media with partial data. Inverse Problems and Imaging, 2022, 16(5): 1085-1111. doi: 10.3934/ipi.2022013

\bibitem{Pod98} Podlubny, I., 1998. Fractional differential equations: an introduction to fractional derivatives, fractional differential equations, to methods of their solution and some of their applications. Elsevier.

\bibitem{SShi2017} Scherzer, O. \& Shi, C., 2017. Reconstruction formulas for photoacoustic imaging in attenuating media. Inverse Problems, 34(1), 015006.

\bibitem{StUh2009} Stefanov, P. \& Uhlmann, G., 2009. Thermoacoustic tomography with variable sound speed. Inverse Problems, 25(7), 075011.

\bibitem{StU2011} Stefanov, P. \& Uhlmann, G. 2011. Thermoacoustic tomography arising in brain imaging. Inverse Problems, 27(4), 045004.

\bibitem{StUh2013} Stefanov, P. \& Uhlmann, G., 2013. Recovery of a source term or a speed with one measurement and applications. Transactions of the American Mathematical Society, 365(11), pp.5737-5758.

\bibitem{StYa2015} Stefanov, P., \& Yang, Y., 2015. Multiwave tomography in closed domain:averaged sharp time reversal. Inverse Problems, 31(6), 065007.


\bibitem{StYa2017a} Stefanov, P., \& Yang, Y., 2017. Thermo-and photoacoustic tomography with variable speed and planar detectors. SIAM Journal on Mathematical Analysis, 49(1), 297-310.

\bibitem{StYa2017b} Stefanov, P., \& Yang, Y., 2017. Multiwave tomography with reflectors: Landweber's iteration. Inverse Problemas and Imaging, 11(2), 373-401.


\bibitem{Ta2007} Tartar, L. , 2007. An introduction to Sobolev spaces and interpolation spaces (Vol. 3). Springer Science \& Business Media.

\bibitem{Ta1994} Tataru, D.,  1995. Unique continuation for solutions to pde's; between H\"ormander's theorem and Holmgren's theorem. Communications in Partial Differential Equations, 20(5-6), 855-884.

\bibitem{Ta1999} Tataru, D., 1999. Unique continuation for operators with partially analytic coefficients. Journal de math\'ematiques pures et appliqu\'ees, 78(5), 505-521.

\bibitem{Taylor_PsiDO_book} 
Taylor, M., Pseudodifferential operators, Princeton mathematical series; vol. 34, Princeton University Press, Princeton, N.J.

\bibitem{TrevesII} Tr\`eves, F., 1980. Introduction to pseudodifferential and Fourier integral operators Volume 2: Fourier integral operators (Vol. 2). Springer Science \& Business Media.



\bibitem{Yosida2012} Yosida, K., 2012. Functional analysis (Vol. 123). Springer Science \& Business Media.


\end{thebibliography}
\end{document}